\documentclass[a4paper,11pt]{amsart}
 \usepackage{amssymb,amsfonts,latexsym}
 \usepackage{graphics,verbatim}
 \usepackage{graphicx}
 \usepackage{comment}
 \usepackage{upref}
 \usepackage{hyperref}
 \usepackage{xcolor}
 \usepackage{amsthm}
 \usepackage{enumitem}
 \usepackage{appendix}
 
 \newtheorem{theorem}{Theorem}[section]
 \newtheorem{proposition}{Proposition}[section]
 \newtheorem{lemma}{Lemma}[section]
 \newtheorem{corollary}{Corollary}[section]
 \newtheorem{remark}{Remark}[section]
 \newtheorem{definition}{Definition}[section]

 \renewcommand{\(}{\left(}
 \renewcommand{\)}{\right)}

 \newcommand{\eps}{\varepsilon}
 \newcommand{\dvg}{\mathrm{~d}V_{\mathbb B^N}}
 
 \newcommand{\grad}{\ensuremath{\nabla}}
 
 \newcommand{\N}{\ensuremath{\mathbb{N}}}
 \newcommand{\R}{\ensuremath{\mathbb{R}}}

 \newcommand{\B}{\mathbb{B}}

 \newcommand{\abs}[1]{\left\vert#1\right\vert}

 \newcommand{\norm}[1]{\left\Vert#1\right\Vert}

 \newcommand{\be} {\begin{equation}}
 \newcommand{\ee} {\end{equation}}
 \newcommand{\bea} {\begin{eqnarray}}
 \newcommand{\eea} {\end{eqnarray}}
 \newcommand{\Bea} {\begin{eqnarray*}}
 	\newcommand{\Eea} {\end{eqnarray*}}

 \newcommand{\ba} {\beta}
 \newcommand{\de} {\delta}

 \newcommand{\Om} {\Omega}
 
 \newcommand{\De} {\Delta}
 \newcommand{\la} {\lambda}

 \newcommand{\nequiv} {\not\equiv}

 \newcommand{\no} {\nonumber}
 \newcommand{\noi} {\noindent}
 \newcommand{\I}{\infty}

 \newcommand{\f}{\frac}

 \newcommand{\ef}{\eqref}
 \newcommand{\nee}{\notag\ee}
  \newcommand{\bn}{\mathbb B^N}
   \newcommand{\Rn}{\mathbb R^N}
  \newcommand{\cn}{c(N,\la)}
   \newcommand{\na}{\nabla}

 \catcode`\@=11
 
 \makeatletter \@addtoreset{equation}{section} \makeatother

\begin{document}

 	\title[Multiplicity of Positive Solutions]
	{Global compactness result and multiplicity of solutions for a class of critical exponent problem in the hyperbolic space}
 	
 \author[Bhakta]{Mousomi Bhakta}
\address{M. Bhakta, Department of Mathematics\\
Indian Institute of Science Education and Research Pune (IISER-Pune)\\
Dr Homi Bhabha Road, Pune-411008, India}
\email{mousomi@iiserpune.ac.in}
\author[Ganguly]{Debdip Ganguly}
\address{D. Ganguly, Department of Mathematics\\
Indian Institute of Technology Delhi\\
IIT Campus, Hauz Khas, New Delhi, Delhi 110016, India}
\email{debdipmath@gmail.com, debdip@maths.iitd.ac.in}	
\author[Gupta]{Diksha Gupta}
\address{D. Gupta, Department of Mathematics\\
Indian Institute of Technology Delhi\\
IIT Campus, Hauz Khas, New Delhi, Delhi 110016, India}
\email{maz208233@maths.iitd.ac.in}
\author[Sahoo]{Alok Kumar Sahoo}
 \address{A.K.Sahoo, Department of Mathematics\\
Indian Institute of Science Education and Research Pune (IISER-Pune)\\
Dr Homi Bhabha Road, Pune-411008, India}
\email{alok.sahoo@acads.iiserpune.ac.in}

 	
 	
 	\date{}
 	
 	\begin{abstract}
 	This paper deals with global compactness and the multiplicity of positive solutions to problems of the type	 
	$$
	 -\De_{\mathbb B^N} u -\lambda u=a(x) |u|^{2^*-2}u+f(x) \quad\text{in } \mathbb B^N, \quad u\in H^1(\mathbb B^N),
	$$
	where $\mathbb B^N$ denotes the ball model of the hyperbolic space of dimension $N\geq 4$, $2^*=\frac{2N}{N-2}$,  $\frac{N(N-2)}{4}<\lambda<\f{(N-1)^2}{4}$ and $f\in H^{-1}(\mathbb B^N)$ ($f\not\equiv 0$) is a non-negative functional in the dual space of $H^1(\bn)$. The potential $a\in L^\I(\bn)$ is assumed to be strictly positive, such that $\lim_{ d(x,0)\to \I}a(x)=1$, where  $d(x,0)$ denotes the geodesic distance. We establish profile decomposition of the functional associated with the above equation. We show that concentration occurs along two different profiles, namely, hyperbolic bubbles and localized Aubin-Talenti bubbles. For $f=0$ and $a\equiv 1$, profile decomposition was studied in \cite{BS}. However, due to the presence of the potential $a(.)$, an extension of profile decomposition to the present set-up is highly nontrivial and requires several delicate estimates and geometric arguments concerning the isometry group (M\"obius group) of the hyperbolic space. Further, using the decomposition result, we derive various energy estimates involving the interacting hyperbolic bubbles and hyperbolic bubbles with localized Aubin-Talenti bubbles. Finally, combining these estimates with topological and variational arguments, we establish a multiplicity of positive solutions in the cases: $a\geq 1$ and $a<1$ separately.  The equation studied in this article can be thought of as a variant of a scalar-field equation with a critical exponent in the hyperbolic space, although such a critical exponent problem in the Euclidean space $\mathbb{R}^N$ has only a trivial solution when $f \equiv 0,$  $a(x)\equiv1$ and $\lambda < 0.$
	
\medskip

	\end{abstract}

 	\maketitle
 	\tableofcontents

 	\section{Introduction}
 	In this article, we
 	consider a nonhomogeneous elliptic equation with critical exponent and potential on the ball model of the hyperbolic space $\mathbb B^N$, namely
 	\begin{equation}
 	\tag{$P_{\lambda,a,f}$}\label{Paf}
 	\left\{\begin{aligned}
 	 	 -\De_{\mathbb B^N} u -\lambda u&=a(x) \abs{u}^{2^*-2}u+f(x) \quad\text{in } \mathbb B^N,\\
 	u>0 \text{ in } \mathbb B^N, &\,\,\quad u\in H^1(\mathbb B^N),
 	\end{aligned}
 	\right.
 	\end{equation}
 	where $N\geq 4$, $2^*:=\f{2N}{N-2}$ and $\frac{N(N-2)}{4}<\lambda<\f{(N-1)^2}{4}$. $H^1(\bn)$ denotes the Sobolev space on the ball model of the hyperbolic space $\bn$, $\Delta_{\bn}$ denotes the
Laplace-Beltrami operator on $\bn$, $\f{(N-1)^2}{4}$  being the bottom of the $L^2$ spectrum of $-\De_{\bn}$. Throughout this paper, we assume that $0\not\equiv f\in H^{-1}(\mathbb B^N)$, the dual space of $H^1(\bn)$, is a non-negative functional, and potential $a$ satisfies the following hypothesis:
	 \begin{equation}
 	\tag{\textbf{A}}\label{assum_A}
 	\left\{\begin{aligned}& 0<a(x)\in C^1(\mathbb B^N), \\
	&\lim_{ d(x,0)\to \I}a(x)=1,\end{aligned}
 	\right.
 	\end{equation} 
  where $d(x,0)$ is the geodesic distance of $x$ from $0$ in the hyperbolic space $\bn$.  The Euclidean unit ball $B^N:=\{x\in\Rn:|x|<1\}$ endowed with the metric 
  $g_{\bn}:= \left(\frac{2}{1-|x|^2}\right)^2\;g_{Eucl}$ represents the ball model for the hyperbolic N-space where $g_{Eucl}$ is the Euclidean metric. The corresponding volume element is given by $\dvg=(\frac{2}{1-|x|^2})^N\mathrm{~d} x$, where $\mathrm{~d} x$ is the Lebesgue measure on $\Rn$.  $\na_{\bn}$ denotes the gradient vector field. Then the Sobolev inequality or rather Poincar\'e-Sobolev inequality obtained by Mancini-Sandeep \cite{MS} takes the following form:
for $N \geq 3$ and $\lambda \leq \frac{(N-1)^2}{4}$ and $1 < p \leq 2^* -1$,  there exists a best constant $\mathcal{S}_{\la,p}:=\mathcal{S}_{\lambda,p}(\bn)>0$ such that 
\begin{align}\label{S-inq}
\mathcal{S}_{\lambda,p}\left(~\int \limits_{\bn}|u|^{p+1} \dvg \right)^{\frac{2}{p+1}}\leq\int \limits_{\bn}\left(|\nabla_{\bn}u|^{2}-\lambda u^{2}\right)\dvg
\end{align}
holds for all  $u\in C_c^{\infty}(\bn)$. 

Before we proceed further, some comments about \eqref{S-inq} are in order:

\begin{itemize}
\item[(i)] Here $|\nabla_{\B^{N}}u|^{2}$ is to be understood concerning the inner product on the tangent space induced by the metric $g_{\B^{N}}$ (we refer to Section \ref{pre} for an appropriate definition).

\medskip

\item[(ii)] By density, \eqref{S-inq} continues  to hold for every $u \in H^1(\bn)$ defined by the closure of $C_c^{\infty}(\B^N)$ with respect to the norm $\|\nabla_{\bn}u\|_{2} + \|u\|_{2},$ where
$\|\cdot\|_{r}$ denotes the $L^r$-norm with respect to the volume measure for $1\leq r\leq\infty.$

\medskip 
\item[(iii)]   $\|u\|_{\lambda}$ defined by 
\begin{align} \label{lambda norm}
\|u\|_{\lambda}^2 : = \int \limits_{\B^{N}}\left(|\nabla_{\B^{n}}u|^{2}-\lambda u^{2}\right) \dvg~\hbox{ for all } u\in H^1(\B^{N}),
\end{align}
is an equivalent norm w.r.t. the standard norm in $H^1(\mathbb{B}^N)$ for $\lambda < \frac{(N-1)^2}{4}.$ 
\end{itemize}

We denote $S_{\lambda}:=S_{\lambda, 2^*-1}(\bn)$ and  $S:=S_{\frac{N(N-2)}{4},2^*-1}(\bn)$. It is well known that $S$ is the best Euclidean Sobolev constant, i.e.,

\begin{equation} 
 S \Big(\int_{ \mathbb R^N}|u|^{2^*}\mathrm{~d} x\Big)^\frac{2}{2^*}\leq \int_{ \mathbb R^N} |\nabla_{\mathbb R^N}u|^2 \mathrm{~d} x,\label{sob_euc}
 \end{equation}
	for all $u\in C_c^\I(\mathbb{R}^N)$, and $S_{\lambda}<S$ for $\frac{N(N-2)}{4}<\lambda < \frac{(N-1)^2}{4}$ (See \cite[Proposition 3.2]{BGKM}).
Moreover, by density argument, \eqref{sob_euc} holds for all $u$ such that $\|\nabla_{\mathbb R^N}u\|_2 < \infty.$ We will now refer to the closure as $D^{1,2}(\mathbb{R}^N).$
 \medskip 
  
Under the stated assumptions \eqref{assum_A}, Eq. \eqref{Paf} can be considered as a perturbation problem of the homogeneous
equation:
  \begin{equation}
 	\tag{$P_{\lambda,1,0}$}\label{E}
 	\left\{\begin{aligned}
 	 	 -\De_{\mathbb B^N} u -\lambda u&=\abs{u}^{2^*-2}u \quad\text{in } \mathbb B^N,\\
 	u>0 \text{ in } \mathbb B^N &\,\,\,\text{ and } u\in H^1(\mathbb B^N).
 	\end{aligned}
 	\right.
 	\end{equation}
In the celebrated paper, \cite{MS} Mancini and Sandeep
proved that \eqref{E} has a unique weak solution $\mathcal{U}:=\mathcal{U}_{N,\la}$ (up to hyperbolic isometries), which is also an extremal of the Poincar\'e-Sobolev inequality \eqref{S-inq}. By elliptic regularity, any solution of \eqref{E} is also $C^\infty$. In \cite{MS}, authors have also shown that any solution of \eqref{E} is radially symmetric, symmetric-decreasing  and satisfies the following decay property:  for every $\varepsilon > 0,$ there exist positive constants $C_1^{\varepsilon}$ and $C_2^{\varepsilon},$ depending on $\varepsilon,$ such that 
  \begin{equation*}
	C_1^{\varepsilon}\;e^{-(\cn + \varepsilon) \, d(x,0)} \leq \mathcal{U}(x) \leq C_2^{\varepsilon}\; e^{-(\cn - \varepsilon) \, d(x,0)}, \ \ \ \mbox{for all} \ \ x \in \bn,
\end{equation*}
where $\cn=\frac{N-1+\sqrt{(N-1)^2-4\la}}{2}$ and $d(x,0)$ is the hyperbolic distance. Later, Bandle and Kabeya in \cite[Lemma~2.3]{CYK} improved the above result and showed that the estimate holds for  $\varepsilon =0,$ i.e.,  there exist constants $\chi_i := \chi_i(N,\la),\,  i = 1,2$  such that 
\begin{align} \label{bubble decay}
\chi_1e^{-\cn d(x,0)} \leq \mathcal{U}(x) \leq \chi_2 e^{- \cn d(x,0)}, \ \ \ \mbox{for all} \ \ x \in \bn.
\end{align}
For simplicity of notations, we will always denote the radially symmetric
solution by $\mathcal{U}$ and the dependence of $N,\la$ will be implicitly assumed. 

Concerning the multiplicity of \eqref{E}, the multiplicity of sign-changing solutions, compactness, and non-degeneracy were studied in \cite{BS, GS1, GS2}. The main question addressed in this paper is whether positive solutions can still survive for the perturbed Eq. \eqref{Paf}. For the subcritical problem, this question has been recently studied by the second and third authors of this paper jointly with Sreenadh \cite{GGS} (also see \cite{GGS2} for the homogeneous counterpart). To the best of our knowledge, so far, there have been no papers in the literature where the existence and multiplicity of positive solutions of equations with critical exponents in $\bn$ have been established in the nonhomogeneous case $f\not \equiv 0$. In a forthcoming article \cite{BGGS1}, we plan to address the question of the existence and multiplicity of solutions for the corresponding homogeneous problem.

 Now let us describe all the necessary assumptions before stating our main theorems. We investigate the solutions of \eqref{Paf} under the following cases separately:
 	 \begin{align}
 	 	&\tag{\textbf{A1}}\qquad\qquad a(x)\geq 1 \quad\forall \,\,\,\,x\in \mathbb B^N. \label{A1}
    \end{align} 
    \begin{align}
 	 	&\tag{\textbf{A2}}\qquad\qquad a(x)\in (0, 1] \quad\forall \,\,x\in \mathbb B^N \text{ and }  \mbox{vol}_{\bn}(\{x: a(x)\neq 1\})>0, \label{A2}
 	 \end{align}
  where $\mbox{vol}_{\bn}(X)$ denotes the volume of $X$ in $\bn$.
   
 	Below we state the main results of this article.  
 	 \begin{theorem}\label{Thm-A1}
 	 	Assume that \eqref{assum_A} and \eqref{A1} are satisfied and $S_{\la}:=S_{\la, 2^*-1}$ is the best constant in the Poincar\'e-Sobolev inequality \eqref{S-inq} corresponding to $p=2^*-1$ and $S$ is the best Sobolev constant in the Euclidean space. 
		 If \begin{equation}
 	 		\norm{f}_{H^{-1}(\mathbb B^N)}\le C_aS_\lambda^\frac{N}{4}, \quad \text{where}\quad C_{a}:=\left( \frac{4}{N+2}\right)\Big[(2^*-1)\norm{a}_{\I}\Big]^\frac{-(N-2)}{4},
 	 	\end{equation}
 	 	then $\eqref{Paf}$ admits  a positive solution.
 	 	
 	 	Furthermore, if $\norm{a}_{\I}< (\f{S}{S_{\lambda}})^\f{N}{N-2}$, then \eqref{Paf} admits at least two positive solutions.
 	 \end{theorem}
   
   \begin{theorem}\label{Thm-A2}
       Let   \eqref{assum_A} and \eqref{A2} be satisfied. In addition, if\begin{equation}
           \label{A_cond}a(x)\geq 1-Cexp(-\delta d(0,x)) \end{equation}
      for some positive constants $C,\delta$ and there exists a   constant $d>0$  such that $\norm{f}_{H^{-1}(\mathbb B^N)}\le d$. Then $\eqref{Paf}$ admits at least three positive solutions.
   \end{theorem} 
  
Adachi-Tanaka \cite{Adachi} considered a subcritical version of \eqref{Paf} in the whole Euclidean space with $\la=-1$, which is considered as a perturbation of the classical scalar field equation and studied the multiplicity results. From the mathematical point of view, it is natural to ask whether that problem admits a positive solution and if yes, then its multiplicity/uniqueness, i.e., whether the positive solutions are stable under the perturbation of type \eqref{Paf} (subcritical exponent), is studied. These questions were quite comprehensively studied by Adachi-Tanaka \cite{Adachi}.  Moreover, variant of the scalar field equations with potential $a(.)$ has been studied thoroughly in the past few decades, starting from the seminal papers of Bahri-Berestycki, Bahri-Li, Bahri-Lions, Berestycki-Lions \cite{BB1, Bahri-Li, Bahri, BL1, BL2}. Subsequently, these topics have been brought to a high level of sophistication concerning the assumptions on the potential $a(.)$ and related multiplicity questions, see e.g., \cite{Ad, AT2, CG1, CG2, CG3, CG4, CW} and references quoted therein, although this list is far from being complete.  In \cite{Adachi}, the existence of four solutions has been obtained under the hypothesis \eqref{A2}. Moreover, in \cite{CZ, J}, the existence of two positive solutions is established when the potential $a(.)$ satisfies \eqref{A1}, and $f\not\equiv 0$ (but small). The papers mentioned above employ topological arguments, like Lusternik-Schnirelman (L-S) category and the min-max arguments, to obtain their multiplicity results. But for such arguments to work, precise energy estimates of solutions to the "limiting problem" are required to avoid the critical level (breaking level) of the Palais-Smale sequences. In all these articles the subcritical problem in the whole Euclidean space $\mathbb{R}^N$ has been considered, and one can easily see using Pohazaev identity that the homogeneous scalar-field equation with critical exponent in the entire $\mathbb{R}^N$ would not have any solutions. Therefore, the problem we considered in the present article has no counterpart in Euclidean space. Therefore the techniques of the proof in the present article are not a straightforward adaptation from Euclidean space. There are some intrinsic features of the hyperbolic space which appear in the proofs.

   Like the Euclidean case, the Poincar\'e-Sobolev inequality is invariant under the action of the conformal group of the ball model $\bn.$ In the Euclidean space, the translations, dilations and inversions constitute the conformal group of $\Rn$, and Sobolev inequality is invariant under their actions. In the case of the hyperbolic space, the conformal group coincides with the isometry group and its elements are given by the \it hyperbolic translations \rm (see Section \ref{pre}).  As a result, both the critical and the subcritical Poincar\'e-Sobolev inequality are invariant under the action of the same conformal group (reminiscent of the translations in the Euclidean setting), and therefore, there is no compactness.
Thus  the variational functional associated with \eqref{Paf} fails to satisfy the Palais-Smale condition, briefly called (PS)
condition. Therefore, the standard
variational technique cannot be applied directly. To overcome this difficulty, the a priori knowledge of the energy range where
the Palais-Smale condition holds is helpful and sometimes suffices to construct critical points. Theorem \ref{ps_decom} studies the profile decomposition of any Palais-Smale sequence possessed by the underlying functional associated with \eqref{Paf}. We show that concentration takes place along two different profiles. For $f=0$ and $a\equiv 1$, i.e., for \eqref{E}, profile decomposition was studied in \cite{BS}. However, an extension of profile decomposition from \cite{BS} to the present set-up
is highly nontrivial (due to the presence of the potential $a(.)$) and requires several delicate estimates and geometric arguments concerning the isometry group (M\"obius group) of the hyperbolic space. The PS decomposition enables us to precisely locate the loss of compactness of the corresponding Palais-Smale sequence at different energy levels.  Our existence results intricately depend on this decomposition. 

Now let us briefly explain the methodology that we have applied to obtain the existence of solutions.  As one anticipates, we follow the topological/variational arguments to obtain multiple solutions. Still, the major hurdle lies in the energy estimates involving solutions to \eqref{E}, which is one of the key ingredients in the proof of Theorem~\ref{Thm-A1} and Theorem~\ref{Thm-A2}. While proving Theorem~\ref{Thm-A1}, we estimate precisely the interactions of different bubbles. Moreover, the interactions are of three types: between hyperbolic bubbles, Aubin-Talenti bubbles (which are quite well-known in the literature) and localized  Aubin-Talenti-hyperbolic bubbles. The latter is a completely new phenomenon, and it is quite delicately handled in this case, although we require only a qualitative interaction between the bubbles.  Also, one can note that the hyperbolic volume grows exponentially, i.e.,
$$\dvg\approx e^{(N-1)r}, \quad r\to\infty,$$
where $r$ denotes the geodesic distance $d(x,0)$, and $\dvg$ denotes the hyperbolic volume form, and to compensate for the exponential volume growth of the hyperbolic space, one requires a new way to tackle the integrals involving interacting bubbles.

We establish a series of new estimates for interacting terms, which are crucial in proving the main theorems. 

 
  To prove Theorem \ref{Thm-A1}, we first do the Palais-Smale decomposition of the functional associated with \eqref{Paf}. Then we decompose $H^1(\bn)$ into
three components: homeomorphic to the interior, boundary and exterior of the unit ball in $H^1(\bn)$, respectively. Thus, using assumptions \eqref{assum_A} and \eqref{A1}, we prove that the energy functional associated with \eqref{Paf}
attains its infimum on one of the components, which serves as our first positive solution. The second positive
solution is obtained via a careful analysis of the (PS) sequences associated with the energy functional, and we
construct a min-max critical level where the PS condition holds. That leads to the existence of a second
positive solution.

We deploy a mix of variational and topological techniques to prove Theorem \ref{Thm-A2}. We begin by estimating the energy of a combination of a solution of \eqref{Paf} and the solution of the limiting problem \eqref{E}. It is established that this energy is below a level that is later proved to be the level below which the PS condition is valid. Then we exploit the convexity of the functional and locate the first solution near zero as a local minimizer. Next, we use the global compactness approach to discover other possible solutions. We determine the first level below which the compactness remains intact or, in other words, the first level below which the PS condition holds by applying the Palais Smale decomposition discussed in Section \ref{PSsection}. It is worth noting that due to the loss of compactness in two profiles, one because of the hyperbolic translations and the other due to the presence of the critical exponent, it was intriguing as well as challenging to examine this first level of breakage of the PS condition. However, we could compare the energies of the Aubin Talenti bubble and the hyperbolic bubble thanks to the well-known inequality $S_{\lambda}<S$ (see \cite{BGKM}) and our assumption \eqref{A2}. Thus we successfully analysed the initial energy level where the lack of compactness happens. This level, along with the L-S category theory, allows us to determine the lower bound on the critical points of the energy functional associated with \eqref{Paf}. The key step here is to prove the category of a level set $\geq 2$ for which we built a map homotopic to the identity map defined on a unit geodesic sphere. Finally, the critical values ensure these solutions are distinct, giving us multiple solutions.

The paper is organized as follows: Section \ref{pre} is the preliminaries, where we introduce various notations, geometric
definitions, and basic results in the hyperbolic space.  In Section \ref{PSsection}, we state and prove the Palais-Smale decomposition theorem as
Theorem \ref{ps_decom}.  Section \ref{Thm1sec} deals with the proof of Theorem \ref{Thm-A1}. Section \ref{KER}
is devoted to the key energy estimates for the functional associated with \eqref{Paf} when the potential $a\leq 1$. In Section \ref{Thm2sec}, we prove Theorem \ref{Thm-A2}, and Section \ref{appen} is the appendix.

\medskip

{\bf Notation:} We denote by $\mathcal{H}$ as $H^1(\mathbb B^N)$ and $\mathcal{H'}:=H^{-1}(\mathbb B^N)$. Throughout this paper, we denote by $\mathcal{U}$ as the unique positive solution of \eqref{E}. $S_{\la}$ denotes the best constant in \eqref{S-inq} corresponding to $p=2^*-1$ and $S$ denotes the best constant in Sobolev inequality in $\Rn$. $\tau_b(.)$ denotes the hyperbolic translation (see \eqref{hyperbolictranslation}) with $\tau_b(0)=b$.  Further, define
$$ \underline{a}:=\inf_{x\in \mathbb{B}^N}a(x)>0\,\,\text{ and }\,\,\bar{a}:=\sup_{x\in \mathbb{B}^N}a(x).$$

    \section{Preliminaries}\label{pre}
   
 We denote the inner product on the tangent space of $\bn$ by $\langle,\rangle_{\bn}$. $\nabla_{\bn}$ and $\Delta_{\bn}$ denote gradient 
 vector field and Laplace-Beltrami operator, respectively. Therefore in terms of local (global) coordinates $\nabla_{\bn}$ and $\Delta_{\bn}$ takes the form
\begin{align*} 
 \nabla_{\bn} = \left(\frac{1 - |x|^2}{2}\right)^2\nabla,  \quad 
 \Delta_{\bn} = \left(\frac{1 - |x|^2}{2}\right)^2 \Delta + (N - 2)\left(\frac{1 - |x|^2}{2}\right)  x \cdot \nabla,
\end{align*}
where $\nabla, \Delta$ are the standard Euclidean gradient vector field and Laplace operator, respectively, and '$\cdot$' denotes the 
standard inner product in $\mathbb{R}^N.$

 \medskip 

\noindent

 {\bf Hyperbolic distance on $\bn.$} The hyperbolic distance between two points $x$ and $y$ in $\bn$ will be denoted by $d(x, y).$ The hyperbolic distance between
$x$ and the origin can be computed explicitly  
\begin{align}\label{hyp-dist}
\rho := \, d(x, 0) = \int_{0}^{|x|} \frac{2}{1 - s^2} \, {\rm d}s \, = \, \log \frac{1 + |x|}{1 - |x|},
\end{align}
and therefore  $|x| = \tanh \frac{\rho}{2}.$ Moreover, the hyperbolic distance between $x, y \in \bn$ is given by 
\begin{align*}
d(x, y) = \cosh^{-1} \left( 1 + \dfrac{2|x - y|^2}{(1 - |x|^2)(1 - |y|^2)} \right).
\end{align*}
As a result, a subset of $\bn$ is a hyperbolic sphere in $\bn$ if and only if it is a Euclidean sphere in $\mathbb{R}^N$ and contained in $\bn$, possibly 
with a different centre and different radius, which can be explicitly computed from the formula of $d(x,y)$ \cite{RAT}.  Geodesic balls in $\bn$ of radius $r$ centred at $x \in \bn$ will be denoted by 
$$
B_r(x) : = \{ y \in \bn : d(x, y) < r \}.
$$

Next, we introduce the concept of hyperbolic translation.

 \medskip 

\noindent

\textbf{Hyperbolic translation.} Given a point $a \in \Rn$ such that $|a|>1$  and $r>0,$ let $S(a,r):=\{x \in \Rn \ | \ |x-a| = r\}$ be the sphere in $\Rn$ with center $a$ and radius $r$ that intersects $S(0,1)$ orthogonally. It is known that it is the case if and only if $r^2 = |a|^2 -1,$ and hence $r$ is determined by $a.$ Let $\rho_{a}$ denotes the reflection with respect to the plane $P_a := \{x \in \Rn \ | \ x\cdot a = 0\}$ and $\sigma_a$ denotes the inversion with respect to the sphere $S(a,r).$ Then $\sigma_a\rho_a$ leaves $\bn$ invariant (see \cite{RAT}).
 
 For $b \in \mathbb{B}^N,$ the hyperbolic translation $\tau_{b}: \mathbb{B}^N \rightarrow \mathbb{B}^N$ that takes $0$ to $b$ is defined by $\tau_b = \sigma_{b^*}\circ\rho_{b^*}$ and can be expressed by the following formula
 
  \begin{align} \label{hyperbolictranslation}
  \tau_{b}(x) := \frac{(1 - |b|^2)x + (|x|^2 + 2 x \cdot b + 1)b}{|b|^2|x|^2 + 2 x\cdot b  + 1},
 \end{align}
where $b^* = \frac{b}{|b|^2}.$
 It turns out that $\tau_{b}$ is an isometry and forms the M\"obius group of $\bn$ (see \cite{RAT}, Theorem 4.4.6 for details and further discussions on isometries). Note that $\tau_{-b} = \sigma_{-b^*}\circ\rho_{-b^*}$ is the hyperbolic translation that takes $b$ to the origin. In other words, the hyperbolic translation that takes $b$ to the origin is the composition of the reflection $\rho_{-b^*}$ and the inversion $\sigma_{-b^*}.$

 For the convenience of the reader, we recall a lemma from \cite{AX14618}(also see \cite{Stoll}). 

\begin{lemma}\cite[Lemma 2.1]{AX14618}\label{lemma1}
  For $b \in \B^N,$ let $\tau_b$ be the hyperbolic translation of $\mathbb{B}^N$ by $b.$ Then for every $u \in C_c^{\infty}(\mathbb{B}^N),$ there holds,

\begin{itemize}
  
\item[(i)] $\Delta_{\bn} (u \circ \tau_b) = (\Delta_{\bn} u) \circ \tau_b$.

  \item[(ii)]  $\langle \nabla_{\bn} (u \circ \tau_b),
   \nabla_{\bn} (u \circ \tau_b)\rangle_{\bn} = \langle(\nabla_{\bn} u) \circ \tau_b, (\nabla_{\bn} u) \circ \tau_b\rangle_{\bn}.$
  
  \item[(iii)] For every open subset $U$ of $\mathbb{B}^N$
  \begin{align*}
  \int_{U} |u \circ \tau_b|^p \dvg = \int_{\tau_b(U)} |u|^p \dvg , \ \mbox{for all} \ 1 \leq p < \infty.
 \end{align*}
 
 \item[(iv)] For every $\phi, \psi \in C_c^{\infty}(\B^N)$,
 \begin{align*}
 \int_{\B^N} \phi(x) (\psi \circ \tau_b)(x) \dvg = \int_{\B^N} (\phi \circ \tau_{-b})(x) \psi(x) \dvg. 
 \end{align*}
\end{itemize}
\end{lemma}
The above formulas hold by density as long as the integrals involved are finite.

 \medskip 

\noindent
   
  \textbf{Conformal change of metric.} Now we briefly discuss some properties of conformal change of metric (for more details, refer to \cite[Section 1.1]{BGKM}). Consider the following problem
  \begin{equation}
  	-\De_{\mathbb B^N} u -\lambda u=a(x) \abs{u}^{2^*-2}u.\label{la_a_H}
  \end{equation} Let us denote the energy function corresponding to the problem \eqref{la_a_H} by
  $$ I_h(u)=\f{1}{2}\Big[\int_{ \mathbb B^N}\left(|\nabla_{\mathbb{B}^{N}}u|^2-\lambda u^2 \right) \mathrm{~d} V_{\bn}\Big]-\f{1}{2^*}\int_{\mathbb B^N}a(x)|u|^{2^*}\mathrm{~d} V_{\bn}(x).$$
  Observe that $u$ solves the problem \eqref{la_a_H} iff $v:=\Big(\frac{2}{1-|x|^2}\Big)^\frac{N-2}{2}u$ \,  solves the Euclidean equation
  \begin{equation}
  	-\De v -\tilde\lambda\Big(\frac{2}{1-|x|^2}\Big)^2 v=a(x) \abs{v}^{2^*-2}v\label{la_a_E}, \quad v\in H^1_0(  B(0,1)),
  \end{equation}
  where $\tilde\lambda=\big(\lambda-\frac{N(N-2)}{4}\big)$, and $B(0,1)$ is the open unit ball in the Euclidean space with the centre at the origin. Let us denote the energy function corresponding to  \eqref{la_a_E} by
  $$ I_e(v)=\f{1}{2}\Big[\int_{ B(0,1)}\left(|\nabla v|^2-\tilde\lambda\Big(\frac{2}{1-|x|^2}\Big)^2 v^2\right)\mathrm{~d} x\Big]-\f{1}{2^*}\int_{ B(0,1)}a(x)|v|^{2^*}\mathrm{~d} x.$$
  Therefore $I_h(u)=I_e(v) $. Then for any $u,w\in  H^1(\mathbb B^N)$, we have  $\langle I'_h(u), w \rangle{_{\bn}} = \langle I'_e(\bar u), \bar w \rangle$, where $\bar u=\big(\frac{2}{1-|x|^2}\big)^\frac{N-2}{2}u$ and $\bar w=\big(\frac{2}{1-|x|^2}\big)^\frac{N-2}{2}w$. 
  

 	\section{Palais-Smale  decomposition} \label{PSsection}
 	In this section, we study the profile decomposition of the Palais-Smale sequences (PS sequences) of the energy functional associated with \eqref{Paf}.  
	Corresponding to  \eqref{Paf}, we define the energy functional $\mathcal I_{\lambda,a,f}:\mathcal{H}\to\R$ by
 	 	\be
 	 	\mathcal I_{\lambda,a,f}(u)=\f{1}{2}\norm{u}^2_\la-\f{1}{2^*}\int_{\mathbb B^N}a(x)|u|^{2^*}\mathrm{~d} V_{\bn}(x)- \int_{\mathbb B^N}f(x)u \mathrm{~d} V_{\bn}(x).
 	 	\ee
 	 A sequence $(u_n)_n$ in $\mathcal{H}$ is said to be a PS sequence of $\mathcal I_{\lambda,a,f}$ at a level $c$  if  $\mathcal I_{\lambda,a,f}(u_n)\to c$  and $\mathcal I^{\prime}_{\lambda,a,f}(u_n)\to 0$ in $\mathcal H'$.
	
	It is worth noticing that if $\mathcal{U}$ is a solution of \eqref{E}, then $\mathcal{U}\circ\tau$ is also a solution of the same equation 
for any $\tau\in I(\mathbb B^N)$, set of all isometries on $\mathbb B^N$. In particular, if we define
		\be\label{seq-un}u_n=\mathcal{U}\circ T_n,\ee
with $T_n\in I(\mathbb B^N)$ such that $T_n(0)\to\infty$, then $u_n$ is a PS sequence converging weakly to zero. Thus in the limiting case, i.e., $a(x)\to 1$ as $d(x,0)\to\infty$  the functional $\mathcal{I}_{\la, a,0}$ exhibits Palais-Smale sequences of the form $u_n$. 

On the other hand, we can exhibit another PS sequence coming from
the concentration phenomenon. Let $W$ be a solution of 
\be\label{Talenti}
-\De W =  \abs{W}^{2^*-2}W, \,\,\,\,\, W\in D^{1,2}(\mathbb{R}^N). 
\ee
The corresponding energy $J(W)$ is given by
$$J(W)=\frac{1}{2}\int_{\mathbb{R}^N}|\nabla W|^2\,\mathrm{~d} x-\frac{1}{2^*}\int_{\mathbb{R}^N}|W|^{2^*}\mathrm{~d} x.$$
 
Choose $y_n\in\mathbb{B}^N$ such that $|y_n| \leq 2- \sqrt{3}<r<R<1$, $y_n\to y_0$ and $\phi \in C_c^{\infty}\left(\mathbb{B}^N\right)$ such that $0 \leq \phi \leq 1,\; \phi(x)=1$ for $|x|<r$ and $\phi(x)=0$ for $|x|>R$. 

Define 
\be\label{seq-vn}
v_{n}(x):=\Big(\f{1-|x|^2}{2}\Big)^\f{N-2}{2}a(y_0)^{\frac{2-N}{4}}\phi(x)  \eps_n^{\f{2-N}{2}}W\(\f{x-y_n}{\eps_n}\),
\ee
  where $\eps_n>0$ and $\eps_n\to0$, then direct calculations show that $v_n$ is also a PS sequence for $\mathcal{I}_{\la, a, 0}$ (see claim I in the proof of Theorem \ref{ps_decom}).

 	\begin{theorem}\label{ps_decom}
 		Let $(u_n)_n\in \mathcal{H}$ be a PS sequence of $\mathcal I_{\lambda,a,f}$ at a level $\ba \geq 0$. Then there exist a critical point $u$ of $\mathcal I_{\lambda, a,f}$,  integers $n_1,n_2\in \N\cup\{0\}$ and functions $u^j_n, v^k_n\in \mathcal{H}$, $y^k\in\bn$ for all $1\leq j\leq n_1$ and $1\leq k\leq n_2$ such that  up to a subsequence (still denoted by the same index)
 		$$u_n=u+\sum_{j=1}^{n_1}u_n^j+\sum_{k=1}^{n_2}v_n^k+o(1),$$
 	where $u_n^j$ and $v_n^k$ are Palais-Smale sequence of the form \eqref{seq-un} and \eqref{seq-vn}, respectively and $o(1)\to 0$ in $\mathcal{H}$. Moreover, 
	$$\ba=\mathcal I_{\lambda,a,f}(u)+\sum_{j=1}^{n_1}\mathcal I_{\lambda,1,0}(\mathcal{U}^j)+\sum_{k=1}^{n_2} a(y^k)^{\frac{2-N}{2}}J(V^k)+o(1),$$
where $\mathcal{U}^j$ and $V^k$ are the solutions of \eqref{E} (without the sign condition) and \eqref{Talenti} corresponding to $u^j_n$ and $v_n^k$. Furthermore, we have
\begin{enumerate}
	\item $T^j_n \circ T^{-i}_n(0)\to \I $ as $n\to \I$ for $i\ne j$.
	\item $\Big| \log\big(\frac{\eps^k_{n} }{\eps^l_{n}}\big)\Big|+ \Big| \frac{y_n^k-y^l_n}{\eps^l_{n}}\Big|\to \I$ as $n\to \I$ for $k\ne l$.
\end{enumerate}
\end{theorem}
Before we start to prove the above theorem, let us recall a definition and a lemma from \cite{BS}, which will be used to prove Theorem \ref{ps_decom}.

\begin{definition}
For $r>0$, define $S_r:=\{x\in\Rn: |x|^2=1+r^2\}$ and for $a\in S_r$, define $$A(a,r):=B(a,r)\cap\bn, $$
where $B(a,r)$ is the open ball in the Euclidean space with center $a$ and radius $r>0$.
Note that for the above choice of $a$ and $r$, $\partial B(a,r)$ is orthogonal to $S^{N-1}$.
\end{definition}

\begin{lemma}\label{l:pd1}\cite[Lemma 3.5]{BS}
Let $r_1,\, r_2>0$ and $A(a_i, r_i)$, $i=1,2$ be as in the above definition, then
there exists $\tau\in I(\bn)$ such that $\tau\big(A(a_1, r_1)\big)=A(a_2, r_2)$.
\end{lemma}

\noindent
{\bf Proof of Theorem \ref{ps_decom}:}
 	\begin{proof} We will prove the theorem in several steps.
  
	\begin{enumerate}[label = \textbf{Step \arabic*:}]
		\item  $\(u_n\)_n$ is a bounded PS sequence in $\mathcal{H}$.  More precisely, as $n\to\infty$
\begin{align*}
\beta+o(1)+o(1)\|u_n\|_\lambda&=\mathcal I_{\lambda,a,f}(u_n)-\frac{1}{2^*}\mathcal I'_{\lambda,a,f}(u_n)[u_n]\\
&=
\bigg(\f{1}{2}-\frac{1}{2^*}\bigg)\norm{u_n}_\lambda^2- \bigg(1-\f{1}{2^*}\bigg) \int_{\mathbb B^N}f(x)u_n(x) \mathrm{~d} V_{\bn}(x) \\
&\geq\f{1}{N}\norm{u_n}_\lambda^2-\frac{N+2}{2N}\|f\|_{H^{-1}(\bn)}\norm{u_n}_\lambda.
 	\end{align*}
  
From the above estimate, it follows that $(u_n)_n$ is a bounded sequence in $\mathcal{H}$, and thus up to a subsequence $u_n\rightharpoonup u $ in $\mathcal{H}$. Therefore, a standard computation yields $\mathcal I'_{\lambda,a,f}(u)=0$ , i.e., $u$ solves the problem  \eqref{Paf}. Set 
$$\bar{u}_n:=u_n-u.$$ Thus $\bar u_n\rightharpoonup 0$ in $\mathcal{H}$.
\medskip

  \item $(\bar{u}_n)_n$ is a PS sequence of $\mathcal I_{\lambda,a,0}$ at the level $\ba-\mathcal I_{\lambda,a,f}(u)$, i.e., 
   $$	\mathcal I_{\lambda,a,0}(\bar{u}_n)\to \ba-\mathcal I_{\lambda,a,f}(u) \qquad\text{and}\qquad	\mathcal I^{\prime}_{\lambda,a,0}(\bar{u}_n)\to 0 \quad\text{in}\quad \mathcal{H'}.$$
Indeed using the Brezis-Lieb lemma, a straightforward computation yields

\begin{align*}
\mathcal I_{\lambda,a,0}(\bar{u}_n)&=\f{1}{2}\norm{u_n-u}_\lambda^2- \f{1}{2^*}\int_{\mathbb B^N}a(x)|u_n-u|^{2^*}\mathrm{~d} V_{\bn} \\
&=\f{1}{2}\|u_n\|^2_{\la}-\f{1}{2}\|u\|^2_{\la}-\f{1}{2^*}\int_{\mathbb B^N}a(x)|u_n|^{2^*}\mathrm{~d} V_{\bn}+\f{1}{2^*}\int_{\mathbb B^N}a(x)|u|^{2^*}\mathrm{~d} V_{\bn}\\
&\qquad+\f{1}{2^*}\int_{\mathbb B^N}a(x)\big(|u_n|^{2^*}-|u|^{2^*}-|u_n-u|^{2^*}\big)\mathrm{~d} V_{\bn}\\
&\qquad+\int_{\mathbb B^N}f(x)u_n\mathrm{~d} V_{\bn}-\int_{\mathbb B^N}f(x)u\mathrm{~d} V_{\bn}+o(1)\\
&=\mathcal I_{\lambda,a,f}(u_n)-\mathcal I_{\lambda,a,f}(u)+o(1).\\
\end{align*}
On the other hand, for any $\psi\in C^\infty_c(\bn)$
\begin{align*}
\mathcal I^{\prime}_{\lambda,a,0}(\bar{u}_n)[\psi]&= \langle \bar u_n,\psi\rangle_{\lambda}-\int_{\bn}a(x)|\bar u_n|^{2^*-2}\bar u_n\psi \dvg\\
&\leq o(1)+\|a\|_{\I}\|\psi\|_{\I}\int_{supp(\psi)}|\bar u_n|^{2^*-1}\dvg=o(1).
\end{align*} 
Hence step 2 follows.

\medskip

\item Since $(\bar{u}_n)_n$ is uniformly bounded in $\mathcal{H}$, by step 2, $\mathcal I'_{\lambda,a,0}(\bar{u}_n)[\bar{u}_n]=o(1)$, i.e., 
 \begin{align*}
 \norm{\bar{u}_n }_\la^2=\int_{\mathbb B^N}a(x)|\bar{u}_n|^{2^*}\mathrm{~d} V_{\bn}(x)+o(1).
 \end{align*}
Suppose $\bar{u}_n\to 0$ in $\mathcal{H}$; then we are done with $n_1=0=n_2$.  If not, there exists $\de>0$ such that
  \begin{align*}
  	\liminf_{n\to\I}\int_{\mathbb B^N}a(x)|\bar{u}_n|^{2^*}\mathrm{~d} V_{\bn}(x)> \de.
  \end{align*}
  We choose  $\delta'>0$ in the following way and fix it, 
$$2\norm{a}_{\I}^\f{(N-2)}{2}\delta'<\de<S_\la^{\f{2^*}{2^*-2}}.$$
Next, we define the concentration function as in \cite[pg. 254]{BS}, namely

   \be
   Q_n(r):=\sup_{x\in S_r}\int_{A(x,r)}a(x)|\bar{u}_n|^{2^*}\mathrm{~d} V_{\bn}(x)\nonumber.
   \ee

Thus
 \be
 \lim_{r\to 0}Q_n(r)=0
\quad \text{and} \quad \lim_{r\to \I}Q_n(r)>\de'.\nonumber\ee
Consequently, we can find a large $r$ say $r_\I$ such that $Q_n(r)>\f{\de'}{2}$ for all $r>r_\I$. 
  \be
  \sup_{x\in S_r}\int_{A(x,r)}a(x)|\bar{u}_n|^{2^*}\mathrm{~d} V_{\bn}(x)= \int_{A(x_n,r_n)}a(x)|\bar{u}_n|^{2^*}\mathrm{~d} V_{\bn}(x)=\delta' \nonumber.
  \ee
Fix  $x_0\in S_{\sqrt{3}}$ and using Lemma~\ref{l:pd1} choose $T_n\in I(\mathbb B^N)$ such that $A(x_n,r_n)=T_nA(x_0,\sqrt{3})$ for all $n$. Define $$v_n(x):=(\bar{u}_n\circ T_n)(x)\quad\text{and}\quad a_n(x):=(a \circ T_n)(x).$$ 
  Therefore,  $(v_n)_n$ is also a bounded sequence, $\norm{v_n}_{\la}=\norm{\bar{u}_n}_{\la}$ and  $v_n\rightharpoonup v$ in $\mathcal{H}$ for some $v\in \mathcal{H}$. Further,
  \begin{align}\label{1c'}
& \mathcal I_{\lambda,a_n,0}(v_n) \no\\
  &= \f{1}{2}\norm{v_n}_\lambda^2- \f{1}{2^*}\int_{\mathbb B^N}a_n(x)|v_n|^{2^*}\mathrm{~d} V_{\bn}(x)\no\\
  &= \f{1}{2}\norm{\bar{u}_n}_\lambda^2- \f{1}{2^*}\int_{\mathbb B^N}a_n(x)|v_n|^{2^*}\mathrm{~d} V_{\bn}(x)\no\\
  &=\ba-\mathcal I_{\lambda,a,f}(u)+o(1)+\f{1}{2^*}\int_{\mathbb B^N}a(x)|\bar{u}_n|^{2^*}\mathrm{~d} V_{\bn}(x)\no\\
  &\quad- \f{1}{2^*}\int_{\mathbb B^N}a_n(x)|v_n|^{2^*}\mathrm{~d} V_{\bn}(x)\no\\
  &= \ba-\mathcal I_{\lambda,a,f}(u)+o(1). 
  \end{align}
We denote $T^{-1}_n$ by $T_{-n}$. Consequently, $T_n\circ T_{-n}= I_d$. 
Let $\Psi\in \mathcal{H}$ be arbitrary and define $\Psi_n:=\Psi\circ T_{-n}$, which implies $\norm{\Psi_n}_\la=\norm{\Psi}_\la$ for all $n$. Therefore, by step 2, we get
 \begin{align}\label{1b'}
o(1)&=\mathcal I^\prime_{\lambda,a,0}(\bar{u}_n)[\Psi_n]= \langle  \bar{u}_n, \Psi_n\rangle_\lambda -  \int_{\mathbb B^N}a(x)|\bar{u}_n|^{2^*-2}\bar{u}_n\Psi_n \mathrm{~d} V_{\bn}(x) \no\\
&=\langle  \bar{u}_n\circ T_n, \Psi_n\circ T_n\rangle_\lambda -  \int_{\mathbb B^N}a(x)|\bar{u}_n|^{2^*-2}\bar{u}_n\Psi_n \mathrm{~d} V_{\bn}(x) \no\\
&=\langle  v_n, \Psi\rangle_\lambda -  \int_{\mathbb B^N}a_n(x)|v_n|^{2^*-2}v_n\Psi \mathrm{~d} V_{\bn}(x)  \no\\
&=\mathcal I^\prime_{\lambda,a_n,0}(v_n)[\Psi].
\end{align}
Moreover, observe 
  \begin{align*}
  &\int_{A(x_0,\sqrt{3})}a_n(x)|v_n|^{2^*}\mathrm{~d} V_{\bn}(x)=\int_{A(x_0,\sqrt{3})}(a\circ T_n)(x)|\bar{u}_n\circ T_n|^{2^*}\mathrm{~d} V_{\bn}(x)\\
  &=\int_{A(x_n,r_n)}a(x)|\bar{u}_n|^{2^*}\mathrm{~d} V_{\bn}(x)=\de'.
   \end{align*}
We now study the following two cases: $v\equiv0$ and $v \not\equiv 0$.
  
\medskip

  \item Firstly, let us assume that $v\equiv0.$
   First, we prove that for any $1>r>2-\sqrt{3}$
  \be
 \int_{\mathbb B^N\cap \{|x|>r\}}a_n(x)|v_n|^{2^*}\mathrm{~d} V_{\bn}(x)\leq \norm{a}_{\I} \int_{\mathbb B^N\cap\{ |x|>r\}}|v_n|^{2^*}\mathrm{~d} V_{\bn}(x)=o(1). \label{1a}
  \ee
 
 Let us fix  a point $b\in S_{\sqrt{3}}$. Also, let $\phi\in C_c^\I(A(b,\sqrt{3}))$ such that $0\leq \phi\leq 1$. 
 We choose  $\Psi\in \mathcal{H}$ as $\Psi=\phi^2v_n$, and thanks to \eqref{1b'}, it holds
\be \langle  v_n,\phi^2v_n\rangle_\lambda =  \int_{\mathbb B^N}a_n(x)|v_n|^{2^*}\phi^2 \mathrm{~d} V_{\bn}(x)+o(1)\no.
\ee
Using local compact embedding, we can simplify the LHS of the above expression as 
\be \langle  v_n,\phi^2v_n\rangle_\lambda =   \norm{\phi v_n}_\la^2+o(1), \no
\ee
(\cite[pg. 255]{BS}).
Therefore, 
\begin{align*} 
\norm{\phi v_n}_\la^2 &=  \int_{\mathbb B^N}a_n(x)|v_n|^{2^*}\phi^2 \mathrm{~d} V_{\bn}(x)+o(1)\\
&= \int_{\mathbb B^N} a_n(x)|v_n|^{2^*-2}(v_n\phi  )^2 \mathrm{~d} V_{\bn}(x)+o(1).
\end{align*}
Now using Poincar\'e-Sobolev inequality on the LHS and H\"older inequality on the RHS, we obtain the following
\begin{align*}
& S_{\la}\(\int_{\mathbb B^N}|\phi v_n|^{2^*}\mathrm{~d} V_{\bn}\)^{\f{2}{2^*}} \\
& \leq \norm{a}^\f{N-2}{N}_{\I}\(\int_{\mathbb B^N}|\phi v_n|^{2^*}\mathrm{~d} V_{\bn}\)^{\f{2}{2^*}}\(\int_{A(b,\sqrt{3})}a_n(x)| v_n|^{2^*}\mathrm{~d} V_{\bn}\)^{\f{2^*-2}{2^*}}.
\end{align*}
 Suppose $\int_{\mathbb B^N}|\phi v_n|^{2^*}\mathrm{~d} V_{\bn}\not\to0$ as $n \rightarrow \I$, then we have 
   \be
   S_{\la} \leq \norm{a}_{\I}^\f{N-2}{N} \(\int_{A(b,\sqrt{3})}a_n(x)| v_n|^{2^*}\mathrm{~d} V_{\bn}\)^{\f{2^*-2}{2^*}}\leq\norm{a}^\f{N-2}{N}_{\I}\(\de'\)^{\f{2^*-2}{2^*}}.
   \nee
  
Combining this with the choice of $\delta'$ yields that
   \be
 \norm{a}^\f{N-2}{N}_{\I}\(\de'\)^{\f{2^*-2}{2^*}}= \(\norm{a}_{\I}^\f{(N-2)}{2}\de'\)^{\f{2^*-2}{2^*}}< S_{\la} \leq  \norm{a}^\f{N-2}{N}_{\I}\(\de'\)^{\f{2^*-2}{2^*}},
   \nee
which is a contradiction. Thus  $\int_{\mathbb B^N}|\phi v_n|^{2^*}\mathrm{~d} V_{\bn}\to0$. Since $b\in S_{\sqrt{3}}$ is arbitrary, \eqref{1a} follows. 

 Further, we fix $2-\sqrt{3}<r<R<1$ and choose $\theta \in C_c^{\infty}\left(\mathbb{B}^N\right)$ such that $0 \leq \theta \leq 1,\,  \theta(x)=1$ for $|x|<r$ and $\theta(x)=0$ for $|x|>R$. Then define
 $$\bar{v}_n:=\theta v_n.$$  
As $v_n\rightharpoonup 0$, using \eqref{1a}, a straight forward computation shows that $\bar{v}_n$ is again a PS sequence of  $\mathcal I_{\lambda,a,0}$. Also, note that arguing in the same way, it also holds $\mathcal I_{\lambda,a_n,0}'(\bar v_n)=o(1)$.  \\
Next, we claim: $v_n= \bar{v}_n+ o(1)$ where $o(1) \rightarrow 0$ in $\mathcal{H}.$ 

\noi
To prove the claim, we first show that $\|v_n\|^2_{\la}=\|\bar v_n\|^2_{\la}+o(1)$.
Indeed, by \eqref{1c'}--\eqref{1a}, it follows
\begin{align}\label{1a'}
\|v_n\|^2_{\la}=\int_{\mathbb{B}^N}a_n(x)|v_n|^{2^*}\dvg+o(1)
&=\int_{\mathbb{B}^N \cap\{|x|\leq r\}}a_n(x)|v_n|^{2^*}\dvg+o(1)\no\\
&=\int_{\mathbb{B}^N \cap\{|x|\leq r\}}a_n(x)|\bar v_n|^{2^*}\dvg+o(1)\no\\
&\leq\int_{\mathbb{B}^N}a_n(x)|\bar v_n|^{2^*}\dvg+o(1)\no\\
&= \|\bar v_n\|^2_{\la}+o(1).
\end{align}
\noi
Conversely,
\begin{align}\label{1f'}
\qquad\|\bar v_n\|^2_{\la}=\int_{\mathbb{B}^N}a_n(x)|\bar v_n|^{2^*}\dvg+o(1)
&=\int_{\mathbb{B}^N \cap\{|x|\leq R\}}a_n(x)\theta^{2^*}|v_n|^{2^*}\dvg+o(1)\no\\
&\leq \int_{\mathbb{B}^N \cap\{|x|\leq R\}}a_n(x)| v_n|^{2^*}\dvg+o(1)\no\\
&=\int_{\mathbb{B}^N}a_n(x)|v_n|^{2^*}\dvg+o(1)\no\\
&= \| v_n\|^2_{\la}+o(1),
\end{align}
where, the second equality from the last  follows from \eqref{1a} as $2-\sqrt{3}<r<R<1$. 
Consequently, \eqref{1a'} and \eqref{1f'} together yield $\|v_n\|^2_{\la}=\|\bar v_n\|^2_{\la}+o(1)$. 
Therefore,
\begin{align*}
 \qquad&\|v_n- \bar{v}_n\|^2_{\la}\\
&=\|v_n\|^2_{\la}+\|\bar v_n\|^2_{\la}-2\langle v_n, \bar v_n\rangle_\la \\
&=2\bigg[\|v_n\|^2_{\la}-\int_{\mathbb{B}^N}a_n(x)|v_n|^{2^*-2}v_n\bar v_n \dvg\bigg]+o(1) \\
&=2\bigg[\|v_n\|^2_{\la}-\int_{\mathbb{B}^N}a_n(x)\theta|v_n|^{2^*}\dvg\bigg]+o(1)\\
&\leq 2\bigg[\|v_n\|^2_{\la}-\int_{\mathbb{B}^N\cap\{|x|\leq r\}}a_n(x)|v_n|^{2^*}\dvg\bigg]+o(1)\\
&=2\bigg[\|v_n\|^2_{\la}-\int_{\mathbb{B}^N}a_n(x)|v_n|^{2^*}\dvg\bigg]+2\int_{\mathbb{B}^N\cap\{|x|> r\}}a_n(x)|v_n|^{2^*}\dvg +o(1)\\
&=2\mathcal{I}'_{\la,a_n,0}(v_n)[v_n]+o(1)=o(1).
\end{align*}
Hence the claim follows. 
   
Now, let us consider a conformal change of the metric from the hyperbolic to the Euclidean metric. Define 
   \be \label{con-ps}
   \tilde v_n:=\Big(\f{2}{1-|x|^2}\Big)^\f{N-2}{2}\bar{v}_n.
   \ee
  Then $\tilde v_n\in H^1_0(B(0,R))$ for some $R>0$, and $\tilde v_n$ is a PS sequence for the problem
\be\label{con-ecl-eq}
-\De u -  c(x) u=  a(x) \abs{u}^{2^*-2}u, \,\,\,\, u\in H^1_0(B(0,R)),  
\ee
where $c(x)=\f{4\la-N(N-2)}{(1-|x|^2)^2},$ which is a smooth bounded function. 

Let $y_{n}\in B(0,R)$. Define $w_{r,n}(x):=r^\f{N-2}{2}\,u(rx+y_{n})$. By a simple calculation, we can see that if $u$ solves the problem \ef{con-ecl-eq}, then $w_{r,n}$ solves the following problem
\be\label{ecl-scal-eq}
-\De w_{r,n} -  r^2c(rx+y_{n}) w_{r,n}=  a(rx+y_{n}) \abs{w_{r,n}}^{2^*-2}w_{r,n}, \;\; u\in H^1_0\bigg(\f{1}{r}(B\big(0,R)-y_{n}\big)\bigg).  
\ee
It is easy to see that as $r\to 0$, the limiting problem associated with  \ef{ecl-scal-eq} is 
\be\label{lim-scal-eq}\tag{$ P_c$}
-\De w =  a(\tilde{y}) \abs{w}^{2^*-2}w,  \,\,\,\, w\in D^{1,2}(\mathbb{R}^N).  
\ee
where $y_n \rightarrow \tilde{y}$.\\
\noi
Using the change of variable $W(x):=a(\tilde{y})^{(\frac{N-2}{4})}w(x)$, it follows that $W$ solves 
\be\label{lim-pro}\tag{$ P_\I$}
-\De W =  \abs{W}^{2^*-2}W.
\ee

Now we will construct a PS sequence of \ef{con-ecl-eq} using the solutions of \ef{lim-scal-eq}. Then we will use this PS sequence with conformal change to construct a PS sequence of $\mathcal{I}_{\la, a, 0}$. 

Let $W$ and $w$ be the positive solutions to the problem \ef{lim-pro} and \ef{lim-scal-eq}, respectively.
Define 
\be\label{wnk}
w_{n}(x):=\eps_n^{\f{2-N}{2}} a(\tilde{y})^{\frac{2-N}{4}}W\(\f{x-y_{n}}{\eps_n}\)\phi(x),
\ee
where $\epsilon_n>0,|y_{n}| \leq 2- \sqrt{3}<r<R,\; \epsilon_n \rightarrow 0,\; y_{n} \rightarrow \tilde{y}$ as $n \rightarrow \infty$,
and $\phi \in C_c^{\infty}\left(\mathbb{B}^N\right)$ such that $0 \leq \phi \leq 1,\; \phi(x)=1$ for $|x|<r$ and $\phi(x)=0$ for $|x|>R$. \\

\textbf{Claim I}: $(w_{n})_n$ is a PS sequence corresponding to the functional associated with \ef{con-ecl-eq}.
To prove the above claim, we denote the energy functional corresponding to \ef{con-ecl-eq} by $I$, defined as follows
\be
I(u)=\f{1}{2}\int_{B(0,R)}\big(|\grad u|^2-c(x)u^2\big)\mathrm{~d} x-\f{1}{2^*}\int_{B(0,R)}a(x)|u|^{2^*}\mathrm{d}x.
\nee

\noi
Thus for $\psi\in H^1_0(B(0,R))$, 
\begin{align}\label{l:new est1}
	 \qquad \quad & I'(w_{n})[\psi]= \int_{B(0,R)}(\grad w_{n}\grad \psi-c(x)w_{n}\psi)\mathrm{~d} x-\int_{B(0,R)}a(x)|w_{n}|^{2^*-2}w_{n}\psi\, \mathrm{~d} x\no\\
	 &=\int_{B(0,R)} (-\psi\De w_{n} -c(x)w_{n}\psi)\mathrm{~d} x-\int_{B(0,R)}a(x)|w_{n}|^{2^*-2}w_{n}\psi\,\mathrm{~d} x\no\\
	 &=-\eps_n^{\f{2-N}{2}} \int_{B(0,R)}\bigg[\eps_n^{-2}\phi\psi(x) \De w\( \f{x-y_{n}}{\eps_n}\)+\psi(x)  w\( \f{x-y_{n}}{\eps_n}\)\De \phi\no\\
  &+2\eps_n^{-1}\psi(x) \grad w\( \f{x-y_{n}}{\eps_n}\)  \grad\phi\bigg]\mathrm{~d} x -\int_{B(0,R)}\bigg[c(x)w_{n}\psi+a(x)|w_{n}|^{2^*-2}w_{n}\psi\bigg]\,\mathrm{~d} x.
\end{align}

Set $x=y\eps_n+y_{n}$, $\Om=B(0,R)$ and $\Om_{n}:=\f{\Om-y_{n}}{\eps_{n}}$. Then we estimate the 4th term on the RHS of \eqref{l:new est1} as
\begin{align*}
	&\int_{B(0,R)}c(x)w_{n}\psi(x) \mathrm{~d} x=\eps^N_n\int_{\Om_{n}}c(y\eps_n+y_n) w_{n}(y\eps_n+y_{n})\psi(y\eps_n+y_{n}) \mathrm{~d} y\\
	&=  \eps_n^{\f{2+N}{2} }\int_{\Om_{n}}c(y\eps_n+y_{n}) \psi(y\eps_n+y_{n})\phi(y\eps_n+y_{n})w(y) \mathrm{~d} y\\
	&= \eps_n^{\f{N-2}{2}}\,\eps^2_n\,\int_{\Om_{n}}c(y\eps_n+y_{n}) \psi(y\eps_n+y_{n})\phi(y\eps_n+y_{n})w(y) \mathrm{~d} y.
\end{align*}
Now using H\"older inequality, we derive (set $z=y\eps_n+y_{n}$)
\begin{align*}
	\qquad&|\int_{B(0,R)}c(x)w_{n}\psi(x) \mathrm{~d} x|\\
 &\leq  \eps_n^{\f{N-2}{2}}\,\eps^2_n\,\int_{\Om_{n}}|c(y\eps_n+y_{n})| |\psi(y\eps_n+y_{n})||\phi(y\eps_n+y_{n})||w(y)|\mathrm{d} y\\
&	\leq \eps_n^{\f{N-2}{2}}\,\eps^2_n\,\(\int_{\Om_{n}}|c(y\eps_n+y_{n})\psi(y\eps_n+y_{n})\phi(y\eps_n+y_{n})|^2\mathrm{d} y\)^\f{1}{2} \(\int_{\Om_{n}}|w(y)|^2\mathrm{d} y\)^\f{1}{2}\\
&=\eps_n\(\int_{\Om}|c(z)\psi(z)\phi(z)|^2\mathrm{d} y\)^\f{1}{2} \(\int_{\Om_{n}}|w(y)|^2\mathrm{d} y\)^\f{1}{2}\\
&\leq C\eps_n \(\int_{\R^N}|w(y)|^2\mathrm{d} y\)^\f{1}{2}=o(1).
\end{align*}
\noi
Now let us estimate the 2nd and 3rd terms of \eqref{l:new est1}
\begin{align*}
\qquad \quad&\eps_n^{\f{2-N}{2}}\int_{B(0,R)}\bigg(|\psi(x)  w\( \f{x-y_{n}}{\eps_n}\)\De \phi|+2\eps_n^{-1}|\psi(x) \grad w\( \f{x-y_{n}}{\eps_n}\)  \grad\phi|\bigg)\mathrm{~d} x\\
\qquad \quad &\leq C\eps_n^{\f{N-2}{2}}\,\eps^2_n\,\int_{\Om_{n}}\bigg(|\psi(y\eps_n+y_{n})  w\( y\)| +2\eps_n^{-1}|\psi(y\eps_n+y_{n}) \grad w\( y\)  \grad\phi(y\eps_n+y_{n})|\bigg)\mathrm{d} y.
\end{align*}
As before, we can see that $C\eps_n^{\f{N-2}{2}}\,\eps^2_n\,\displaystyle\int_{\Om_{n}}|\psi(y\eps_n+y_{n})  w\( y\)\mathrm{d} y|=o(1)$. Also, using similar arguments, we can verify that $$C\eps_n^{\f{N-2}{2}}\,\eps_n\,\displaystyle\int_{\Om_{n}}|\psi(y\eps_n+y_{n}) \grad w\( y\)  \grad\phi(y\eps_n+y_{n})|\mathrm{d} y=o(1).$$
\noi
Next, we estimate the first term and last term in \eqref{l:new est1} as
\begin{align}\label{l:new est2}
	&-\eps_n^{-\f{N+2}{2}}\int_{B(0,R)}\phi(x)\psi(x) \De w\( \f{x-y_{n}}{\eps_n}\)\mathrm{~d} x\no\\
	&=-\eps_n^{\f{N-2}{2}}\int_{ \Om_{n}}\phi(y\eps_n+y_{n})\psi(y\eps_n+y_{n}) \De w\( y\) \mathrm{d} y \no\\
	&=\eps_n^{\f{N-2}{2}}  \int_{ \Om_{n}}a(\tilde{y})\phi(y\eps_n+y_n)\psi(y\eps_n+y_{n})  |w|^{2^*-2}w(y) \mathrm{d} y,
\end{align}
and 
\begin{align}\label{l:new est3}
&	- \int_{B(0,R)}a(x)|w_{n}|^{2^*-2}w_{n}\psi \mathrm{~d} x\no\\
&=-\eps_n^N\int_{\Om_{n}}a(y\eps_n+y_{n})|w_{n}(y\eps_n+y_{n})|^{2^*-2}w_{n}(y\eps_n+y_{n})\psi(y\eps_n+y_{n})\mathrm{~d} y\no\\
&=-\eps_n^N\eps_n^{\f{2-N}{2}(2^*-1)}\int_{\Om_{n}}a(y\eps_n+y_{n})|w(y)|^{2^*-2}w(y)\psi(y\eps_n+y_{n})\phi^{2^*-1}(y\eps_n+y_{n})\mathrm{~d} y\no\\
&=- \eps_n^{\f{N-2}{2}} \int_{\Om_{n}}a(y\eps_n+y_{n})|w(y)|^{2^*-2}w(y)\psi(y\eps_n+y_{n})\phi^{2^*-1}(y\eps_n+y_{n})\mathrm{~d} y.
\end{align}
Therefore, to complete the proof of claim, we only need to show \eqref{l:new est2}+\eqref{l:new est3}  is $o(1)$. Indeed,
\begin{align*}
	\qquad \quad \eps_n^{\f{N-2}{2}} &\left[\int_{ \Om_{n}}a(\tilde{y})\phi(y\eps_n+y_{n})\psi(y\eps_n+y_{n})  |w|^{2^*-2}w\mathrm{~d} y\right.\\
  \qquad \quad  &-\left. \int_{\Om_{n}}a(y\eps_n+y_{n})|w(y)|^{2^*-2}w(y)\psi(y\eps_n+y_{n})\phi^{2^*-1}(y\eps_n+y_{n})\mathrm{~d} y\right]\\
   \qquad \quad \leq \eps_n^{\f{N-2}{2}}&\int_{ \Om_{n}} \psi(y\eps_n+y_{n})|w(y)|^{2^*-2}w(y) \left[a(\tilde{y})-a(y\eps_n+y_{n})\phi^{2^*-1}(y\eps_n+y_{n})\right] \mathrm{~d} y\\
  \qquad \quad  \leq \eps_n^{\f{N-2}{2}}&\int_{ \Om_{n}}\biggl[\psi(y\eps_n+y_{n})|w(y)|^{2^*-2}w(y)\biggr.\\
\qquad \quad &\hspace{2cm}\biggl.\left[a(\tilde{y})- \left[a(y_{n})+ y \eps_n a'(y_{n})+ o(|\eps_n y|) \right]\phi^{2^*-1}(y\eps_n+y_{n})\right]\biggr] \mathrm{d} y\\
   \qquad \quad \leq \eps_n^{\f{N-2}{2}} &\left[\int_{ \Om_{n}} \psi(y\eps_n+y_{n})|w(y)|^{2^*-2}w(y) \left[a(\tilde{y})- a(y_{n})\phi^{2^*-1}(y\eps_n+y_{n})\right]\right.\\
 \qquad \quad   &-\psi(y\eps_n+y_{n})|w(y)|^{2^*-2}w(y) y \eps_n a'(y_{n}) \phi^{2^*-1}(y\eps_n+y_{n})\\
 \qquad \quad   & \left.- \psi(y\eps_n+y_{n})|w(y)|^{2^*-2}w(y) o(|\eps_n y|) \phi^{2^*-1}(y\eps_n+y_{n})\mathrm{~d} y\right]\\
   = o(1),
   \end{align*}
we get the integral to be o(1) by considering $y_{n} \rightarrow \tilde{y}$ and  substituting the well known expression of $w.$  This completes the proof of the claim.  
Now it follows from the PS sequence characterization of \eqref{con-ecl-eq} that 
$$\tilde v_n=\sum_{k=1}^{n_2}w_n^k+o(1),$$
where $w_n^k$ is of the form \eqref{wnk}, where $\eps_n$ is replaced by $\eps_n^k$, $y_n$ is replaced by $y_n^k$ and $y_n^k \rightarrow y^k$.
Therefore,

$$v_n=\bar v_n+o(1)=\Big(\f{2}{1-|x|^2}\Big)^{-\f{N-2}{2}}\tilde v_n+o(1)=\sum_{k=1}^{n_2}v_n^k+o(1),$$
where $v_n^k=\Big(\f{2}{1-|x|^2}\Big)^{-\f{N-2}{2}}w_n^k$.  Hence the theorem follows in this case.

\medskip

\item Now consider the remaining possibility, i.e., $v \not\equiv $ 0.
  
Note that in this case, we have $T_n^{-1}(0) \rightarrow \infty$. To see this, let $\phi \in C_{c}^{\infty}\left(\mathbb{B}^{N}\right)$ and define $\phi_{n}(x) = \phi\circ T_n^{-1}(x)$. Then
\begin{align*}
		0 & \neq \left\langle v,\phi \right\rangle_{\lambda} \\ 
  & =\lim _{n \rightarrow \infty}\left\langle v_{n},\phi \right\rangle_{\lambda} \\ 
&= \lim _{n \rightarrow \infty}\left\langle \bar{u}_n\circ T_n,\phi \right\rangle_{\lambda}\\ 
& = \lim _{n \rightarrow \infty} \int_{\mathbb{B}^{N}} \langle\nabla_{\mathbb{B}^{N}}( \bar{u}_n\circ T_n), \nabla_{\mathbb{B}^{N}} \phi\rangle  \mathrm{~d} V_{\mathbb{B}^{N}} -\lambda\int_{\mathbb{B}^{N}} (\bar{u}_n\circ T_n) \phi \mathrm{~d} V_{\mathbb{B}^{N}}\\ 
& = \lim _{n \rightarrow \infty} \int_{\mathbb{B}^{N}}\langle \nabla_{\mathbb{B}^{N}} \bar{u}_n, \nabla_{\mathbb{B}^{N}} \phi_{n}\rangle \mathrm{~d} V_{\mathbb{B}^{N}} -\lambda\int_{\mathbb{B}^{N}} \bar{u}_n \phi_{n} \mathrm{~d} V_{\mathbb{B}^{N}}\\ 
& = \lim _{n \rightarrow \infty} \int_{\mathbb{B}^{N}} a(x)|\bar{u}_n|^{2^*-2}\bar{u}_n\phi_n \mathrm{~d} V_{\mathbb{B}^{N}}.
\end{align*}

Now if $T_n^{-1}(0) \nrightarrow \infty$ then there exists an $R>0$ such that $supp(\phi_n) \subseteq B_{R}(0)$. Then using Vitali's convergence theorem and the fact that $\bar{u}_n \rightharpoonup 0$ in the last step of the above computation yields a contradiction.

Without any confusion, we assume that $T_n$ is an M\"obius transformation of $\B^N$. Choose $y_n\in\B^N$ such that $T_n(y_n)=0$. In the spirit of \cite[Section 2]{MR2542095}[i.e. $T_{-a}=-\tau_a$], we can consider $T_n=\tau_{y_n}A_{n}$, where $\tau_{y_n}$ is a hyperbolic translation with $\tau_{y}(0)=y$ and $A_n\in O(N).$ We now argue that $T_n(0)\to\I$ as $n\to\I$. We have already seen that $T^{-1}_n(0)\to\I$ as $n\to\I$. Since $T^{-1}_n(0)=y_n$, so $y_n\to\infty$. As a result,  $T_n(0)=\tau_{y_n}A_n(0)=\tau_{y_n}(0)=y_n\to\infty.$
  
 Next, we show that $v$ is a solution of \eqref{E} (without sign condition). Using \eqref{1b'}, for any $\phi \in C_c^{\infty}\left(\mathbb{B}^{N}\right)$, it holds
 $$\left\langle v,\phi \right\rangle_{\lambda}=  \lim _{n \rightarrow \infty} \int_{\mathbb{B}^{N}} a(T_nx)|{v}_n|^{2^*-2}{v}_n\phi \mathrm{~d} V_{\mathbb{B}^{N}}.$$\\
Consider \begin{align*}
&\int_{\mathbb{B}^{N}}a(T_nx)|{v}_n|^{2^*-2}{v}_n\phi\dvg-\int_{\mathbb{B}^{N}}  |{v}|^{2^*-2}{v}\phi\dvg \\ &=\underbrace{\int_{\mathbb{B}^{N}}[a(T_nx)-1]|{v}|^{2^*-2}{v}\phi\dvg}_{J_1}+\underbrace{\int_{\mathbb{B}^{N}}a(T_nx) [|{v}_n|^{2^*-2}{v}_n-|{v}|^{2^*-2}{v}]\phi\dvg}_{J_2}.
\end{align*} 
Both $J_1$ and $J_2$ are of $o(1)$. The former follows from dominated convergence theorem (since, by \eqref{hyperbolictranslation} and the fact that $y_n\to\infty$, it follows that $\tau_{y_n}(x)\to \infty$ as $n\to\infty$, and therefore $a_n(x)\to 1$ for all $x\in\bn$). The latter one is from Vitali's convergence theorem. Hence, $v$ is a solution of \eqref{E}.\\
\noi
Now define,
$$z_n(x):=\bar{u}_n(x)- v(T_n^{-1}x).$$
\textbf{Claim II:} $\left(z_n\right)_n$ is a $(P S)$ sequence for $\mathcal{I}_{\la,a,0}$ at the level $\beta-\mathcal{I}_{\la,a, f}(u)-\mathcal{I}_{\la,1,0}(v)$.
To prove the claim, set
$$
\tilde{z}_n(x):= z_n\left(T_n x\right).
$$
Then
\begin{equation*}
\tilde{z}_n(x)=v_n(x)-v(x) \quad \text { and } \quad\left\|\tilde{z}_n\right\|_\la=\left\|v_n-v\right\|_\la=\left\|z_n\right\|_\la .
\end{equation*}
Therefore, using the above relations, as $n \rightarrow \infty$, we deduce

\begin{equation*}
\begin{aligned}
 \hspace{1cm}&\mathcal{I}_{\la, a, 0}\left(z_n\right)\\
 &=\frac{1}{2}\left\|z_n\right\|_\la^2-\frac{1}{2^{\star}} \int_{\mathbb{B}^N} a(x) \left|z_n\right|^{2^*} \mathrm{~d} V_{\mathbb{B}^{N}}\\
& =\frac{1}{2}\left\|v_{n}-v\right\|_\la^2-\frac{1}{2^{\star}} \int_{\mathbb{B}^N} a\left(T_n x\right) \left|v_{n}-v\right|^{2^*}  \mathrm{~d} V_{\mathbb{B}^{N}} \\
& =\frac{1}{2}\left(\left\|v_{n}\right\|_\la^2-\|v\|_\la^2\right)-\frac{1}{2^{\star}} \int_{\mathbb{B}^N} a\left(T_n x\right) \left|v_{n}(x)\right|^{2^*} \mathrm{~d} V_{\mathbb{B}^{N}}+\frac{1}{2^{\star}} \int_{\mathbb{B}^N} |v|^{2^*} \mathrm{~d} V_{\mathbb{B}^{N}}+o(1) \\
& =\mathcal{I}_{\la, a, 0}\left(\bar{u}_n\right)-\mathcal{I}_{\la,1,0}(v)+o(1) \\
& =\beta-\mathcal{I}_{\la,a, f}({u})-\mathcal{I}_{\la,1,0}(v)+o(1),
\end{aligned}
\end{equation*}
where the third equality follows from the subsequent calculations 
\begin{equation*}
\begin{aligned}
&\int_{\mathbb{B}^N} a\left(T_n x\right) \left|v_{n}(x)\right|^{2^*}\mathrm{~d} V_{\mathbb{B}^{N}}-\int_{\mathbb{B}^N} |v|^{2^*}(x)\mathrm{~d} V_{\mathbb{B}^{N}}\\
& =\int_{\mathbb{B}^N} \left|a^{\frac{1}{2^*}}\left(T_n x\right) v_{n}-v\right|^{2^*}\mathrm{~d} V_{\mathbb{B}^{N}}+o(1) \\
& =\int_{\mathbb{B}^N} a\left(T_n x\right) \left|v_{n}- v\right|^{2^*}d V_{\mathbb{B}^{N}}+o(1).
\end{aligned}
\end{equation*}
In the above expression, the first equality is a straightforward application of the Brezis-Lieb lemma. For the second equality, we perform elementary calculations and appropriately apply the dominated convergence theorem along with the hypothesis that $a(x) \rightarrow 1 \text { as } d(x,0) \rightarrow \infty$.\\
Next, let $\phi \in C_c^{\infty}\left(\mathbb{B}^N\right)$ be arbitrary, then define $\tilde\phi_{n}(x)= \phi(T_n x)$.  Therefore, 
\begin{equation*}
\begin{aligned}
& \mathcal{I}^\prime_{\lambda,a,0}(z_n)[\phi]\\
 &=\langle z_n,\phi \rangle_{\la}-\int_{\mathbb{B}^N}a(x)|z_n|^{2^*-2}z_n\phi\dvg\\
 &=\langle  \bar{u}_n, \phi\rangle_\la -\langle  v, \tilde\phi_n\rangle_\la  -\int_{\mathbb{B}^{N}} a(x)|\bar{u}_n-v\circ T_n^{-1}|^{2^*-2}(\bar{u}_n-v\circ T_n^{-1})\phi \,\mathrm{~d} V_{\mathbb{B}^{N}} \\
 &=\int_{\mathbb{B}^{N}} a(x)|\bar{u}_n|^{2^*-2}\bar{u}_n\phi\,\mathrm{~d} V_{\mathbb{B}^{N}}+o(1)-\int_{\mathbb{B}^{N}}|v|^{2^*-2}v\tilde\phi_n \,\mathrm{~d} V_{\mathbb{B}^{N}}+o(1)\\
 &\quad-\int_{\mathbb{B}^{N}} a(x)|\bar{u}_n-v\circ T_n^{-1}|^{2^*-2}(\bar{u}_n-v\circ T_n^{-1})\phi \,\mathrm{~d} V_{\mathbb{B}^{N}} \\
 &=\int_{\mathbb{B}^{N}}a_n(x)|v_n|^{2^*-2}v_n\tilde\phi_n\,\mathrm{~d} V_{\mathbb{B}^{N}}-\int_{\mathbb{B}^{N}}|v|^{2^*-2}v\tilde\phi_n \,\mathrm{~d} V_{\mathbb{B}^{N}}+o(1)\\
&\quad-\int_{\mathbb{B}^{N}} a_n(x)|v_n-v|^{2^*-2}(v_n-v)\tilde\phi_n \,\mathrm{~d} V_{\mathbb{B}^{N}} \\
&=  \underbrace{\int_{\mathbb{B}^{N}} \left(a_n-1\right)|v|^{2^*-2}v\tilde\phi_n \mathrm{~d} V_{\mathbb{B}^{N}}}_{\text{$I_{n}^1$}}\\
 &\quad+  \underbrace{\int_{\mathbb{B}^{N}} a_n\left[|v_n|^{2^*-2}v_n- |v|^{2^*-2}v- |v_n- v|^{2^*-2}(v_n-v) \right]\tilde\phi_n \mathrm{~d} V_{\mathbb{B}^{N}}}_{\text{$I_{n}^2$}}+o(1).
\end{aligned}
\end{equation*}

 \noi 
 Firstly, consider $I_{n}^1$, i.e.,
 \begin{equation*}
 \begin{aligned}
 \hspace{0.5cm}&\int_{\mathbb{B}^{N}}\left(a(T_n x)-1\right)|v|^{2^*-2}v\tilde\phi_n \mathrm{~d} V_{\mathbb{B}^{N}}\\
 &=\int_{B_{R}(0)} \left(a(T_n x)-1\right)|v|^{2^*-2}v\tilde\phi_n\, \mathrm{~d} V_{\mathbb{B}^{N}}+\int_{B_{R}(0)^c} \left(a(T_n x)-1\right)|v|^{2^*-2}v\tilde\phi_n \,\mathrm{~d} V_{\mathbb{B}^{N}}\\
 &\leq C\int_{B_{R}(0)}|v|^{2^*-2}v\tilde\phi_n\, \mathrm{~d} V_{\mathbb{B}^{N}}+\eps\int_{B_{R}(0)^c}|v|^{2^*-2}v\tilde\phi_n\, \mathrm{~d} V_{\mathbb{B}^{N}},
 \end{aligned}
 \end{equation*}
 where for the last inequality, we have chosen $R>0$ large enough so that in $B_{R}(0)^c$, $|a(T_n x)-1|<\eps$. Since the definition of $\tilde\phi_n$ implies $\tilde\phi_n\rightharpoonup 0$, which in turn implies that $\tilde\phi_n\to 0$ a.e., using Vitali's convergence theorem via H\"older inequality, we conclude that 
 $\int_{B_{R}(0)}|v|^{2^*-2}v\tilde\phi_n\, \mathrm{~d} V_{\mathbb{B}^{N}}=o(1)$. On the other hand, since $\|\tilde\phi_n\|_{\la}=\|\phi\|_{\la}$, using H\"older inequality followed by Sobolev inequality, it follows that 
 $\int_{B_{R}(0)^c}|v|^{2^*-2}v\tilde\phi_n\, \mathrm{~d} V_{\mathbb{B}^{N}}$ is bounded. Hence $I_n^1=o(1)$.  

Next, as $\tilde\phi_{n}\in C_c^{\infty}\left(\mathbb{B}^N\right)$, the standard arguments using Vitali's convergence theorem yield $I_{n}^2=o(1)$. This proves the claim.

So far, we have established the following\:
$\left(z_n\right)_n$ is a $(P S)$ sequence for $\mathcal{I}_{\la,a,0}$ at the level $\beta-\mathcal{I}_{\la,a, f}({u})-\mathcal{I}_{\la,1,0}(v)$, i.e.,
\be
\beta=\mathcal{I}_{\la,a, f}({u})+\mathcal{I}_{\la,1,0}(v)+ \mathcal{I}_{\la,a,0}(z_n)+o(1)\no,
\ee
and 
\be
 z_n(x)=\bar{u}_n(x) - v(T_n^{-1}x)=u_n(x)-u(x)-v(T_n^{-1}x),\no
\ee
where $v$ solves the problem \eqref{E} (without sign condition), and $\mathcal{I}'_{\lambda,a,f}(u)=0$, i.e., $u$ solves the problem  \eqref{Paf}.

Given the above claim, if $z_n$ does not converge to zero in $\|.\|_{\la}$, we can repeat the above procedure for the PS sequence $z_n$ to land in Step 4 or Step 5. In the first case, we are through; in the second case, we will either end up with a converging PS sequence or repeat the process. But this process has to stop in finitely many stages as any PS sequence has nonnegative energy, and in each stage, we are reducing the energy by a fixed
positive constant. This completes the decomposition part of the proof of the theorem.  The Proof of $(1)$ and $(2)$ are standard; we refer to the appendix. 
\end{enumerate}
 	\end{proof}

 \section{Multiplicity of Positive solutions when $a(x) \geq 1$} \label{Thm1sec}

 In this section, we will prove the existence of at least two positive solutions of \eqref{Paf}. We will establish two positive critical points of the energy functional:
 $ \bar{\mathcal I}_{\lambda,a,f}:\mathcal{H}\to\R$  where
 \begin{equation}\label{4.1}
 \bar{\mathcal I}_{\lambda,a,f}(u)=\f{1}{2}\norm{u}^2_\la-\f{1}{2^*}\int_{\mathbb B^N}a(x)u_+^{2^*}\mathrm{~d} V_{\bn}(x)- \int_{\mathbb B^N}f(x)u\, \mathrm{~d} V_{\bn}(x).
 \end{equation}
Clearly, if $u$ is a critical point of $ \bar{\mathcal I}_{\lambda,a,f}$, then $u$ is the solution to the following problem
\begin{equation} 
	\begin{aligned}
		-\Delta_{\mathbb{B}^{N}} u- \lambda u &=a(x)u_{+}^{2^*}+f(x) \text { in } \mathbb{B}^{N}, \\
		u & \in \mathcal{H}. 
	\end{aligned}
	\label{4.pp}
\end{equation}

\begin{remark}
	If we use $v=u_{-}$ as a test function in \eqref{4.pp}, where $u$ denotes a weak solution, and $f$ is a nonnegative functional, implying that $u_{-}=0$, i.e., $u \geq 0$. Thus $u>0$ results from the maximum principle, and we get $u$ to be a solution to \eqref{Paf}.
	\label{rmk1.1}
\end{remark}
\noi
Let $g:\mathcal{H}\to\R$ be a functional defined by
  \be\label{g-def}
 g(u)=\norm{u}^2_\la-( 2^*-1)\norm{a}_{\I} \norm{u}^{2^*}_{2^*}.
 \ee
 First, we partition $\mathcal{H}$ into three disjoint sets, namely
 \begin{align*}
 	&U_1:=\Big\{u\in \mathcal{H}\,:\, u=0 \,\,\text{or}\,\, g(u)>0\Big\},\,\,\,
 	U_2:=\Big\{u\in \mathcal{H}\,:\,  \, g(u)<0\Big\},\\
 	&U_0:=\Big\{u\in \mathcal{H}\backslash \{0\}\,:\,   g(u)=0\Big\}.
 \end{align*}
\noi
Now define \begin{equation}\label{ci_def}
	c_i:=\inf_{U_i}\bar{\mathcal I}_{\lambda,a,f}(u), \,\,\,\,\text{for }\, i=0,1,2.
\end{equation}

\begin{lemma}\label{U0_Lbnd}
	Let $C_a$ be as defined in Theorem \ref{Thm-A1}. Then 
	\begin{equation}
		\frac{4}{N+2}\norm{u}_\la\geq C_{a}S_{\lambda}^\frac{N}{4}>0\quad \forall\, u\in U_0.
	\end{equation}
 
\end{lemma}
\begin{proof}
	Using the definition of $U_0$
	\begin{align*}
&0=g(u)=\norm{u}^2_\la-( 2^*-1)\norm{a}_{\I} \norm{u}^{2^*}_{2^*}\implies \norm{u}^\frac{2}{2^*}_\la=\big[( 2^*-1)\norm{a}_{\I}\big]^\frac{1}{2^*} \norm{u}_{2^*}\, \forall\,u\in U_0.
	\end{align*}
From definition of $S_\lambda$, we have
\begin{equation}
  \norm{u}_\la\geq S^\frac{1}{2}_\lambda\norm{u}_{2^*}=S^\frac{1}{2}_\lambda\frac{\norm{u}^\frac{2}{2^*}_\la}{\big[( 2^*-1)\norm{a}_{\I}\big]^\frac{1}{2^*}}\,\,\,\forall\,u\in U_0.
\end{equation}
Hence the lemma follows.
\end{proof}
\begin{remark}\label{rem_2*_bnd}
As a consequence of Lemma \ref{U0_Lbnd}, $\|u\|_{\la}$ and $\norm{u}_{2^*}$ are bounded away from zero for all $u\in U_0$.
 
\end{remark}

\begin{remark}\label{rem_U_012}
	 For any $t>0$ and $u\in \mathcal{H}$  $$g(tu)=t^2\norm{u}^2_\la-t^{2^*}( 2^*-1)\norm{a}_{\I} \norm{u}^{2^*}_{2^*}.$$ 
  Clearly, $g(0)=0$, and $t\mapsto g(tu)$ is a strictly concave function.  Thus for any $u\in \mathcal{H} $ with $\norm{u}_{\lambda}=1$, there exists a unique $t=t(u)$ such that $tu\in U_0$.
	Moreover, for any $v\in U_0$, $g(tv)=(t^2-t^{2^*})\norm{v}^2_\la$. As an immediate consequence, 
	 \begin{equation*}
	 	tv\in U_1 \text{ for all }t\in (0,1) \quad\text{ and } \quad	tv\in U_2 \text{ for all }t\in (1,\I).
	 \end{equation*}
\end{remark}
\begin{theorem}\label{comp_c1c0}
	Assume $C_a$ is defined as in Theorem \ref{Thm-A1}, and $c_0$ and $c_1$ are defined as in \eqref{ci_def}. Furthermore, if 
	\begin{equation}\label{Caf_comp}
		\inf_{u\in \mathcal{H}, \norm{u}_{2^*}=1}\Big\{ C_a\norm{u}^{\frac{N+2}{2}}_\la-\langle f,u\rangle\Big\}>0,
	\end{equation}
then $c_1<c_0$.
\end{theorem}
\begin{proof}
	Define a functional $J: \mathcal{H}\to \mathbb R$ as
	\begin{equation*}
		J(u)=\frac{1}{2} \norm{u}_\lambda^2-\frac{\norm{a}_{\I}}{2^*}\norm{u}^{2^*}_{2^*}-\langle f,u\rangle.
	\end{equation*}
\begin{enumerate}[label = \textbf{Step \arabic*:}]
\item In this first step, we prove that there exists an $\alpha>0$ such that 
\begin{equation*}
	\frac{d}{dt}J(tu)\big|_{t=1}\geq \alpha\quad \text{for all  }\,\,u\in U_0.
\end{equation*}
Indeed, using the definition of $U_0$ and the value of $C_a$, we have for $u\in U_0$
\begin{align}\label{l:new est5}
	\frac{d}{dt}J(tu)\big|_{t=1}=  \norm{u}_\lambda^2-\norm{a}_{\I}\norm{u}^{2^*}_{2^*}-\langle f,u\rangle
	&=\Big( \frac{4}{N+2}\Big)\norm{u}_\lambda^2-\langle f,u\rangle\\ \notag
		&=C_a\frac{\norm{u}^\frac{N+2}{2}_\lambda}{\norm{u}^{\frac{N}{2}}_{2^*}} -\langle f,u\rangle.
\end{align}
Moreover, from the given hypothesis \eqref{Caf_comp}, there exists a $\beta>0$ such that
\begin{equation}\label{inf-prop}
\inf_{u\in \mathcal{H}, \norm{u}_{2^*}=1}\Big\{ C_a\norm{u}^{\frac{N+2}{2}}_\la-\langle f,u\rangle\Big\}\geq\beta.
\end{equation} 
Now, \begin{align*}
\eqref{inf-prop}&\iff C_a\frac{\norm{u}^\frac{N+2}{2}_\lambda}{\norm{u}^{\frac{N}{2}}_{2^*}}-\langle f,u\rangle\,\geq\beta, \quad \norm{u}_{2^*}=1\\
&\iff  C_a\frac{\norm{u}^\frac{N+2}{2}_\lambda}{\norm{u}^{\frac{N}{2}}_{2^*}}-\langle f,u\rangle\,\geq\beta \norm{u}_{2^*}, \quad  u\in \mathcal{H}\backslash\{0\}.
\end{align*}
Hence plugging back the above estimate into \eqref{l:new est5} and using Remark~\ref{rem_2*_bnd}, we complete the proof of Step 1.
\begin{align*}
	\frac{d}{dt}J(tu)\big|_{t=1}& =C_a\frac{\norm{u}^\frac{N+2}{2}_\lambda}{\norm{u}^{\frac{N}{2}}_{2^*}} -\langle f,u\rangle\geq\beta \norm{u}_{2^*}, \quad  u\in \mathcal{H}\backslash\{0\}.
\end{align*}

 \item Let $(u_n)_n$ be a minimizing sequence for $\bar{\mathcal I}_{\lambda,a,f}$ on $U_0$, i.e., $\bar{\mathcal I}_{\lambda,a,f}(u_n)\to c_0$ and $\norm{u_n}^2_\la=( 2^*-1)\norm{a}_{\I} \norm{u_n}^{2^*}_{2^*}$. Therefore, as $n\to \infty$
  \begin{equation*}
  	c_0+o(1)=\bar{\mathcal I}_{\lambda,a,f}(u_n)\geq J(u_n)\geq \Big( \frac{1}{2}-\frac{1}{2^*(2^*-1)}\Big)\norm{u_n}_\lambda^2-\norm{f}_{\mathcal{H}'}\norm{u_n}_\lambda.
  \end{equation*}
This implies  $J(u_n)$ is a bounded sequence, and  $\norm{u_n}_\lambda$, $\norm{u_n}_{2^*}$ are bounded too.\\
\textbf{Claim: $c_1<0$.}\\
Indeed, to prove the above claim, it is enough to show that $\bar{\mathcal I}_{\lambda, a,f}(v)<0$ for some $v\in U_1$. Thanks to Remark \ref{rem_U_012}, we can choose $u\in U_0$ such that $\langle f,u\rangle>0$. Therefore,
\begin{align*}
	\bar{\mathcal I}_{\lambda,a,f}(tu)&=\f{t^2}{2}\norm{u}^2_\la-\f{t^{2^*}}{2^*}\int_{\mathbb B^N}a(x)u_+^{2^*}\mathrm{~d} V_{\bn}(x)- t\langle f,u \rangle\\
	&\leq\f{t^2}{2}( 2^*-1)\norm{a}_{\I} \norm{u}^{2^*}_{2^*}-\f{t^{2^*}}{2^*}\norm{u}^{2^*}_{2^*}- t\langle f,u \rangle\\
	&<0,
\end{align*}
for $t\ll 1$ small enough. Furthermore, by Remark \ref{rem_U_012}, $v:=tu\in U_1$. Hence the claim follows. Therefore, $\bar{\mathcal I}_{\lambda,a,f}(u_n)<0$ for large $n$. As an immediate consequence,
\begin{align*}
	0>\bar{\mathcal I}_{\lambda,a,f}(u_n)&\geq \f{1}{2}\norm{u_n}^2_\la-\f{\norm{a}_{\I}}{2^*}\norm{u_n}^{2^*}_{2^*} - \langle f,u_n \rangle\\
 &=\Big( \frac{1}{2}-\frac{1}{2^*(2^*-1)}\Big)\norm{u_n}_\lambda^2-\langle f,u_n \rangle.
\end{align*}
Thus for $n$ large enough, $\langle f,u_n \rangle>0$. Consequently, $\frac{d}{dt}J(tu_n)<0$ for $t>0$ small enough. Therefore by Step 1, there exists $t_n\in (0,1)$ such that $\frac{d}{dt}J(t_nu_n)=0$. Moreover, a direct calculation shows that the function $t\to \frac{d}{dt}J(tu)$ is strictly increasing in $[0,1)$ for all $u\in U_0$. As a result, the choice of $t_n$ is unique.

\item In this step, we prove that
\begin{equation}\label{liminf-J}
  \liminf_{n\to\I}\Bigl\{ J(u_n)-J(t_nu_n)\Bigr\}>0.
\end{equation}
Observe that
 $J(u_n)-J(t_nu_n)=\int_{t_n}^{1}\frac{d}{dt}\left\{J(tu_n)\right\}dt$, and for all $n\in \mathbb N$ there is a $\zeta_n>0$ such that $t_n\in(0, 1-2\zeta_n)$ and $\frac{d}{dt}J(tu_n)>\tfrac{\alpha}{2}$ for $t\in [1-\zeta_n,1]$.\\
 	To prove \eqref{liminf-J}, it suffices to show that $\xi_{n}>0$ can be chosen independent of $n \in \mathbb{N}$, which follows, as by step 1,  $\left.\frac{d}{d t} J\left(t u_{n}\right)\right|_{t=1} \geq \alpha$. Moreover, the boundedness of $(u_{n})_n$ gives
\begin{align*}
\left|{\frac{d^{2}}{d t^{2}} J\left(t u_{n}\right)}\right|&=\left|\left\|u_{n}\right\|_\lambda^{2}-(2^*-1)\|a\|_{\I} t^{2^*-2}\left\|u_{n}\right\|_{2^*}^{2^*}\right|\\
&=\left|\left(1-t^{2^*-2}\right)\left\|u_{n}\right\|_{\la}^{2}\right| \leq C,
\end{align*}
for all $n \geq 1$ and $t \in[0,1]$.\\
\item
  From the definition of $\bar{\mathcal I}_{\lambda,a,f}$ and $J$, it directly follows that  $\frac{d}{dt}	\bar{\mathcal I}_{\lambda,a,f}(tu)\geq \frac{d}{dt}J(tu)$ for all $u\in \mathcal{H}$ and $\forall\, t>0$. Therefore,
 \begin{align*}
 \bar{\mathcal I}_{\lambda,a,f}(u_n)- \bar{\mathcal I}_{\lambda,a,f}(t_nu_n)&=\int_{t_n}^{1}\frac{d}{dt}\left\{\bar{\mathcal I}_{\lambda,a,f}(tu_n)\right\}dt \\
 &\geq \int_{t_n}^{1}\frac{d}{dt}\left\{ J(tu_n)\right\}dt=J(u_n)-J(t_nu_n).
 \end{align*}
Since $(u_n)_n$ is a minimizing sequence for $ \bar{\mathcal I}_{\lambda,a,f}$ on $U_0$ and $t_nu_n\in U_1$ for all $n$, we deduce using \eqref{liminf-J} that
 \begin{equation*}
 	c_1=\inf_{U_1}\bar{\mathcal I}_{\lambda,a,f}(u)<\inf_{U_0}\bar{\mathcal I}_{\lambda,a,f}(u)=c_0.
 \end{equation*}
 This completes the proof. 
 \end{enumerate}
 \end{proof}
 The following result specifies the criteria for the validity of \eqref{Caf_comp}.
 \begin{proposition}\label{cond_1st sol}
 	If $\norm{f}_{\mathcal{H}'}<C_aS_\lambda^\frac{N}{4}$, then \eqref{Caf_comp} holds.
 \end{proposition}
 \begin{proof}
 	From the hypothesis, there exists an $\epsilon>0$ small enough such that  $  \norm{f}_{\mathcal{H}'}<C_aS_\lambda^\frac{N}{4}-\epsilon$. Therefore, using Lemma \ref{U0_Lbnd}, for all $u\in U_0$, we have 
 	\begin{equation*}
 		\langle f,u \rangle\leq \norm{f}_{\mathcal{H}'}\norm{u}_\lambda<\bigg[C_aS_\lambda^\frac{N}{4}-\epsilon\bigg]\norm{u}_\lambda\leq \frac{4}{N+2}\norm{u}_\lambda^2-\epsilon \norm{u}_\lambda.
 	\end{equation*}
 Thus, \begin{equation*}
\frac{4}{N+2}\norm{u}_\lambda^2	-\langle f,u \rangle \geq \epsilon \norm{u}_\lambda \quad\forall\, u\in U_0.
 \end{equation*}
 Combining this with Remark~\ref{rem_2*_bnd} yields 
\begin{equation}\label{inf_ufu_nrm}
\inf_{U_0}	\bigg[\frac{4}{N+2}\norm{u}_\lambda^2	-\langle f,u \rangle \bigg]\geq \epsilon \inf_{U_0}\norm{u}_\lambda>0.
\end{equation}
On the other hand, 
\begin{align*}
\eqref{Caf_comp}&\iff\inf_{u\in \mathcal{H}, \norm{u}_{2^*}=1}\Big\{ C_a\norm{u}^{\frac{N+2}{2}}_\la-\langle f,u\rangle\Big\}>0\\&\iff C_a\frac{\norm{u}^\frac{N+2}{2}_\lambda}{\norm{u}^{\frac{N}{2}}_{2^*}}-\langle f,u\rangle\,> \,0 \quad\text{for   } \norm{u}_{2^*}=1\\ 
&\iff C_a\frac{\norm{u}^\frac{N+2}{2}_\lambda}{\norm{u}^{\frac{N}{2}}_{2^*}}-\langle f,u\rangle\,> \,0 \quad\text{for   } u\in U_0\\ 
&\iff \frac{4}{N+2}\norm{u}_\lambda^2-\langle f,u\rangle\,> \,0 \quad\text{for   } u\in U_0.
\end{align*}
Thus the proposition follows from \eqref{inf_ufu_nrm} and the last equivalent estimations.
 \end{proof}
 
Now we are ready to prove the existence of at least two solutions, $u_1$ and $u_2$, to \eqref{Paf}.
 \begin{proposition}\label{1st_sol}
 	Assume that \eqref{Caf_comp} holds. Then $\mathcal{\bar I}_{\lambda,a,f}$ has a critical point $u_1\in U_1$ such that $\mathcal{\bar I}_{\lambda,a,f}(u_1)=c_1$. Moreover, $u_1$ solves the problem \eqref{Paf}.
 \end{proposition}
 
 \begin{proof}
 We divide the proof into the following few steps.  
 \begin{enumerate}[label = \textbf{Step \arabic*:}]
\item We claim that $c_1>-\I$.\\
As $\bar{\mathcal I}_{\lambda, a,f}(u)\geq J(u)$, to prove the claim, it is enough to show that $J$ is bounded from below. From the definition of $U_1$, we have 
\begin{align}\label{J_Lbnd}
	J(u)\geq \bigg[\frac{1}{2} -\frac{1}{2^*(2^{*}-1)}\bigg]\norm{u}_\lambda^2-\norm{f}_{\mathcal H'}\norm{u}_\lambda\quad\forall\, u\in U_1.
\end{align}
Since the RHS of the above expression is a quadratic function in $\|u\|_{\la}$, $J$ is bounded from below. This proves Step 1.

\medskip

\item In this step, we show that there exists a nonnegative bounded $PS$ sequence $(u_n)_n\subset U_1$ for $\bar{\mathcal I}_{\lambda, a,f}$ at the level $c_1$.\\
		Let $(u_{n})_n \subset \bar{U}_{1}$ such that $\bar{\mathcal I}_{\lambda,a,f}\left(u_{n}\right) \rightarrow c_{1}$. Since by Theorem \ref{comp_c1c0}, $c_{1}<c_{0}$, without restriction we can assume $u_{n} \in U_{1}$. Therefore, by Ekeland's variational principle, we can extract a PS sequence from $(u_{n})_n$ in $U_{1}$ for $\bar{\mathcal I}_{\lambda, a,f}$ at the level $c_{1}$. We still denote this PS sequence by $(u_{n})_n$. As $\bar{\mathcal I}_{\lambda,a,f}(u) \geq J(u)$ so, from \eqref{J_Lbnd}, we get $(u_{n})_n$ is a bounded sequence in $\mathcal{H}$. Therefore, up to a subsequence, $u_n \rightharpoonup u$ in $\mathcal{H}$ and $u_n \rightarrow u$ a.e. in $\mathbb{B}^N$. In particular, $\left(u_n\right)_{+} \rightarrow\left(u\right)_{+}$ and $\left(u_n\right)_{-} \rightarrow\left(u\right)_{-}$ a.e. in $\mathbb{B}^N$. Since $f$ is a nonnegative functional, we get as $n \rightarrow \infty$
\begin{equation*}
\begin{aligned}
o(1) & = \bar{I}_{\la, a, f}^{\prime}\left(u_n\right)\left[(u_n)_{-}\right]\\
& =\langle u_n,\left(u_n\right)_{-}\rangle_\lambda- \int_{\mathbb B^N}a(x)(u_n)_{+}^{2^*-1} (u_n)_{-}\;\mathrm{~d} V_{\bn}- \langle f, (u_n)_{-} \rangle  \\
& \leq-\left\|\left(u_n\right)_{-}\right\|_\lambda^2.
\end{aligned}
\end{equation*}
This shows that $\left(u_n\right)_{-} \rightarrow 0$ in $\mathcal{H}$, which implies $\left(u\right)_{-} \equiv 0$, that is, $u \geq 0$ a.e. in $\mathbb{B}^N$. As a result, without loss of generality, we can assume $(u_{n})_n$ is a nonnegative $P S$ sequence. Thus the proof of Step 2 is completed.

\medskip

\item  $(u_n)_n$ is a bounded $PS$ sequence as in Step 2, and let $u_n\rightharpoonup u$ in $ \mathcal{H}$.   We claim that $u\in U_1$ and $u_n\to u$ in $ \mathcal{H}$.\\

\noi
Applying the PS-decomposition, i.e., Theorem \ref{ps_decom}, we know that
$$u_n-\Big(u+\sum_{j=1}^{n_1}u_n^j+\sum_{k=1}^{n_2}v_n^k\Big)\to 0\,\,\text{  in  } \mathcal{H},$$
with $\bar{\mathcal I}'_{\lambda,a,f}(u)=0$. To prove Step 3, we need to prove that both $n_1=n_2=0$. We demonstrate it through contradiction. First, we consider the case when both $n_1\ne 0$ and $n_2\ne 0$. If one of them is zero, that case is again comparatively easier and the argument in that case
will follow from this case. We notice that
\begin{align*}
	g\big(u_n^j \big)&=g\big(\mathcal{U}\circ T_{n}^j \big)=\norm{\mathcal{U}\circ T_{n}^j}^2_\la-( 2^*-1)\norm{a}_{\I} \norm{\mathcal{U}\circ T_{n}^j}^{2^*}_{2^*}\\&=\norm{\mathcal{U}}^2_\la-( 2^*-1)\norm{a}_{\I} \norm{\mathcal{U}}^{2^*}_{2^*}=\Big[1-( 2^*-1)\norm{a}_{\I} \Big]\norm{\mathcal{U}}^2_\la<0,
\end{align*}
and
\begin{align*}
	g\big(v_n^j \big)&=\norm{ v_n^j}^2_\la-( 2^*-1)\norm{a}_{\I} \norm{ v_n^j}^{2^*}_{2^*}\\
     &=\int_{\mathbb{B}^N}a(x)|v^j_n|^{2^*}\mathrm{~d} V_{\bn}-( 2^*-1)\norm{a}_{\I} \norm{ v_n^j}^{2^*}_{2^*}+o(1)\\
	&\leq \|a\|_{\I} \norm{ v_n^j}^{2^*}_{2^*}(2-2^*)<0.
\end{align*}
Now Theorem \ref{ps_decom} gives
$$o(1)+c_1=\bar{\mathcal I}_{\lambda,a,f}(u_n)\to \bar{\mathcal I}_{\lambda,a,f}(u)+\sum_{j=1}^{n_1}{\mathcal I}_{\lambda,1,0}(\mathcal{U}^j)+\sum_{k=1}^{n_2} a(y^k)^{-\frac{N-2}{2}}J(V^k).$$
Since $a>0$, and both ${\mathcal I}_{\lambda,1,0}(\mathcal{U}^j), \, J(V^k)>0$, it follows  that $\mathcal I_{\lambda,a,f}(u)<c_1$. Therefore $u\not\in U_1$, which in turn leads to $g(u)\leq0$.

Next, we evaluate $g\Big(u+\sum_{j=1}^{n_1}u_n^j+\sum_{k=1}^{n_2}v_n^k\Big)$. Since $g(u_n)>0$ (as $u_n\in U_1$), the uniform continuity of $g$ implies
\begin{equation}\label{g_unifrm}
	0\leq \liminf_{n\to \I}g(u_n)=\liminf_{n\to \I} g\Big(u+\sum_{j=1}^{n_1}u_n^j+\sum_{k=1}^{n_2}v_n^k\Big).
\end{equation}
\noi
Since $u,\, u_n^j,\, v_n^k\geq 0$ for all $j=1,\cdots, n_1$, $k=1,\cdots n_2$,  
\begin{align*}
&\qquad \quad \;\;g\Big(u+\sum_{j=1}^{n_1}u_n^j+\sum_{k=1}^{n_2}v_n^k\Big)\\
&\qquad \quad=\norm{  u+\sum_{j=1}^{n_1}u_n^j+\sum_{k=1}^{n_2}v_n^k}^2_\la-( 2^*-1)\norm{a}_{\I} \norm{ u+\sum_{j=1}^{n_1}u_n^j+\sum_{k=1}^{n_2}v_n^k}^{2^*}_{2^*}\\
&\qquad \quad\leq \norm{  u }^2_\la+\sum_{j=1}^{n_1}\norm{  u_n^j }^2_\la+\sum_{k=1}^{n_2}\norm{   v_n^k}^2_\la +2\sum_{j=1}^{n_1}\langle u, u_n^j  \rangle_\lambda+2\sum_{k=1}^{n_2}\langle u, v_n^k  \rangle_\lambda+2\sum_{k=1}^{n_2}\sum_{j=1}^{n_1}\langle u_n^j, v_n^k  \rangle_\lambda \\
&\qquad \quad+\sum_{l,j=1, j\ne l}^{n_1}\langle u_n^j,  u_n^l\rangle_\lambda +\sum_{i,k=1,i\ne k}^{n_2}\langle v_n^k,  v_n^i  \rangle_\lambda -( 2^*-1)\norm{a}_{\I} \Big( \norm{  u }^{2^*}_{2^*}+\sum_{j=1}^{n_1}\norm{  u_n^j }^{2^*}_{2^*}+\sum_{k=1}^{n_2}\norm{   v_n^k}^{2^*}_{2^*}\Big)\\
&\qquad \quad=g(u)+\sum_{j=1}^{n_1}g(u_n^j)+\sum_{k=1}^{n_2}g(  v_n^k)\\ 
&\qquad \quad+	\underbrace{2\sum_{j=1}^{n_1}\langle u, u_n^j  \rangle_\lambda}_{(i1)}+    \underbrace{2\sum_{k=1}^{n_2}\langle u, v_n^k  \rangle_\lambda}_{(i2)}+\underbrace{2\sum_{k=1}^{n_2}\sum_{j=1}^{n_1}\langle u_n^j, v_n^k  \rangle_\lambda}_{(i3)} 
\quad+\underbrace{\sum_{l,j=1, j\ne l}^{n_1}\langle u_n^j,  u_n^l\rangle_\lambda}_{(i4)}+\underbrace{\sum_{i,k=1, i\ne k}^{n_2}\langle v_n^k,  v_n^i  \rangle_\lambda}_{(i5)}.
\end{align*}
\begin{equation}\label{l:o(1)}
\hspace{-3.8cm}\text{\textbf{Claim: }All the above inner-product terms are }  o(1). 
\end{equation}

\noi
In the affirmative case, we obtain \begin{equation*}
  g\Big(u+\sum_{j=1}^{n_1}u_n^j+\sum_{k=1}^{n_2}v_n^k\Big)<0.
\end{equation*}
This contradicts \eqref{g_unifrm}. Hence $n_1=n_2=0$. Thus $u_n\to u$ in $ \mathcal{H}$. From the continuity of $g$, we conclude that $g(u)\geq0$. Therefore  $u\in \bar U_1$. Since $c_1<c_0$, it implies that $u\in  U_1$.  Hence to conclude the proof of Step 3, we are only left to prove \eqref{l:o(1)}.
\noi
Now we prove the above claim. \\
Consider a single term of $(i1)$. Since  $\mathcal I_{\lambda,a,f}'(u)=0$, we have 
$$
 \langle u, u_n^j  \rangle_\lambda= \int_{\mathbb B^N}a(x)|u|^{2^*-1}u_n^j\mathrm{~d} V_{\bn}(x)+ \int_{\mathbb B^N}f(x) u_n^j\, \mathrm{~d} V_{\bn}(x).
$$
In this case, as $T_n(0)\to \I$, we know that $u_n^j=\mathcal{U}(T^j_n)\rightharpoonup 0$. Therefore,\\$\int_{\mathbb B^N}f(x) u_n^j\, \mathrm{~d} V_{\bn}(x)=o(1)$. On the other hand, applying Vitali's convergence theorem via H\"older inequality on the other term, we obtain $\langle u, u_n^j  \rangle_\lambda=o(1)$.

For the second term ($i_2$), applying the conformal change of metric (see Section \ref{pre}), we obtain 
\begin{align*}
&	\langle u, v_n^k  \rangle_\lambda=\int_{\mathbb B^N}\nabla_{\mathbb{B}^{N}} u\nabla_{\mathbb{B}^{N}}v_n^k\dvg-\int_{\mathbb B^N}\lambda uv_n^k\dvg\\
&= {(\eps_n^k)}^{{\frac{2-N}{2}}}a(y^k)^{\frac{2-N}{4}}  \int_{B(0,1)}  \nabla \Big(	\Big(\frac{2}{1-|x|^2}\Big)^\f{N-2}{2} u   \Big)\cdot\nabla \Big( \phi(x)   W\(\f{x-y^k_n}{\eps_n^k}\)\Big)\mathrm{~d} x\\
&\quad-{(\eps_n^k)}^{\frac{2-N}{2}}a(y^k)^{\frac{2-N}{4}} \int_{B(0,1)}  c(x)  	  	\Big(\frac{2}{1-|x|^2}\Big)^\f{N-2}{2}u\phi(x)   W\(\f{x-y^k_n}{\eps_n^k}\)  \mathrm{~d} x \\
&= -{(\eps_n^k)}^{\frac{2-N}{2}}a(y^k)^{\frac{2-N}{4}}  \int_{B(0,1)}      W\(\f{x-y^k_n}{\eps_n^k}\)   \left(\phi(x) \Delta \Big(	\Big(\frac{2}{1-|x|^2}\Big)^\f{N-2}{2} u   \Big)\right)\\
&\quad-{(\eps_n^k)}^{\frac{2-N}{2}}a(y^k)^{\frac{2-N}{4}} \int_{B(0,1)}  c(x)  	\phi(x)  	\Big(\frac{2}{1-|x|^2}\Big)^\f{N-2}{2}u   W\(\f{x-y^k_n}{\eps_n^k}\)  \mathrm{~d} x \\
& = o(1),  
\end{align*}
where the first term is $o(1)$, follows from the fact that ${(\eps_n^k)}^{\frac{2-N}{2}}W\(\f{x-y^k_n}{\eps_n^k}\)\rightharpoonup 0$, and for the second term, we can argue as in the estimates in \eqref{l:new est1}.  Here $c(.)$ is as defined in \eqref{con-ecl-eq}.

\medskip

\noi Our next aim is to show that each term of $(i_{3})=o(1)$. 

Note that  $\mathcal U(T_n^j)$ is a PS sequence for the energy functional $\mathcal I_{\lambda,1,0}$. Therefore using H\"older inequality
\begin{align} \label{HA}
   \quad \langle  \mathcal U(T_n^j), v_n^k\rangle_\la &= \int_{\mathbb{B}^N} | \mathcal U(T_n^j)|^{2^*-1}  v_n^k \dvg +o(1)\no \\
    &= \int_{B_R(0)}| \mathcal U(T_n^j)|^{2^*-1}  v_n^k \dvg +o(1) \\ \notag
    & \leq \norm{\mathcal U(T_n^j)}_{L^{2^*}(B_R(0))}^{2^*-1}\norm{v_n^k}_{2^*}+o(1) \leq C\norm{\mathcal U(T_n^j)}_{L^{2^*}(B_R(0))}^{2^*-1} +o(1).
\end{align}
Now to prove $(i_{3})=o(1)$, it is enough to show $\norm{\mathcal U(T_n^j)}_{L^{2^*}(B_R(0))}=o(1)$. Set $z_n:=T_n^j(0) $, 
so $d(0,z_n)\to\I$. Therefore
$$\norm{\mathcal U(T_n^j)}_{L^{2^*}(B_R(0))}^{2^*}=\int_{B_R(0)}| \mathcal U(T_n^j)|^{2^*}   \dvg=\int_{B_R(z_n)}\mathcal U^{2^*}   \dvg=o(1),$$  
\noi
Next, we estimate $(i_4)$.
\begin{align*}
\langle u_n^j,  u_n^l\rangle_\lambda&=\int_{\mathbb B^N}\nabla_{\mathbb{B}^{N}} u_n^j\nabla_{\mathbb{B}^{N}}u_n^l\mathrm{~d} V_{\bn}-\lambda\int_{\mathbb B^N} u_n^ju_n^l\mathrm{~d} V_{\bn}\\
&=\int_{\mathbb B^N}\nabla_{\mathbb{B}^{N}}  \mathcal U(T^j_n)\nabla_{\mathbb{B}^{N}} \mathcal U(T^l_n)\dvg-\lambda\int_{\mathbb B^N}  \mathcal U(T^j_n) \mathcal U(T^l_n)\dvg\\
&=\int_{\mathbb B^N}\nabla_{\mathbb{B}^{N}} \mathcal U(T^j_nT^{-l}_n)\nabla_{\mathbb{B}^{N}} \mathcal U\dvg-\lambda\int_{\mathbb B^N}  \mathcal U(T^j_nT^{-l}_n) \mathcal U\dvg.
\end{align*}
From Theorem~\ref{ps_decom}(i), we have $T^j_nT^{-l}_n(0)\to \I$ and thus, $\mathcal U(T^j_nT^{-l}_n)\rightharpoonup 0$. Consequently,  $i_4 = o(1).$

Finally, we prove $(i_{5})=o(1)$. Following the notation $w_n^k$  as in  \eqref{wnk} and $(w_{n})_n$ is a PS sequence corresponding for the functional associated with \ef{con-ecl-eq}, we get
\begin{align*}
    \langle v_n^k,  v_n^i  \rangle_\lambda&=\int_{B(0,R)}(\grad w_n^k\grad w_n^i-c(x)w_n^kw_n^i)\mathrm{~d} x\\ &=\int_{B(0,R)}a(x)|w_{n}^k|^{2^*-1}w_{n}^i\, \mathrm{~d} x+o(1).
\end{align*}
Since $w_n^i$ are compactly supported for each $i$, 
\begin{align*}
   &\int_{\mathbb R^N}a(x)|w_{n}^k|^{2^*-1}w_{n}^i\, \mathrm{~d} x\\ &\leq C\int_{\mathbb R^N}\Big|(\eps_n^j)^{\f{2-N}{2}}  W\(\f{x-y_n^j}{\eps_n^j}\) \Big|^{2^*-1}(\eps_n^i)^{\f{2-N}{2}}  W\(\f{x-y_n^i}{\eps_n^i}\) \, \mathrm{~d} x \\
   &= \Big(\f{\eps_n^j}{\eps_n^i}\Big)^{ \frac{N-2}{2}} \int_{\mathbb R^N}\Big|  W\(y\) \Big|^{2^*-1}  W\(\f{\eps_n^j}{\eps_n^i}y+ \f{y_n^j-y_n^i}{\eps_n^i}\) \,  \mathrm{d} y \\
   &=\int_{\mathbb R^N}\grad W\grad W^{ij}_n=\langle W,\, W^{ij}_n \rangle=o(1),
\end{align*}
where $W^{ij}_n=\Big(\f{\eps_n^j}{\eps_n^i}\Big)^{ \frac{N-2}{2}} W\(\f{\eps_n^j}{\eps_n^i}y+ \f{y_n^j-y_n^i}{\eps_n^i}\)$. We used the fact that $W$ solves the problem \eqref{lim-pro} and  $W^{ij}_n\to 0$ weakly in $D^{1,2}(\mathbb{R}^N)$(see Lemma \ref{lemTn})  as $$\Big| \log\big(\frac{\eps^j_{n} }{\eps^i_{n}}\big)\Big|+ \Big| \frac{y_n^j-y^i_n}{\eps^i_{n}}\Big|\to \I$$ as $n\to \I$ for $i\ne j$.\\
This completes the proof of this proposition. 
We denote the critical point $u$ of $\bar{\mathcal I}_{\la,a,f}$ found here as $u_1$.
\end{enumerate}
 \end{proof} 

 \subsection{Existence of Second critical point} In this section, we prove the existence of a second critical point distinct from the former one. 
\begin{proposition}\label{2nd_sol}
	Assume that \eqref{Caf_comp} holds. Furthermore, if $\norm{a}_{\I}< (\f{S}{S_{\lambda}})^\f{N}{N-2}$. Then $\mathcal{\bar I}_{\lambda,a,f}$ has a second critical point $v_1\not\equiv u_1$. 
\end{proposition}

\begin{proof}
Let $u_1$ be a critical point of $\mathcal{\bar I}_{\lambda, a,f}$ that was obtained in Proposition \ref{1st_sol} and $\mathcal{U}$ be a positive ground state solution of \eqref{E}. 
Set $\mathcal{U}_t(x):=t\;\mathcal{U}(x)$ for $t\ge 0$.\\
\begin{enumerate}[label = \textbf{Claim \arabic*:}]
\item $u_1+\norm{a}_{\I}^{\frac{2-N}{4}}\mathcal{U}_t\in U_2$ for sufficiently  large $t>0$. Since $\mathcal{U}_t,\, u_1>0$ and $\norm{a}_{\I}\geq 1$, we have
\begin{align*}
& \qquad \quad	g(u_1+\mathcal{U}_t)=\norm{ u_1+  \norm{a}_{\I}^{\frac{2-N}{4}}\mathcal{U}_t}^2_\la-( 2^*-1)\norm{a}_{\I} \norm{ u_1+ \norm{a}_{\I}^{\frac{2-N}{4}}\mathcal{U}_t}^{2^*}_{2^*}\\
&\qquad \quad\leq \norm{  u_1}^2_\la+\norm{a}_{\I}^{\frac{2-N}{2}} \norm{  \mathcal{U}_t}^2_\la+2\norm{a}_{\I}^{\frac{2-N}{4}}\langle u_1\,,\,\mathcal{U}_t\rangle_\la-( 2^*-1) \Big(\norm{  u_1}^{2^*}_{2^*}+\norm{a}_{\I}^{\frac{-N}{2}}\norm{  \mathcal{U}_t}^{2^*}_{2^*}\Big)\\
&\qquad \quad\leq \norm{  u_1}^2_\la+\norm{a}_{\I}^{\frac{2-N}{2}}t^2 \norm{  \mathcal{U}}^2_\la+2\norm{a}_{\I}^{\frac{2-N}{4}}\big[\frac{1}{2\eps}\norm{  u_1}^2_\la+ \frac{\eps}{2}\norm{  t\mathcal{U}}^2_\la\big] \\
& \qquad \quad-( 2^*-1) \Big(\norm{  u_1}^{2^*}_{2^*}+t^{2^*}\norm{a}_{\I}^{\frac{-N}{2}}\norm{  \mathcal{U}}^{2^*}_{2^*}\Big)	\\
&\qquad \quad<0\quad\text{ for large enough t}.	
\end{align*}
Thus Claim 1 follows.
\medskip
\item $	\mathcal I_{\lambda,a,f}(u_1+\norm{a}_{\I}^{\frac{2-N}{4}}\mathcal{U}_t)<\mathcal I_{\lambda,a,f}(u_1)+\mathcal I_{\lambda,1,0}(\norm{a}_{\I}^{\frac{2-N}{4}}\mathcal{U}_t)$ for all $t>0$.\\

\medskip
\noi
Using $\norm{a}_{\I}^{\frac{2-N}{4}}\mathcal{U}_t$ as  a test function yields
$$ \quad  \langle u_1\,,\,\norm{a}_{\I}^{\frac{2-N}{4}}\mathcal{U}_t\rangle_{\la}= \norm{a}_{\I}^{\frac{2-N}{4}}\int_{\mathbb B^N}a(x)|u_1|^{2^*-1}\mathcal{U}_t\mathrm{~d} V_{\bn}(x)+ \langle f(x),\norm{a}_{\I}^{\frac{2-N}{4}}\mathcal{U}_t\rangle.$$
\noi
Therefore, 
\begin{align*}
\qquad \quad\; &\mathcal I_{\lambda,a,f}(u_1+\norm{a}_{\I}^{\frac{2-N}{4}}\mathcal{U}_t)\\&=\f{1}{2}\norm{\norm{a}_{\I}^{\frac{2-N}{4}}\mathcal{U}_t+u_1}^2_\la-\f{1}{2^*}\int_{\mathbb B^N}a(x)|\norm{a}_{\I}^{\frac{2-N}{4}}\mathcal{U}_t+u_1|^{2^*}\mathrm{~d} V_{\bn}- \langle f(x),\norm{a}_{\I}^{\frac{2-N}{4}}\mathcal{U}_t+u_1\rangle\\
&\leq I_{\lambda,a,f}(u_1)+I_{\lambda,1,0}(\norm{a}_{\I}^{\frac{2-N}{4}}\mathcal{U}_t)+\langle \norm{a}_{\I}^{\frac{2-N}{4}}\mathcal{U}_t\,,\,u_1 \rangle_{\la}- \langle f(x),\norm{a}_{\I}^{\frac{2-N}{4}}\mathcal{U}_t\rangle\\&-\f{1}{2^*}\int_{\mathbb B^N}a(x)\Big(|\norm{a}_{\I}^{\frac{2-N}{4}}\mathcal{U}_t+u_1|^{2^*} -| u_1|^{2^*}\mathrm{~d} V_{\bn}(x)-|\norm{a}_{\I}^{\frac{2-N}{4}}\mathcal{U}_t|^{2^*}\Big)\mathrm{~d} V_{\bn} \\
&=\mathcal I_{\lambda,a,f}(u_1)+\mathcal I_{\lambda,1,0}(\norm{a}_{\I}^{\frac{2-N}{4}}\mathcal{U}_t)  \\&+\f{1}{2^*}\int_{\mathbb B^N}a(x)\Big( | u_1|^{2^*} +\norm{a}_{\I}^{\frac{-N}{2}}|\mathcal{U}_t|^{2^*}+2^*\norm{a}_{\I}^{\frac{2-N}{4}}|u_1|^{2^*-1}\mathcal{U}_t-\Big|\norm{a}_{\I}^{\frac{2-N}{4}}\mathcal{U}_t+u_1\Big|^{2^*}\Big)\mathrm{~d} V_{\bn} \\
&<\mathcal I_{\lambda,a,f}(u_1)+\mathcal I_{\lambda,1,0}(\norm{a}_{\I}^{\frac{2-N}{4}}\mathcal{U}_t).
\end{align*}
Thus claim 2 follows.  
\noi
Now, note that 
$$\mathcal I_{\lambda,1,0}(\norm{a}_{\I}^{\frac{2-N}{4}}\mathcal{U}_t)
	=\f{t^2}{2}\norm{\norm{a}_{\I}^{\frac{2-N}{4}}\mathcal{U}}^2_\la-\f{t^{2^*}}{2^*}\int_{\mathbb B^N}|\norm{a}_{\I}^{\frac{2-N}{4}}\mathcal{U}|^{2^*}\mathrm{~d} V_{\bn} \rightarrow -\infty $$
 as $t\to\I$. Consequently, a straightforward computation yields that
\begin{align*}
    \qquad \quad \sup_{t>0} \mathcal I_{\lambda,1,0}(\norm{a}_{\I}^{\frac{2-N}{4}}\mathcal{U}_t)=\mathcal I_{\lambda,1,0}\Big(\norm{a}_{\I}^{\frac{2-N}{4}}\mathcal{U}_{t_{max}}\Big) =\mathcal I_{\lambda,1,0} ( \mathcal{U} ) \;\;\text{ where } t_{max} = \norm{a}_{\I}^{\frac{  N-2 }{4}}.
    \end{align*}

\noi
Thus, using the above observations and from  Claim 2, we conclude
\begin{align*}
\qquad \quad \;\;\mathcal I_{\lambda, a,f}(\norm{a}_{\I}^{\frac{2-N}{4}}\mathcal{U}_t+u_1)<\mathcal I_{\lambda, a,f}(u_1)+\mathcal I_{\lambda,1,0}(\norm{a}_{\I}^{\frac{2-N}{4}}\mathcal{U}_t)<\mathcal I_{\lambda, a,f}(u_1)\text{ for large }t>0
\end{align*}
and  
\begin{equation} \mathcal I_{\lambda,a,f}(\norm{a}_{\I}^{\frac{2-N}{4}}\mathcal{U}_t+u_1)<\mathcal I_{\lambda,a,f}(u_1)+\mathcal I_{\lambda,1,0}(\mathcal{U})\,\,\,\text{ for all }t>0.\label{for_all_t}\end{equation}
Now using Claim 1 and the above estimate, we fix a $t_0>0$  large enough such that\\ 
$\Big(\norm{a}_{\I}^{\frac{2-N}{4}}\mathcal{U}_{t_0}+u_1\Big)\in U_2$ and $\mathcal I_{\lambda,a,f}\Big(\norm{a}_{\I}^{\frac{2-N}{4}}\mathcal{U}_{t_0}+u_1\Big)<\mathcal I_{\lambda,a,f}(u_1)$. Next, we define the min-max level $k$ as
$$ k:=\inf_{\gamma\in\Gamma}\sup_{t\in[0,1]}\mathcal I_{\lambda,a,f}(\gamma(t)), $$
where $\Gamma:=\Bigl\{ \gamma \in C\Big([0,1], \mathcal{H}\Big): \gamma(0)=u_1,\quad \gamma(1)=\norm{a}_{\I}^{\frac{2-N}{4}}\mathcal{U}_{t_0}+u_1\Bigr\}.$  

As  $u_1\in U_1$ and $\norm{a}_{\I}^{\frac{2-N}{4}}\mathcal{U}_{t_0}+u_1\in U_2$, for every $\gamma\in\Gamma$, there exists $t_\gamma\in(0,1)$ such that $\gamma(t_\gamma)\in U_0$. Therefore,
$$ \sup_{t\in[0,1]}\mathcal I_{\lambda,a,f}(\gamma(t))\geq  \mathcal I_{\lambda,a,f}(\gamma(t_\gamma))\geq \inf_{u\in U_0}\mathcal I_{\lambda,a,f}(u)=c_0.$$ Thus allowing infimum on both sides and using Theorem \ref{comp_c1c0}, we get
$$ k=\inf_{\gamma\in\Gamma}\sup_{t\in[0,1]}\mathcal I_{\lambda,a,f}(\gamma(t)) \geq  c_0>c_1=\mathcal I_{\lambda,a,f}(u_1).$$
\item $I_{\lambda,a,f}(u_1)<k<\min\Big\{\mathcal I_{\lambda,a,f}(u_1)+\mathcal I_{\lambda,1,0}(\mathcal{U}), \, \mathcal I_{\lambda,a,f}(u_1)+ a(x)^{-\frac{N-2}{2}}J(V)\Big\}$,
where 
$a(x)= max\{a(y^i); \;1\leq i \leq n_2 \text{ where } n_2 \text{ is as found in Theorem \ref{ps_decom}}\}$.

\medskip

Define $\tilde{\gamma}(t):=u_1+\norm{a}_{\I}^{\frac{2-N}{4}}\mathcal{U}_{t_0t}$, then $\tilde{\gamma}\in\Gamma$ and $\lim_{t\to 0}\|\tilde{\gamma}(t)-u_1\|_{\la}=0$. Therefore, from the definition of $k$ and \eqref{for_all_t}, it follows

\begin{align} \label{k_bnd}
\mathcal I_{\lambda,a,f}(u_1)<k\leq \sup_{t\in[0,1]}\mathcal I_{\lambda,a,f}(\tilde\gamma(t))&= \sup_{t\in[0,1]}\mathcal I_{\lambda,a,f}\Big(u_1+\norm{a}_{\I}^{\frac{2-N}{4}}\mathcal{U}_{t_0t}\Big)\\ \notag
& < \mathcal I_{\lambda,a,f}(u_1)+\mathcal I_{\lambda,1,0}(\mathcal{U}). 
\end{align}

\noi
On the other hand, since $S_\lambda<S$ for $\lambda>\frac{N(N-2)}{4}$, let us assume  $$ a(x_0)=\norm{a}_{\I}< \Big(\frac{S}{S_{\lambda}}\Big)^\frac{N}{N-2}.$$  As $S_\lambda$ and $S$ are achieved  by  $\mathcal{U} \in \mathcal{H}$ and $V\in D^{1,2}(\mathbb{R}^N)$ respectively, a direct computation yields that $\norm{\mathcal{U}}_\lambda^2=S_\lambda^\frac{N}{2}$ and $\norm{V}_{D^{1,2}(\mathbb{R}^N)}^2=S^\frac{N}{2}$ and 
$\frac{J(V)}{\mathcal I_{\lambda,1,0}(\mathcal{U})} = \Big(\frac{S}{S_\lambda}\Big)^\frac{N}{2}$. Thus
\begin{align*}
	   {\mathcal I_{\lambda,1,0}(\mathcal{U})}< \norm{a}_{\I}^\frac{-(N-2)}{2}{J(V)}\leq a(x)^\frac{-(N-2)}{2}{J(V)}.
\end{align*}
Comparing the above estimation with \eqref{k_bnd}, we obtain  
\begin{equation}\label{k_AT}
\mathcal	I_{\lambda,a,f}(u_1)<k<\mathcal I_{\lambda,a,f}(u_1)+ a(x)^{-\frac{(N-2)}{2}}{J(V)}.
\end{equation}
Combining \eqref{k_bnd} and \eqref{k_AT}, Claim 3 follows. 

Applying Ekeland's variational principle, there exists a PS sequence $(u_n)_n$ at the level $k$. Hence it is bounded with a weak limit $v_1$. Using the PS decomposition of $(u_n)_n$, i.e., Theorem~\ref{ps_decom}, we obtain $u_n\to v_1$. Further, $\mathcal	I_{\lambda,a,f}(v_1)= k$ and $\mathcal	I'_{\lambda,a,f}(v_1)= 0$. Since $\mathcal	I_{\lambda,a,f}(u_1)<k=\mathcal I_{\lambda,a,f}(v_1)$, we obtain that $v_1\not\equiv u_1$. Thus $v_1$ is a second positive solution of \eqref{Paf}. Since $f$ in non-zero implies both $u_1$ and $v_1$ are non-zero.
\end{enumerate}
\end{proof}

\medskip

{\bf Proof of Theorem~\ref{Thm-A1} concluded.}
\begin{proof}
Combining Propositions~\ref{1st_sol} and \ref{2nd_sol} with Proposition~\ref{cond_1st sol}, we conclude the proof of Theorem~\ref{Thm-A1}.
\end{proof}

\section{Key Energy Estimates}\label{KER}

This section establishes key energy estimates for the functional $\bar{\mathcal I}_{\lambda, a,f}$ when  $a(x) \leq 1$. The following 
energy estimates will play a vital role in the existence of solutions. Using the Proposition~\ref{energy-prop},
we shall prove that the energy of the functional is below the critical level given in the Palais-Smale decomposition. The proof of the next proposition follows on the same line as \cite[Proposition~6.1]{GGS}, with obvious modifications for the case $p = 2^{\star}.$ For brevity, we provide a sketch of the proof. 


\begin{proposition}\label{energy-prop}
	Let $a(.)$ satisfy \eqref{A_cond}, \eqref{A2} and \eqref{assum_A}. Further, suppose that $\|f\|_{\mathcal H'}{\leq}\; d_{2}$, where $d_2$ as in Proposition \ref{Relcrit_I&E}, $\mathcal{U}$ is the unique positive solution of \eqref{E} and $\overline{\mathcal{U}}_{a, f}$ is any critical point of $\bar{\mathcal I}_{\la,a,f}$. Then there exists $R_0 > 0$ such that
	\begin{equation}
		\bar{\mathcal I}_{\la,a, f}\left(\overline{\mathcal{U}}_{a, f} (x)+t \mathcal{U}(\tau_{-y}(x))\right)<\bar{\mathcal I}_{\la,a, f}\left(\overline{\mathcal{U}}_{a, f}(x)\right)+\mathcal I_{\la,1,0}(\mathcal U),\;\;\; \label{4.ab}
	\end{equation} 
	\label{enerest}
	for all $d(y,0) \geq R_0$ and $t>0$.
\end{proposition}
\begin{proof}
	Carrying out simple calculations provides
	\begin{equation}
		\begin{aligned}
			&\bar{\mathcal I}_{\la,a, f}\left(\overline{\mathcal{U}}_{a, f}(x)+t \mathcal{U}(\tau_{-y}(x))\right)
			=\frac{1}{2}\left\|\overline{\mathcal{U}}_{a, f}(x)\right\|_{\la}^{2}+\frac{t^{2}}{2}\|\mathcal U\|_{\la}^{2}
			+t\left\langle \overline{\mathcal{U}}_{a, f}(x), \mathcal U(\tau_{-y}(x))\right\rangle_{\la} \\
		&-\frac{1}{2^*} \int_{\mathbb{B}^{N}} a(x)\left(\overline{\mathcal{U}}_{a, f}(x)\right)^{2^*}\mathrm{~d}V_{\mathbb{B}^{N}}(x)
			-\frac{t^{2^*}}{2^*} \int_{\mathbb{B}^{N}}\; a(x)\; (\mathcal U(\tau_{-y}(x))^{2^*}\mathrm{~d}V_{\mathbb{B}^{N}}  \\
			&-\frac{1}{2^*} \int_{\mathbb{B}^{N}} a(x)\left\{\left(\overline{\mathcal{U}}_{a, f}(x)+t \mathcal U(\tau_{-y}(x))\right)^{2^*}\right.
		-\left. \left(\overline{\mathcal{U}}_{a, f}(x)\right)^{2^*}- t^{2^*}\mathcal U(\tau_{-y}(x))^{2^*}\right\} \mathrm{~d} V_{\mathbb{B}^{N}}  \\
			&-\int_{\mathbb{B}^{N}} f(x)\left(\overline{\mathcal{U}}_{a, f}(x)+t \mathcal U(\tau_{-y}(x))\right)\mathrm{~d}V_{\mathbb{B}^{N}}.
		\end{aligned} \label{4.yy}
	\end{equation}
Moreover, for any $h \in \mathcal H$, it holds
	\begin{align*}
		0&=I_{\la,a, f}^{\prime}\left(\overline{\mathcal{U}}_{a, f}(x)\right)(h) \\
		&=\left\langle \overline{\mathcal{U}}_{a, f}(x), h\right\rangle_{\la}-\int_{\mathbb{B}^{N}} a(x)\left(\overline{\mathcal{U}}_{a, f}(x)\right)^{2^*-1} h \mathrm{~d} V_{\mathbb{B}^{N}}(x)-\int_{\mathbb{B}^{N}} f h \mathrm{~d} V_{\mathbb{B}^{N}}(x),
	\end{align*}
	i.e.,
	\begin{equation*}
		\left\langle \overline{\mathcal{U}}_{a, f}(x), h\right\rangle_{\la}=\int_{\mathbb{B}^{N}} a(x)\left(\overline{\mathcal{U}}_{a, f}(x)\right)^{2^*-1} h \mathrm{~d} V_{\mathbb{B}^{N}}(x)+\int_{\mathbb{B}^{N}} f h \mathrm{~d} V_{\mathbb{B}^{N}}(x).
	\end{equation*}
	Specifically, using $h=t\mathcal U(\tau_{-y}(x))$ in the above equality gives
	\begin{align*}
		&t\left\langle \overline{\mathcal{U}}_{a, f}(x), \mathcal U(\tau_{-y}(x))\right\rangle_{\la} \\
		&=t \int_{\mathbb{B}^{N}} a(x)\left(\overline{\mathcal{U}}_{a, f}(x)\right)^{2^*-1} \mathcal U(\tau_{-y}(x)) \mathrm{~d} V_{\mathbb{B}^{N}}(x)+t \int_{\mathbb{B}^{N}} 
		f \mathcal U(\tau_{-y}(x)) \mathrm{~d} V_{\mathbb{B}^{N}}(x) .
	\end{align*}
 Therefore, applying the equality mentioned above and properly rearranging the terms in \eqref{4.yy} leads to
	\begin{equation*}
		\begin{aligned}
			&\bar{\mathcal I}_{\la,a, f}\left(\overline{\mathcal{U}}_{a, f}(x)+t \mathcal U(\tau_{-y}(x))\right)=\bar{\mathcal I}_{\la,a, f}\left(\overline{\mathcal{U}}_{a, f}\right)+{\mathcal I}_{\la,1,0}(t \mathcal U)\\
			&+\frac{t^{2^*}}{2^*} \int_{\mathbb{B}^{N}}(1-a(x))\mathcal U(\tau_{-y}(x))^{2^*}\mathrm{~d} V_{\mathbb{B}^{N}}(x)\\
			&-\frac{1}{2^*} \int_{\mathbb{B}^{N}} a(x)\left\{\left(\overline{\mathcal{U}}_{a, f}(x)+t\mathcal U(\tau_{-y}(x))\right)^{2^*}-\left(\overline{\mathcal{U}}_{a, f}(x)\right)^{2^*}\right. \\
			&\left.-2^*t\left(\overline{\mathcal{U}}_{a, f}(x)\right)^{2^*-1} \mathcal U(\tau_{-y}(x))-t^{2^*}\mathcal U(\tau_{-y}(x))^{2^*}\right\} \mathrm{~d} V_{\mathbb{B}^{N}}(x)\\
			&=\bar{\mathcal I}_{\la,a, f}\left(\overline{\mathcal{U}}_{a, f}(x)\right)+\mathcal I_{\la,1,0}(t \mathcal U)\;+\; \underbrace{(I)-(II)},
		\end{aligned}
	\end{equation*}
	where
	\begin{equation}
		I:=\frac{t^{2^*}}{2^*} \int_{\mathbb{B}^{N}}(1-a(x))\mathcal U(\tau_{-y}(x))^{2^*} \mathrm{~d} V_{\mathbb{B}^{N}}(x),
		\label{term1}
	\end{equation}
	and
	\begin{equation}
		\begin{aligned}
			II &:=\frac{1}{2^*} \int_{\mathbb{B}^{N}} a(x)\left\{\left(\overline{\mathcal{U}}_{a, f}(x)+t\mathcal U(\tau_{-y}(x))\right)^{2^*}-\left(\overline{\mathcal{U}}_{a, f}(x)\right)^{2^*}\right.\\
			&\left.-2^*t\left(\overline{\mathcal{U}}_{a, f}(x)\right)^{2^*-1} \mathcal U(\tau_{-y}(x))-t^{2^*} \mathcal U(\tau_{-y}(x))^{2^*}\right\}\mathrm{~d} V_{\mathbb{B}^{N}}(x).
		\end{aligned}
		\label{term2}
	\end{equation}
	We are required to exhibit that $(I) \; -\;(II) < 0$, for appropriately chosen $R >0,$ to complete the proof of the proposition.\\
	\noi
Using the continuity, we easily get $$\bar{\mathcal I}_{\la,a, f}\left(\overline{\mathcal{U}}_{a, f}(x)+t\mathcal U(\tau_{-y}(x))\right) \rightarrow \bar{\mathcal I}_{\la,a, f}(\overline{\mathcal{U}}_{a, f}(x)),$$ as $t \rightarrow 0$.
	Additionally, we have
	\begin{equation*}
		\bar{\mathcal I}_{\la,a, f}\left(\overline{\mathcal{U}}_{a, f}(x)+t \mathcal U(\tau_{-y}(x))\right) \rightarrow-\infty \text{ as } t \rightarrow \infty.
	\end{equation*}
 Consequently, by using the above two statements, we can determine $m, M$ with $0<m<M$ such that
	\begin{equation*}
		\bar{\mathcal I}_{\la,a, f}\left(\overline{\mathcal{U}}_{a, f}(x)+t\mathcal U(\tau_{-y}(x)\right)<\bar{\mathcal I}_{\la,a, f}\left(\overline{\mathcal{U}}_{a, f}\right)+\mathcal I_{\la,1,0}(\mathcal U)\;\text { for all } t \in(0, m) \cup(M, \infty).
	\end{equation*}
	
	\medskip
	\noi
	Therefore, proving the above proposition only requires that \eqref{4.ab} be proved for $t \in[m, M]$.
	Thus, we must show $I<II$ to complete the proof. To this end, let us recall the following standard inequalities from calculus: 
	\begin{enumerate}
		\item $(s+t)^{2^*}-s^{2^*}-t^{2^*}-(2^*) s^{2^*-1} t \geq 0$ \text{for all }$(s, t) \in[0, \infty) \times[0, \infty)$.\\
		\item \text{For any} $ r > 0$ \text{we can find a constant} $A(r) > 0$ \text{such that}
		\begin{equation*}
			(s+t)^{2^*}-s^{2^*}-t^{2^*}-(2^*) s^{2^*-1} t \geq A(r) t^{2},
		\end{equation*}
		for all $(s, t) \in[r, \infty) \times[0, \infty)$.
	\end{enumerate}

	Using the inequalities mentioned above, we will estimate $II$ as follows:
	\medskip
	
	Set $A:=A(r):=A\left(\min _{d(x,0) \leq 1} \overline{\mathcal{U}}_{a, f}(x)\right)>$ 0, then
	\begin{equation*}
		\begin{aligned}
			II &:=\frac{1}{2^*} \int_{\mathbb B^N} a(x)\left\{\left(\overline{\mathcal{U}}_{a, f}(x)+t\mathcal U(\tau_{-y}(x))\right)^{2^*}-\left(\overline{\mathcal{U}}_{a, f}(x)\right)^{2^*}\right.\\
			&\qquad\left.-2^*t\left(\overline{\mathcal{U}}_{a, f}(x)\right)^{2^*-1} \mathcal U(\tau_{-y}(x))-t^{2^*} \mathcal U(\tau_{-y}(x))^{2^*}\right\}\mathrm{~d} V_{\mathbb{B}^{N}}(x)\\
			& \geq \frac{1}{2^*} \int_{d(x,0) \leq 1} a(x) A(r) t^{2}\mathcal U^2(\tau_{-y}(x)) \mathrm{~d} V_{\mathbb{B}^{N}}(x) \\
			&\geq \frac{m^{2} \underline{\mathrm{a}} A(r)}{2^*} \underbrace{\int_{d(x,0) \leq 1}\mathcal U^{2}(\tau_{-y}(x)) \mathrm{~d} V_{\mathbb{B}^{N}}(x)}_{E_1}. \\
		\end{aligned}
	\end{equation*}

	{\bf Estimate of $E_1:$} In the domain $d(x, 0) \leq 1$, we will estimate the value of $E_1$. Triangle inequality allows us to have
	
	$$
	1 - \frac{d(x, 0)}{d(y, 0)} \leq \frac{d(x, y)}{d(y, 0)} \leq  1 +  \frac{d(x, 0)}{d(y, 0)}.
	$$
	Since, $d(x, 0) \leq 1,$ there exist $R > 0$ and $\varepsilon_{R} > 0$ such that whenever $d(y, 0) > R,$ there holds

	$$
	1 - \varepsilon_{R} \leq \frac{d(x, y)}{d(y, 0)} \leq 1 + \varepsilon_{R},
	$$
	where $\varepsilon_R \rightarrow 0$ as $R \rightarrow \infty.$ Thus using the above inequality and \eqref{bubble decay}, we deduce

	\begin{align*}
		E_1 &:= \int_{d(x, 0) \leq 1}\mathcal U^{2}(\tau_{-y}(x)) \mathrm{~d} V_{\mathbb{B}^{N}}(x) \geq C \int_{d(x, 0) \leq 1} 
		e^{-2(c(N, \lambda)) d(x, y)} \;  \mathrm{~d} V_{\mathbb{B}^{N}}(x) \\
		& \geq C \; e^{-2(c(N, \lambda))(1 + \varepsilon_{R}) d(y, 0)} \underbrace{\int_{d(x, 0) \leq 1} \mathrm{~d} V_{\mathbb{B}^{N}}(x)}_{:=C} \\
		& = \overline{C} \; e^{-2(c(N, \lambda))(1 + \varepsilon_{R}) d(y, 0)}.
	\end{align*}
	
	Finally, we obtain
	\begin{equation}
		I I \geq \frac{\tilde{C} m^{2} \underline{\mathrm{a}} A(r)}{2^*} \; e^{-2(1 + \varepsilon_{R}) c(N, \lambda)d(y, 0)}. \label{4.rr}
	\end{equation}
	
	\medskip 
	{\bf Estimate of $I$ :} For $\delta > 2^{*} (c(n, \lambda)) +(N-1)$, let's estimate $I$ as follows:
	\begin{equation}
		\begin{aligned}
			I &=\frac{t^{2^*}}{2^*} \int_{\mathbb{B}^{N}}(1-a(x)) \mathcal U(\tau_{-y}(x))^{2^*} \mathrm{~d} V_{\mathbb{B}^{N}}(x)  \\
			& \leq  C\frac{t^{2^*}}{2^\star} \int_{\mathbb{B}^{N}}e^{-\delta d(x,0)} e^{2^*(c(n, \lambda))(d(x,0)- d(y,0))} \mathrm{~d} V_{\mathbb{B}^{N}}(x) \\
			& \leq  C\frac{t^{2^*}}{2^*} e^{-2^*(c(n, \lambda))d(y,0)} \int_{\mathbb{B}^{N}}e^{-\delta d(x,0)+2^*(c(n, \lambda))d(x,0)} \mathrm{~d} V_{\mathbb{B}^{N}}(x) \\
			& \leq C\frac{t^{2^*}}{2^*} e^{-2^*(c(n, \lambda))d(y,0)} \int_{0}^{\infty}e^{-\delta r+2^*(c(n, \lambda))r + (N-1)r} \mathrm{~d}r \\
			& \leq C\frac{M^{2^*}}{2^*} e^{-2^{\star}(c(n, \lambda))d(y,0)}.\\
		\end{aligned}\label{4p}
	\end{equation}
 \noi
Thus we have obtained the following
	\begin{equation}
		I \leq C\frac{M^{2^*}}{2^*} e^{-2^{\star}(c(n, \lambda))d(y,0)}. \label{4.oo}
	\end{equation}
	Now applying\eqref{4.rr} and \eqref{4.oo}, we can choose $R_{0}>R>0$ large enough such that $2^{\star} > 2(1 + \varepsilon_{R})$
	and hence it follows that

	\begin{equation}\label{anot1}
		(I)<(I I) \text { for }d(y,0) \geq R_{0}.
	\end{equation}
	Consequently, \eqref{4.ab} is proved.
	\end{proof}

\medskip

 
\section{Existence and multiplicity results for the case $a(x)\leq 1$} \label{Thm2sec}
  This section demonstrates that there exist at least three positive solutions of \eqref{Paf} when $a(x) \leq 1.$ We thoroughly examine the energy levels related to Aubin-Talenti bubbles and hyperbolic bubbles using the Adachi-Tanaka method combined with energy estimates. Let $\bar{\mathcal I}_{\lambda,a,f}$ denotes the energy functional as defined in \eqref{4.1}. Recall that the best Sobolev constant
	$$ S_{\lambda}=\inf_{u \in \mathcal{H}, u \ne 0} \frac{ \int_{ \mathbb B^N}\Big[|\nabla_{\mathbb B^N}u|^2-\lambda u^2\Big] \, {\rm d}V_{\mathbb{B}^N}}{\Big(\int_{ \mathbb B^N}|u|^{2^*} \, {\rm d}V_{\mathbb{B}^N}\Big)^\frac{2}{2^*}}.$$	
	
\subsection{Existence of local minimizer}
Define $$\Sigma=\{v\in \mathcal H\,: \norm{v}_\lambda=1\}\,\,\text{ and } \Sigma_+=\{v\in \Sigma\, :\; v_{+} \not\equiv 0\}$$
and $E_{\lambda,a,f}:  \Sigma_+\to\mathbb R$ as 
\begin{equation}
E_{\lambda,a,f}(v):=\sup_{t>0}\bar{\mathcal I}_{\lambda,a,f}(tv).
\end{equation}

\medskip
\noi
Let $v\in \Sigma_+$, we see the following relation between $E_{\lambda,1,0}(v)$ and $\bar{\mathcal I}_{\lambda,a,0}(v)$.		
	\begin{align} \label{expval}
		E_{\lambda,a,0}(v)&=\sup_{t>0}\bar{\mathcal I}_{\lambda,a,0}(tv)=\sup_{t>0}\(\frac{1}{2}t^2\norm{v}^2_\la- \frac{1}{2^*}t^{2^*}\int_{\mathbb B^N}a(x)v_+^{2^*}\mathrm{~d} V_{\bn}(x)\)\\ \notag
		&=\bar{\mathcal I}_{\lambda,a,0}\(\(\int_{\mathbb B^N}a(x)v_+^{2^*}\mathrm{~d} V_{\bn}(x)\)^\frac{1}{2-2^*}v\)\\ \notag
		&=\(\frac{1}{2}-\frac{1}{2^*}\)\(\int_{\mathbb B^N}a(x)v_+^{2^*}\mathrm{~d} V_{\bn}(x)\)^{-\frac{2}{2^*-2}}.
	\end{align}
	Since $\frac{2}{2^*-2}>0$, for all $v\in \Sigma_+$, we obtain the following relation
	\begin{equation} \bar{a}^{-\frac{2}{2^*-2}}E_{\lambda,1,0}(v)=E_{\lambda,\bar{a},0}(v)\leq E_{\lambda,a,0}(v)\leq E_{\lambda,\underline{a},0}(v)=\underline{a}^{-\frac{2}{2^*-2}}E_{\lambda,1,0}(v).\label{EI_Rel}\end{equation}
	On the other hand, 
\begin{equation*}
	\sup_{t>0} \mathcal{I}_{\lambda,1,0}(t\mathcal U)=\sup_{t>0} \bar{\mathcal I}_{\lambda,1,0}(t\mathcal U)=  \bar{\mathcal I}_{\lambda,1,0}(\mathcal U).
\end{equation*}	
Therefore
\begin{align}\label{Re_IEaa}
\bar{a}^{-\frac{2}{2^*-2}}\bar{\mathcal I}_{\lambda,1,0}(\mathcal U)  \leq \inf_{v\in \Sigma_+}E_{\lambda,a,0}(v) \leq \underline{a}^{-\frac{2}{2^*-2}}\bar{\mathcal I}_{\lambda,1,0}(\mathcal U).
\end{align}	
	
Now we state some auxiliary lemmas whose proofs are a straightforward adaptation of proofs given in \cite{Adachi}; also, see \cite{GGS} for the hyperbolic space.

\begin{lemma}\label{Lem_reIE}
	(i)	For $v\in \mathcal H$ and $\eps\in(0,1)$, the following inequality holds
		\begin{equation}
(1-\eps)\bar{\mathcal I}_{\lambda,\frac{a}{1-\eps},0}(v)-\frac{1}{2\epsilon}\norm{f}_{\mathcal{H}'}^2\leq \bar{\mathcal I}_{\lambda,a,f}(v)\leq (1+\eps)\bar{\mathcal I}_{\lambda,\frac{a}{1+\eps},0}(v)+\frac{1}{2\epsilon}\norm{f}_{\mathcal{H}'}^2.
		\end{equation}
		(ii) For $v\in \Sigma_+$ and $\eps\in(0,1)$, we obtain
		\begin{equation}
(1-\eps)^{\frac{N}{2}}E_{\lambda,a,0}(v)-\frac{1}{2\epsilon}\norm{f}_{\mathcal{H}'}^2\leq E_{\lambda,a,f}(v)\leq (1+\eps)^{\frac{N}{2}}E_{\lambda,a,0}(v)+\frac{1}{2\epsilon}\norm{f}_{\mathcal{H}'}^2.
		\end{equation} 
		(iii) In particular, there exists a constant $d_0>0$ such that if $\norm{f}_{\mathcal{H}'}<d_0$, then 
		\begin{equation}
			\inf_{v\in \Sigma_+}E_{\lambda,a,f}(v)>0.
		\end{equation}
			\end{lemma}
			
\medskip			
\noi		
For $v\in \Sigma_+$, define a new function ${g_v}: [0,\I)\to \mathbb R$ as
$$ {g_v}(t):=\bar{\mathcal I}_{\lambda,a,f}(tv).$$	
\begin{lemma}\label{Lem_gv}
	\begin{itemize}
		\item[(i)] For every $v\in \Sigma_+$, the function ${g_v}$ has at most  two critical points.
		\item[(ii)] Suppose $\norm{f}_{\mathcal{H}'}<d_0$. For any  $v\in \Sigma_+$, there exists a unique $t_{a,f}(v)>0$ such that $\bar{\mathcal I}_{\lambda,a,f}(t_{a,f}(v)v)=E_{\lambda,a,f}(v)$. Furthermore, $t_{a,f}(v)$ satisfies
		\begin{equation}\label{dI^2_neg}
 t_{a,f}(v)>  \Big((2^*-1)\int_{\mathbb B^N}a(x)v_+^{2^*}\mathrm{~d} V_{\bn}       \Big)^\frac{-1}{2^*-2}\text{ and }\,\, \bar{\mathcal I}''_{\lambda,a,f}\Big(t_{a,f}(v)v\Big)[v,v]<0.
		\end{equation}
		\item[(iii)] Any critical point of $g_v$ different from $t_{a,f}(v)$ lies in the interval $\Big[0, \, \frac{N+2}{4}\norm{f}_{\mathcal{H}'}\Big].$
	\end{itemize}
\end{lemma}

The following proposition provides the first critical point for $\bar{\mathcal I}_{\lambda, a,f}$. 
	
	\begin{proposition}\label{Pro_Locmin}
            There exists $r_1>0$ and $d_1\in(0,d_0]$ such that 
            \begin{itemize}
            	\item[(i)] $\bar{\mathcal I}_{\lambda,a,f}$ is strictly convex in $B_{r_1}(0):=\{u\in\mathcal H:\, \norm{u}_\lambda<r_1\}$.
            	\item[(ii)] Suppose $\norm{f}_{\mathcal{H}'}\leq d_1$. Then 
            		$$  \inf_{  \norm{u}_{\lambda}=r_1}\bar{\mathcal I}_{\lambda,a,f}(u)>0 .$$
            \end{itemize}	
            Moreover, $\bar{\mathcal I}_{\lambda,a,f}$ has a unique critical point $\mathcal V_{a,f}$ in $B_{r_1}(0)$ satisfying
            $$ \bar{\mathcal I}_{\lambda,a,f}(\mathcal V_{a,f})=\inf_{u\in B_{r_1}(0) }\bar{\mathcal I}_{\lambda,a,f}(u).$$ 	
	\end{proposition}
\begin{proof}
	$(i)$ A direct computation gives \begin{align*}
\bar{\mathcal I}''_{\lambda,a,f}(u)[v,v]&=\norm{v}^2_\la-(2^*-1) \int_{\mathbb B^N}a(x)u_+^{2^*-2}v^2\mathrm{~d} V_{\bn}(x). 	\end{align*}
Now using H\"older and Sobolev inequality, and the fact that $a(x)\leq 1$, we obtain
 \begin{align*}
 	\int_{\mathbb B^N}a(x)u_+^{2^*-2}v^2\mathrm{~d} V_{\bn}(x) &\leq \norm{u}_{2^*}^{2^*-2}\norm{v}_{2^*}^2\leq S_\lambda^{-\frac{2^*}{2}}\norm{u}_\lambda^{2^*-2}\norm{v}_\lambda^{2}.
 \end{align*}
 Thus
 \begin{align*}
 	\bar{\mathcal I}''_{\lambda,a,f}(u)[v,v]&\geq \norm{v}^2_\la-(2^*-1) S_\lambda^{-\frac{2^*}{2}}\norm{u}_\lambda^{2^*-2}\norm{v}_\lambda^{2}=
 	\Big(1-(2^*-1) S_\lambda^{-\frac{2^*}{2}}\norm{u}_\lambda^{2^*-2} \Big) \norm{v}^2_\la .	\end{align*} 
 	Set $r_1=\Big((2^*-1) S_\lambda^{-\frac{2^*}{2}}\Big)^{-\frac{1}{2^*-2}}$. Thus $\bar{\mathcal I}_{\lambda,a,f}$ is strictly convex in $B_{r_1}(0)$. \\
 	\noi
 	$(ii)$ Suppose $\norm{u}_{\lambda}=r_1$ then, we have 
 	 \begin{align*}
 		\bar{\mathcal I}_{\lambda,a,f}(u)&=\frac{1}{2}\norm{u}^2_\la- \frac{1}{2^*}\int_{\mathbb B^N}a(x)u_+^{2^*}\mathrm{~d} V_{\bn}(x)- \int_{\mathbb B^N}f(x)u\, \mathrm{~d} V_{\bn}(x)
 		 \\
 		&\geq \frac{1}{2}r_1^2-\frac{r_1^{2^*}}{2^*}S_\lambda^{-\frac{2^*}{2}}-r_1\norm{f}_{\mathcal{H}'}\\
 		&=\Big(\frac{1}{2} -\frac{r_1^{2^*-2}}{2^*}S_\lambda^{-\frac{2^*}{2}} \Big)r_1^2-r_1\norm{f}_{\mathcal{H}'}=\Big(\frac{1}{2} -\frac{ 1}{2^*(2^*-1)}  \Big)r_1^2-r_1\norm{f}_{\mathcal{H}'}.
 			\end{align*}
 			Thus there exists $d_1\in(0,d_0]$ such that  
 			\begin{align*}
 			\inf_{  \norm{u}_{\lambda}=r_1}	\bar{\mathcal I}_{\lambda,a,f}(u)>0=	\bar{\mathcal I}_{\lambda,a,f}(0) \quad \text{ for } \norm{f}_{\mathcal{H}'}\leq d_1.\end{align*}
 			
 			Now using the above information and the fact that	$\bar{\mathcal I}_{\lambda,a,f}$ is strictly convex, we obtain   a unique critical point $\mathcal V_{a,f}$ of $\bar{\mathcal I}_{\lambda,a,f}$  in $B_{r_1}(0)$. Furthermore, it satisfies
 			  \be\label{1st sol energy} \bar{\mathcal I}_{\lambda,a,f}(\mathcal V_{a,f})=\inf_{u\in B_{r_1}(0) }\bar{\mathcal I}_{\lambda,a,f}(u)\leq\bar{\mathcal I}_{\lambda,a,f}(0)=0.\ee 
 			  This completes the proof.
\end{proof}

\medskip 

Now, the proposition that follows characterizes all the possible critical points of the functional $\bar{\mathcal I}_{\lambda, a,f}$ in terms of the functional $E_{\lambda, a, f}.$
	
\begin{proposition}\label{Relcrit_I&E}
Assume $\norm{f}_{\mathcal{H}'}\leq d_2$, where $d_2=\min\{d_1,\frac{4r_1}{N+2}\}>0$ and $d_1, r_1$ as in Proposition \ref{Pro_Locmin}. Then 
	\begin{itemize}
		\item[(i)] $E_{\lambda,a,f}\in C^1\left( \Sigma_+,\mathbb R\right)$ and 
		\begin{equation} 
E'_{\lambda,a,f}(v)[h]=t_{a,f}(v)\bar{\mathcal I}'_{\lambda,a,f}(t_{a,f}(v)v)[h] \label{der}
		\end{equation}
		for all $h\in T_v\Sigma_+=\{h\in \mathcal H: \langle v,h\rangle_\lambda=0\}$.
			\item[(ii)] $v\in \Sigma_+$ is a critical point of $E_{\lambda,a,f}$ if and only if $t_{a,f}(v)v\in \mathcal H$ is a critical point of $\bar{\mathcal I}_{\lambda,a,f}$.
			\item[(iii)] Additionally, the set of critical points of $\bar{\mathcal I}_{\lambda,a,f}$ can be characterized as  
			$$\left\{  t_{a,f}(v)v\,:\, v\in \Sigma_+,\,\, E'_{\lambda,a,f}(v)=0   \right\}\cup \left\{ \mathcal{V}_{a,f}\right\} .$$ 
	\end{itemize}
\end{proposition}

\medskip
	
\subsection{PS-decomposition of $E_{\lambda,a,f}$}
	
Now we look at the Palais-Smale decomposition of $E_{\la,a,f}.$


\begin{proposition}\label{ProPS_ELaf}
	Let us assume that $\|f\|_{\mathcal H' } \leq d_{2}$ for $d_{2}>0$ as given in Proposition \ref{Relcrit_I&E}. Then,
	\begin{enumerate}
	\item As the $\operatorname{dist}_\mathcal{H}\left(v_, \partial \Sigma_{+}\right)=\inf \left\{\left\|v-u\right\|_{\la}: u \in \Sigma, u_{+} \equiv 0\right\} {\xrightarrow{n}}\;\;0$,\\ $E_{\la,a, f}\left(v\right) \rightarrow \infty$. 
	\medskip
\item Let $(v_{n})_n \subset \Sigma_{+}$ is such that the following holds as $n \rightarrow \infty$
	\end{enumerate}
	\begin{equation}
		\begin{aligned}
			&E_{\la,a, f}\left(v_{n}\right) \rightarrow c \text { for some } c>0, \\
			&\left\|E_{\la,a, f}^{\prime}\left(v_{n}\right)\right\|_{T_{v_{n}}^{*} \Sigma_{+}} \equiv \sup \left\{E_{\la,a, f}^{\prime}\left(v_{n}\right) h;\; h \in T_{v_{n}} \Sigma_{+},\;\|h\|_{\lambda}=1\right\} \rightarrow 0.
		\end{aligned}
		\label{4.kk}
	\end{equation}
Then there exists a subsequence, still denoted by $(v_{n})_n$, a critical point $u_{0}(x) \in \mathcal H$ of $\mathcal{I}_{\la, a, f}(u)$, integers $n_1,n_2\in \N\cup\{0\}$ and functions $u^j_n, v^k_n\in \mathcal{H}$, $y^k\in\bn$ for all $1\leq j\leq n_1$ and $1\leq k\leq n_2$ such that  up to a subsequence (still denoted by the same index),
 $$\left\| v_{n}(x)-\frac{u_{0}(x)+\sum_{j=1}^{n_1}u_n^j+\sum_{k=1}^{n_2}v_n^k}{\left\|u_{0}(x)+\sum_{j=1}^{n_1}u_n^j+\sum_{k=1}^{n_2}v_n^k\right\|_\la}\right\|_\la \rightarrow 0 \text { as } n \rightarrow \infty,$$
 where $u_n^j$ and $v_n^k$ are Palais-Smale sequence of the form \eqref{seq-un} and \eqref{seq-vn}, respectively and $o(1)\to 0$ in $\mathcal{H}$. Moreover, we have the following
	$$E_{\la,a,f}(v_n)=\mathcal I_{\lambda,a,f}(u_0)+ \sum_{j=1}^{n_1}\mathcal I_{\lambda,1,0}(\mathcal{U}^j)+\sum_{k=1}^{n_2} a(y^k)^{-\frac{N-2}{2}}J(V^k)\text { as } n \rightarrow \infty,$$
where $\mathcal{U}^j$ and $V^k$ are the solutions of \eqref{E} (without the sign condition) and \eqref{Talenti}, respectively corresponding to $u^j_n$ and $v_n^k$. Furthermore, we have 
\begin{enumerate}
	\item $T^j_n \circ T^{-i}_n(0)\to \I $ for $i\ne j$,
	\item $\Big| \log\big(\frac{\eps^k_{n} }{\eps^l_{n}}\big)\Big|+ \Big| \frac{y_n^k-y^l_n}{\eps^l_{n}}\Big|\to \I$ as $n\to \I$ for $k\ne l$.
\end{enumerate}
\end{proposition}

\medskip

\begin{proof}
	Using Lemma \ref{Lem_reIE}(ii) and \eqref{expval}, we get
	\begin{equation*}
		\begin{aligned}
			E_{\la,a, f}\left(v_{n}\right) & \geq(1-\varepsilon)^{\frac{2^*}{2^*-2}} E_{\la,a, 0}\left(v_n\right)-\frac{1}{2 \varepsilon}\|f\|_{\mathcal{H'}}^{2} \\
			& \geq(1-\varepsilon)^{\frac{2^*}{2^*-2}}\left(\frac{1}{2}-\frac{1}{2^*}\right)\left(\int_{\mathbb{B}^{N}} a(x) v_{n +}^{2^*} \mathrm{~d} V_{\mathbb{B}^{N}}\right)^{-\frac{2}{2^*-2}}-\frac{1}{2 \varepsilon}\|f\|_{\mathcal H'}^{2}.
		\end{aligned}
	\end{equation*}
As $\operatorname{dist}\left(v_{n}, \partial \Sigma_{+}\right) \rightarrow 0$ suggests $(v_n)_+ \rightarrow 0$ \mbox{in}\ $\mathcal{H}.$ Therefore using Poincar\'e-Sobolev inequality, it follows $(v_n)_+ \rightarrow 0$ in $L^{2^{\star}}(\bn).$ Thus we conclude 
$$\left|\int_{\mathbb{B}^{N}} a(x)(v_n)_+^{2^*} \mathrm{~d} V_{\mathbb{B}^{N}}\right| \leq\left.\|a\|\right._{\infty} \int_{\mathbb{B}^{N}} {(v_n)}_{+}^{2^*}\mathrm{~d} V_{\mathbb{B}^{N}} \xrightarrow{n} 0.
$$
As a result, $E_{\la,a, f}\left(v_{n}\right) \rightarrow \infty$ as dist$_{\mathcal H}\left(v_{n}, \partial \Sigma_{+}\right) \rightarrow 0.$ This completes the proof of part $(1).$

\medskip

 To prove $(2),$ we utilize \eqref{der} and \eqref{dI^2_neg} as follows
	\begin{equation*}
		\begin{aligned}
			\left\|\bar{\mathcal I}_{\la,a, f}^{\prime}\left(t_{a, f}\left(v_{n}\right) v_{n}\right)\right\|_{\mathcal{H'}} &=\frac{1}{t_{a, f}\left(v_{n}\right)}\left\|E_{\la,a, f}^{\prime}\left(v_{n}\right)\right\|_{T_{v_{n}}^{*} \Sigma_{+}}
			\\ & \leq\left((2^*-1) S_{\la}^{-\frac{2^*}{2}}\right)^{\frac{1}{2^*-2}}\left\|E_{\la,a, f}^{\prime}\left(v_{n}\right)\right\|_{T_{v_{n}}^* \Sigma_{+}} \stackrel{n}{\rightarrow} 0.\\
		\end{aligned}
	\end{equation*}
Further, we also know that $\bar{\mathcal I}_{\lambda,a,f}\left(t_{a,f}(v_{n})v_{n}\right)=E_{\la,a, f}\left(v_{n}\right) \rightarrow c$ as $n \rightarrow \infty$. Applying Palais-Smale lemma for $\bar{\mathcal I}_{\lambda,a,f}(u)$(Theorem \ref{ps_decom}), we get the desired result.
	\end{proof}
	 \medskip

  \begin{corollary}
      Assume $\|f\|_{\mathcal H'} \leq d_{2}$.  Then $E_{\la, a, f}(v)$ satisfies the condition $(\mathrm{PS})_{c}$ for $c< \mathcal I_{\lambda,a,f}\left(\mathcal{V}_{a, f} (x)\right)+  \mathcal I_{\la,1,0}(\mathcal U), $ where $\mathcal U$ is the unique radial solution of \eqref{E}, and $\mathcal{V}_{a, f} $ is the critical point of 
	$\mathcal I_{\la, a, f}$  found in Proposition~\ref{Pro_Locmin}.
	\label{cor1.8}
  \end{corollary}
  \begin{proof}
       By Proposition \ref{ProPS_ELaf}, we observe that the (PS)$_c$ breaks down only for
       $$c=\mathcal I_{\lambda,a,f}(u_0)+ n_1\; \mathcal I_{\lambda,1,0}(\mathcal U)+\sum_{k=1}^{n_2} a(y^k)^{-\frac{N-2}{2}}J(V^k),$$
where $u_0$ is a critical point of $\mathcal I_{\la, a, f}$.
       Since $\mathcal{V}_{a, f} $ is a least energy critical point of $\mathcal I_{\la, a, f}$, the lowest level of breaking down of PS is
       $$ \min\left\{ \mathcal I_{\la, a, f}(\mathcal{V}_{a, f})+\mathcal I_{\la,1,0}(\mathcal U),\,\,\,\mathcal I_{\la, a, f}(\mathcal{V}_{a, f})+ a(y^1)^{-\frac{N-2}{2}}J(V^1)   \right\}.$$
 \noi      
We know that, for all $x\in \mathbb B^N$, $$\frac{J(V^1)}{\mathcal I_{\lambda,1,0}(\mathcal U)} = \Big(\frac{S}{S_\lambda}\Big)^\frac{N}{2}>1\geq a(x)^{\frac{N-2}{2}}\implies \mathcal I_{\lambda,1,0}(\mathcal U)< a(x)^{-\frac{N-2}{2}}J(V^1). $$ 
       Thus $$ \mathcal I_{\lambda,1,0}(\mathcal U)< a(y^1)^{-\frac{N-2}{2}}J(V^1). $$
        This concludes the proof.
  \end{proof}
  
  \medskip
  
	\begin{lemma}\label{lem4.dd}	
  Assume that \eqref{assum_A}, $\eqref{A2}$ and \eqref{A_cond} hold. Then  
	\begin{enumerate}
		\item $\inf _{v \in \Sigma_{+}} E_{\la,a, 0}(v)=\mathcal I_{\la,1,0}(\mathcal U)$.
		\item $\inf _{v \in \Sigma_{+}} E_{\la,a, 0}(v)$ is not attained.
		\item  $E_{\la,a, 0}(v)$ satisfies $(P S)_{c}$\, for\, $c \in\left(-\infty,\,\,\mathcal I_{\la,1,0}(\mathcal U)\right)$.  
	\end{enumerate}
\end{lemma}
 \begin{proof}
The proof of the above lemma follows in the same line as of \cite[Lemma~7.2]{GGS} with minor modifications for the case $p = 2^{\star}.$ Note that the assumption $a \not\equiv 1$ is crucial in this proof. We skip the proof for brevity. 
\end{proof}
  
	\begin{lemma} \label{centofm}
	Let $a(.)$ be as in Theorem \ref{Thm-A2}. Then there exists a constant $\delta_{0}>0$ such that  if $E_{\la,a, 0}(v) \leq \mathcal I_{\la,1,0}(\mathcal U)+\delta_{0}$, then
	\begin{equation}
		\int_{\mathbb{B}^{N}} \frac{x}{m(x)}|v(x)|^{2^*} \mathrm{~d} V_{\mathbb{B}^{N}}(x) \neq 0,
	\end{equation}
	where $m(x) >0$ is defined such that $d(\frac{x}{m},0)= \frac{1}{2}$, i.e., $m(x)= \frac{|x|}{\tanh \left(\frac{1}{4}\right)}.$ 
\end{lemma} 

\begin{proof}
	We will argue by contradiction. Let us assume that there exists a sequence $\left\{v_{n}\right\} \subset \Sigma_{+}$ such that the following holds
	\begin{center}
		$E_{\lambda,a, 0}\left(v_{n}\right) \leq \mathcal I_{\la,1,0}(\mathcal U)+\frac{1}{n}$ and $\int_{\mathbb{B}^{N}} \frac{x}{m}|v_{n}(x)|^{2^*} \mathrm{~d} V_{\mathbb{B}^{N}}(x) \stackrel{n \rightarrow \infty}{\longrightarrow} 0$.    
	\end{center}
 \noi
 Then, according to Ekeland's variational principle, there exists a
 $\tilde{v}_{n} \subset \Sigma_{+}$ such that
	\begin{equation*}
		\begin{aligned}
			&\left\|v_{n}-\tilde{v}_{n}\right\|_{\la}  \stackrel{n \rightarrow \infty}{\longrightarrow} 0, \\
			E_{\lambda,a, 0}\left(\tilde{v}_{n}\right) &\leq E_{\lambda,a, 0}\left(v_{n}\right)\leq \mathcal I_{\la,1,0}(\mathcal U)+\frac{1}{n}, \\
			E_{\lambda,a, 0}^{\prime}\left(\tilde{v}_{n}\right) &\stackrel{n \rightarrow \infty}{\longrightarrow} 0 \text { in } \mathcal{H'}.
		\end{aligned}
	\end{equation*} 
 \noi
 The above suggests that $(\tilde{v}_{n})_n$ is a Palais Smale sequence for $E_{\lambda,a, 0}$ at the level $\mathcal I_{\la,1,0}(\mathcal U)$. Thus we have found a PS which is very much close to $v_n$. Then we can apply the PS decomposition theorem to $E_{\lambda, a,0}$. Now from the growth restriction condition, i.e., $E_{\lambda, a, 0}(\tilde{v}_{n}) \leq \mathcal I_{\la,1,0}(\mathcal U)+\frac{1}{n}$, and using the fact $ \mathcal I_{\lambda,1,0}(\mathcal U)< a(y^k)^{-\frac{N-2}{2}}J(V^k)$ for any $k$, we will only have at most one hyperbolic bubble in the decomposition. The rest of the proof follows using Proposition~\ref{ProPS_ELaf} and \cite[Lemma 7.3]{GGS}.
\end{proof}

\medskip

Now we will use Lusternik-Schnirelmn (L-S) category theory to find the second and third positive solutions to  \eqref{Paf}. The (L-S) category of $A$ concerning $M$ is denoted by cat $(A, M)$. 
Particularly, cat $(M)$ denotes cat $(M, M)$. In addition, we shall use the following notation: 
\begin{equation*}
\left[E_{\la,a, f} \leq c\right]=\left\{u \in \Sigma_{+} \,: \, E_{\la,a, f}(u) \leq c\right\}
\end{equation*}
for $c \in \mathbb{R}$.
For fundamental properties of the L-S category, we refer to Ambrosetti \cite{Ambro}. Here, we use the following property \cite{Ambro} (also
see \cite[Proposition 2.4]{Adachi}) which will play a pivotal role to obtain the second and third solutions of \eqref{Paf}.
\begin{proposition}\cite{Ambro}\label{multi-cat}
	Suppose $M$  is a Hilbert manifold and  $\Psi \in C^{1}(M, \mathbb{R})$. Assume that for $c_{0} \in \mathbb{R}$ and $k \in \mathbb{N}$
	
	\begin{enumerate}
		\item $\Psi(x)$ satisfies $(P S)_{c}$ for $c \leq c_{0}.$
		\item $\operatorname{cat}\left(\left\{x \in M: \Psi(x) \leq c_{0}\right\}\right) \geq k.$
	\end{enumerate}
	Then $\Psi(x)$ has at least $k$ critical points in $\left\{x \in M: \Psi(x) \leq c_{0}\right\}$.
	\label{propcrit}
\end{proposition}

We recall another very well-known lemma concerning the (L-S) category theory. 

\begin{lemma}(\cite{Adachi}, Lemma 2.5)\label{lemcat}
	Let $N \geq 1$ and $M$ be a topological space. Assume that there exist two continuous mappings
	\begin{equation*}
		F: S^{N-1}_{\mathbb{B}^{N}}\left(:=\left\{x \in \mathbb{B}^{N}:d(x,0)=1\right\}\right) \rightarrow M, \quad G: M \rightarrow S^{N-1}_{\mathbb{B}^{N}}
	\end{equation*}
	such that $G \circ F$ is homotopic to the identity map Id: $S^{N-1}_{\mathbb{B}^{N}} \rightarrow S^{N-1}_{\mathbb{B}^{N}}$, i.e,  there is a continuous $\operatorname{map} \eta:[0,1] \times S^{N-1}_{\mathbb{B}^{N}} \rightarrow S^{N-1}_{\mathbb{B}^{N}}$ such that
	\begin{equation*}
		\begin{aligned}
			&\eta(0, x)=(G \circ F)(x) \quad \forall\, x \in S^{N-1}_{\mathbb{B}^{N}}, \\
			&\eta(1, x)=x \quad \forall\, x \in S^{N-1}_{\mathbb{B}^{N}}.
		\end{aligned}
	\end{equation*}
	Then cat $(M) \geq 2$.
\end{lemma}
\noi
Keeping in mind the aforementioned lemma, our next objective will be to create two mappings:
\begin{equation*}
	\begin{aligned}
		&F: S^{N-1}_{\mathbb{B}^{N}} \rightarrow\left[E_{\la,a, f} \leq \bar{\mathcal I}_{\lambda,a,f}\left(\mathcal{V}_{a, f}\right)+\mathcal I_{\la,1,0}(\mathcal U)-\varepsilon\right], \\
		&G:\left[E_{\la,a, f} \leq \bar{\mathcal I}_{\lambda,a,f}\left(\mathcal{V}_{a, f}\right)+\mathcal I_{\la,1,0}(\mathcal U)-\varepsilon\right] \rightarrow S^{N-1}_{\mathbb{B}^{N}},
	\end{aligned}
\end{equation*}
such that $G \circ F$ is homotopic to the identity map.

\medskip

To this end, let us define $F_{R}: S^{N-1}_{\mathbb{B}^{N}} \rightarrow \Sigma_{+}$ as follows:
For $d(y,0) \geq R_{0}$, where $R_0$ is as determined by Proposition \ref{energy-prop}, we will find $s=s(f, y)$ such that


\begin{equation}
	\left\|\mathcal{V}_{a, f}+s \mathcal U(\tau_{-y}(.))\right\|_{\la}=t_{a, f}\left(\frac{\mathcal{V}_{a, f}+s \mathcal U(\tau_{-y}(.))}{\left\|\mathcal{V}_{a, f}+s \mathcal U(\tau_{-y}(.))\right\|_{\la}}\right). \label{4.ll}
\end{equation}

\begin{proposition}(\cite{Adachi}, Proposition 2.6)
	Let $a(.)$ be as described in Theorem \ref{Thm-A2}. Then there exist $d_{3} \in\left(0, d_{2}\right]$ and $R_{1}>R_{0}$ such that for any $\|f\|_{\mathcal{H'}} \leq d_{3}$ and any $d(y,0) \geq R_{1}$, there exists a unique $s=s(f, y)>0$ in a neighbourhood of 1, satisfying \eqref{4.ll}. Additionally,
	\begin{equation*}
		\left\{y \in \mathbb{B}^{N}:d(y,0)\geq R_{1}\right\} \rightarrow(0, \infty) ; \quad y \mapsto s(f, y)
	\end{equation*}
	is continuous.
\end{proposition}
\noi
	We are now going to define a function $F_{R}: S^{N-1}_{\mathbb{B}^{N}} \rightarrow \Sigma_{+}$ as follows:
\begin{equation*}
	F_{R}(y)=\frac{\mathcal{V}_{a, f}+s(f, \frac{\tanh (\frac{R}{2})}{\tanh\frac{1}{2}} y) \mathcal U(\tau_{-\frac{\tanh (\frac{R}{2})}{\tanh\frac{1}{2}}  y}(x))}{\left\|{\mathcal{V}_{a, f}+s(f, \frac{\tanh (\frac{R}{2})}{\tanh\frac{1}{2}} y) \mathcal U(\tau_{-\frac{\tanh (\frac{R}{2})}{\tanh\frac{1}{2}}  y}(x))}\right\|_{\la}}
\end{equation*}
for $\|f\|_{\mathcal{H'}} \leq d_{3}$ and $R \geq R_{1}$.\\
\noi
Then we have,
\begin{proposition}
	For $0<\|f\|_{\mathcal{H'}} \leq d_{3}$ and $R \geq R_{1}$, there exists $\varepsilon_{0} = \varepsilon_{0}(R)>0$ such that
	\begin{equation*}
		F_{R}\left(S^{N-1}_{\mathbb{B}^{N}}\right) \subseteq\left[E_{\la,a, f} \leq \bar{\mathcal I}_{\lambda,a,f}\left(\mathcal{V}_{a, f}\right)+\mathcal I_{\la,1,0}(\mathcal U)-\varepsilon_{0}\right].
	\end{equation*}
\end{proposition}
\begin{proof}
	The following statement results from the construction of  $F_{R}$
	\begin{equation*}
		F_{R}\left(S^{N-1}_{\mathbb{B}^{N}}\right) \subseteq\left[E_{\la,a, f}<\bar{\mathcal I}_{\lambda,a,f}\left(\mathcal{V}_{a, f}\right)+\mathcal I_{\la,1,0}(\mathcal U)\right].
	\end{equation*}
	As a result, the proposition follows since $F\left(S^{N-1}_{\mathbb{B}^{N}}\right)$ is compact.
\end{proof} 
\noi
Consequently, we construct a mapping
\begin{equation*}
	F_{R}: S^{N-1}_{\mathbb{B}^{N}} \rightarrow\left[E_{\la,a, f} \leq \bar{\mathcal I}_{\lambda,a,f}\left(\mathcal{V}_{a, f}\right)+\mathcal I_{\la,1,0}(\mathcal U)-\varepsilon_{0}(R)\right].
\end{equation*}
The following lemma is vital in establishing the mapping $G$.
\begin{lemma}
	There exists $d_{4} \in\left(0, d_{3}\right]$ such that if $\|f\|_{\mathcal{H'}} \leq d_{4}$, then
	\begin{equation}
		\left[E_{\la,a, f}<\bar{\mathcal I}_{\lambda,a,f}\left(\mathcal{V}_{a, f}\right)+\mathcal I_{\la,1,0}(\mathcal U)\right] \subseteq\left[E_{\la,a, 0}<\mathcal I_{\la,1,0}(\mathcal U)+\delta_{0}\right] \label{4.vv}
	\end{equation}
	where $\delta_{0}>0$ is as obtained in Lemma \ref{centofm}.
	\label{lemG}
\end{lemma}
\begin{proof}
The proof of the above lemma can be concluded as in \cite[Lemma~7.9]{GGS}.
\end{proof}

With all these technical results in hand, we now  define the function $G$ as follows: 
$$G:\left[E_{\la,a, f}<{\bar{\mathcal I}}_{\lambda,a,f}\left(\mathcal{V}_{a, f}\right)+\mathcal I_{\la,1,0}(\mathcal U)\right] \rightarrow S^{N-1}_{\mathbb{B}^{N}}$$
$$G(v):= \tanh(\frac{1}{2})\frac{\int_{\mathbb{B}^{N}} \frac{x}{m}|v|^{2^*} \mathrm{~d} V_{\mathbb{B}^{N}}(x)}{\left|\int_{\mathbb{B}^{N}} \frac{x}{m} |v|^{2^*} \mathrm{~d} V_{\mathbb{B}^{N}}(x)\right|}
$$
where $m$ is as defined in Lemma \ref{centofm}. Further, Lemma \ref{centofm} and Lemma \ref{lemG} ensure that the above function is well defined.
Furthermore, in the following proposition, we will show that these functions, $F$ and $G$, fulfil our purpose.
\begin{proposition}\label{propiden}
	For a sufficiently large $R \geq R_{1}$ and for sufficiently small $\|f\|_{\mathcal{H'}}>0$, we have,
	\begin{equation*}
		G \circ F_{R}: S^{N-1}_{\mathbb{B}^{N}} \rightarrow S^{N-1}_{\mathbb{B}^{N}}
	\end{equation*}
	is homotopic to the identity.
\end{proposition}
\begin{proof}
Let $\gamma(\theta),\;\theta \in\left[\theta_1, \theta_2\right]$ be a regular geodesic joining $G\left(\frac{\mathcal U(\tau_{-\frac{\tanh (\frac{R}{2})}{\tanh\frac{1}{2}}  y}(\cdot))}{\left\|\mathcal U(\tau_{-\frac{\tanh (\frac{R}{2})}{\tanh\frac{1}{2}}  y}(\cdot))\right\|_\lambda}\right)$ and $G\left(F_R(y)\right)$ on $S^{N-1}_{\mathbb{B}^{N}}$, where $\gamma\left(\theta_1\right)= G\left(\frac{\mathcal U(\tau_{-\frac{\tanh (\frac{R}{2})}{\tanh\frac{1}{2}}  y}(\cdot))}{\left\|\mathcal U(\tau_{-\frac{\tanh (\frac{R}{2})}{\tanh\frac{1}{2}}  y}(\cdot))\right\|_\lambda}\right), \gamma\left(\theta_2\right)=G\left(F_R(y)\right)$.
Set
\begin{equation*}
\bar{\gamma}(\theta)=\gamma\left(2\left(\theta_1-\theta_2\right) \theta+\theta_2\right), \quad \theta \in\left[0, \frac{1}{2}\right] .
\end{equation*}
Define a mapping
\begin{equation*}
\zeta(\theta, y):[0,1] \times  S^{N-1}_{\mathbb{B}^{N}} \rightarrow  S^{N-1}_{\mathbb{B}^{N}}
\end{equation*}
by
\begin{equation*}
\zeta(\theta, y)= \begin{cases}\bar{\gamma}(\theta), & \theta \in\left[0, \frac{1}{2}\right) \\ 
G\left(\frac{\mathcal U(\tau_{-\frac{\tanh (\frac{R}{4 \theta})}{\tanh\frac{1}{2}}  y}(\cdot))}{\left\|\mathcal U(\tau_{-\frac{\tanh (\frac{R}{4 \theta})}{\tanh\frac{1}{2}}  y}(\cdot))\right\|_\lambda}\right) , & \theta \in\left[\frac{1}{2}, 1\right) \\ y, & \theta=1 .\end{cases}
\end{equation*}
From now on, we shall denote $\mathcal U(\tau_{-k}(x))$ by $\mathcal U_{k}(x)$.
First, we show that $\zeta(\theta, y)$ is well defined. Indeed for $d(k,0)$ large,
\begin{equation*}
E_{\lambda,a,0}\left(\frac{\mathcal U_{k}}{\left\|\mathcal U_{k}\right\|_\lambda}\right) \leq I_{\lambda,1,0}(\mathcal U)+\delta_0,
\end{equation*}
where $\delta_0$ is as in Lemma \ref{centofm}.
Therefore, by \eqref{4.vv}, $\zeta(\theta, y)$ is well defined.

\noi
Next, we prove that $\zeta(\theta, y) \in C\left([0,1] \times S^{N-1}_{\mathbb{B}^{N}}, S^{N-1}_{\mathbb{B}^{N}}\right)$. It can be easily seen that
\begin{equation*}
\lim _{\theta \rightarrow \frac{1}{2}} \bar{\gamma}(\theta)=\gamma\left(\theta_1\right)=G\left(\frac{\mathcal U(\tau_{-\frac{\tanh (\frac{R}{2})}{\tanh\frac{1}{2}}  y}(\cdot))}{\left\|\mathcal U(\tau_{-\frac{\tanh (\frac{R}{2})}{\tanh\frac{1}{2}}  y}(\cdot))\right\|_\lambda}\right)=\zeta\left(\frac{1}{2}, y\right).
\end{equation*}
Further, we claim that
\begin{equation*}
\lim _{\theta \rightarrow 1} G\left(\frac{\mathcal U(\tau_{-\frac{\tanh (\frac{R}{4 \theta})}{\tanh\frac{1}{2}}  y}(\cdot))}{\left\|\mathcal U(\tau_{-\frac{\tanh (\frac{R}{4 \theta})}{\tanh\frac{1}{2}}  y}(\cdot))\right\|_\lambda}\right)=y \;\;\;\text{for $R$ large.}
\end{equation*}
Observe that for any $z \in \mathbb{B}^{N}$ with $d(z,0)$ large, we have
\begin{align*}
&\int_{\mathbb{B}^{N}} \frac{x}{m(x)}\frac{|\mathcal U(x)|^{2^*}}{\left\|\mathcal U\right\|_\lambda^{2^*}}\mathrm{~d} V_{\mathbb{B}^{N}}(x)- \frac{z}{m(z) \left[\int_{\mathbb{B}^{N}}|\mathcal U|^{2^*}\right]^{\frac{2^*}{2}-1}}\\
&= \int_{\mathbb{B}^{N}} \frac{\tau_{z}(y)}{\mathcal U(\tau_{z}(y))}\frac{|\mathcal U(y)|^{2^*}}{\left[\int_{\mathbb{B}^{N}}|\mathcal U|^{2^*}\right]^{\frac{2^*}{2}}}\mathrm{~d} V_{\mathbb{B}^{N}}(y)- \int_{\mathbb{B}^{N}}\frac{z}{m(z)} \frac{|\mathcal U(y)|^{2^*}}{\left[\int_{\mathbb{B}^{N}}|\mathcal U|^{2^*}\right]^{\frac{2^*}{2}}}\mathrm{~d} V_{\mathbb{B}^{N}}(y)\\
&= \frac{1}{\left[\int_{\mathbb{B}^{N}}|\mathcal U|^{2^*}\right]^{\frac{2^*}{2}}}\int_{\mathbb{B}^{N}} \left[\frac{\tau_{z}(y)}{m(\tau_{z}(y))}- \frac{z}{m(z)}\right]|\mathcal U(y)|^{2^*}\mathrm{~d} V_{\mathbb{B}^{N}}(y)\\
& =\frac{\tanh \frac{1}{4}}{\left[\int_{\mathbb{B}^{N}}|\mathcal U|^{2^*}\right]^{\frac{2^*}{2}}}\int_{\mathbb{B}^{N}} \left[\frac{\tau_{z}(y) |z|- z |\tau_{z}(y)|}{|\tau_{z}(y)||z|}\right] |\mathcal U(y)|^{2^*}\mathrm{~d} V_{\mathbb{B}^{N}}(y)
= o(1) \;\;\;\text{for $d(z,0)$ large,}
\end{align*}
where for the last step, we have used the dominated convergence theorem for $|z| \rightarrow 1.$ 

\medskip

Setting $\mathcal U_{R,\theta}(x)= \mathcal U(\tau_{-\frac{\tanh (\frac{R}{4 \theta})}{\tanh\frac{1}{2}}  y}(x))$ for $y \in S^{N-1}_{\mathbb{B}^{N}} $ and $R$ large.
Thus using the above arguments, we obtain
\begin{align*}
\lim _{\theta \rightarrow 1} G\left(\frac{\mathcal U_{R,\theta}(\cdot)}{\left\|\mathcal U_{R,\theta}(\cdot)\right\|_\lambda}\right)=&\lim _{\theta \rightarrow 1} \tanh(\frac{1}{2}) \frac{\int_{\mathbb{B}^{N}} \frac{x}{m(x)}\frac{\left|\mathcal U_{R, \theta}(x)\right|^{2^*}}{\left\|\mathcal U\right\|_\lambda^{2^*}}\mathrm{~d} V_{\mathbb{B}^{N}}(x)}{\left|\int_{\mathbb{B}^{N}} \frac{x}{m(x)}\frac{\left|\mathcal U_{R, \theta}(x)\right|^{2^*}}{\left\|\mathcal U\right\|_\lambda^{2^*}}\mathrm{~d} V_{\mathbb{B}^{N}}(x)\right|}\\
&= \lim _{\theta \rightarrow 1} \tanh(\frac{1}{2})\frac{\left(\frac{\frac{\tanh (\frac{R}{4 \theta})}{\tanh\frac{1}{2}}y}{\left|\frac{\tanh (\frac{R}{4 \theta})}{\tanh\frac{1}{2}} y\right| \left[\int_{\mathbb{B}^{N}}|\mathcal U|^{2^*}\right]^{\frac{2^*}{2}-1}}\right)}{\left|\frac{\frac{\tanh (\frac{R}{4 \theta})}{\tanh\frac{1}{2}}y}{\left|\frac{\tanh (\frac{R}{4 \theta})}{\tanh\frac{1}{2}} y\right| \left[\int_{\mathbb{B}^{N}}|\mathcal U|^{2^*}\right]^{\frac{2^*}{2}-1}}\right|}=y.
\end{align*}
Hence the claim holds and, therefore, the proposition follows.
\end{proof}

We are now in a situation to establish our main results:
\begin{proposition}
	For sufficiently large $R \geq R_{1}$,
	\begin{equation*}
		\text { cat }\left(\left[E_{\la,a, f}<\bar{\mathcal I}_{\lambda,a,f}\left(\mathcal{V}_{a, f}\right)+\mathcal I_{\la,1,0}(\mathcal{U})-\varepsilon_{0}(R)\right]\right) \geq 2.
	\end{equation*}
	\label{propcat}
\end{proposition}
\begin{proof}
	Lemma \ref{lemcat}, together with Proposition \ref{propiden}, proves the proposition.
\end{proof}
\noi
The above-stated proposition leads us to the following multiplicity result.
\begin{theorem}
	Let the assumptions on potential $a(.)$ be as mentioned in Theorem \ref{Thm-A2}. Then there exists $d_{5}>0$ such that if $\|f\|_{\mathcal{H'}} \leq d_{5},\; f \geq 0,\; f \not\equiv 0$, then $E_{\la,a, f}(v)$ has at least two critical points in
	\begin{equation*}
		\left[E_{\la,a, f}<\bar{\mathcal I}_{\lambda,a,f}\left(\mathcal{V}_{a, f}\right)+\mathcal I_{\la,1,0}(\mathcal{U})\right].
	\end{equation*}
	\label{crit2,3}
\end{theorem}
\begin{proof}
From Corollary~\ref{cor1.8}, we know $(PS)_c$ is satisfied for $E_{\la,a,f}$ when $c<\bar{\mathcal I}_{\lambda,a,f}\left(\mathcal{V}_{a, f}\right)+\mathcal I_{\la,1,0}(\mathcal{U})$. Hence, the theorem follows from Proposition \ref{propcat}, and Proposition \ref{propcrit}.
\end{proof}
\medskip
{\bf {Proof of Theorem~\ref{Thm-A2} concluded:}}
\begin{proof}
We set the first positive solution as $u_1=\mathcal{V}_{a, f}$ which has been found in Proposition~\ref{Pro_Locmin} and from \eqref{1st sol energy}, we have 
$$\bar I_{\la,a,f}(\mathcal{V}_{a,f})<0.$$ 
From Theorem~\ref{crit2,3}, $E_{\la,a, f}$ has at least two critical points $v_2, v_3$ in \\
 $\left[E_{\la,a, f}<\bar{\mathcal I}_{\lambda,a,f}\left(\mathcal{V}_{a, f}\right)+\mathcal I_{\la,1,0}(\mathcal{U})\right].$ Using Proposition~\ref{Relcrit_I&E}(iii), $u_2=t_{a,f}(v_2)v_2$ and $u_3=t_{a,f}(v_3)v_3$  are the second and third positive solutions of \eqref{Paf}. Further, by Lemma~\ref{Lem_reIE}(iii), 
$0<E_{\la,a,f}(v_i)=\bar{\mathcal I}_{\lambda,a,f}(u_i), \quad i=1,2.$ Hence,
$$\bar{\mathcal I}_{\lambda,a,f}(u_1)<0<\bar{\mathcal I}_{\lambda,a,f}(u_i)<\bar{\mathcal I}_{\lambda,a,f}(u_1)+\mathcal I_{\la,1,0}(\mathcal{U}), \,i=2,3.$$
Hence, $u_1, u_2, u_3$ are distinct and \eqref{Paf} has at least three distinct solutions.
\end{proof}

 \section{Appendix}  \label{appen}
 In the appendix, we shall derive the necessary conditions in connection with Theorem~\ref{ps_decom}.
 \begin{lemma} \label{lemTn} Let us assume all the hypotheses and notations as in Theorem \ref{ps_decom}. Then
 	 \begin{enumerate}
 		\item $T^j_n \circ T^{-i}_n(0)\to \I $ for $i\ne j$.
 		\item $\Big| \log\big(\frac{\eps^k_{n} }{\eps^l_{n}}\big)\Big|+ \Big| \frac{y_n^k-y^l_n}{\eps^l_{n}}\Big|\to \I$ as $n\to \I$ for $k\ne l$.
 	\end{enumerate}
 \end{lemma}
 \begin{proof}

 	\textbf{Proof of (1):} We know that $\bar{u}_n\rightharpoonup 0$. If $\bar{u}_n\not\to 0$  strongly, then there exists $(T^1_n)$ such that $v^1_n:=\bar{u}_n\circ T^1_n\rightharpoonup v^1$. Since we are interested in the Hyperbolic bubble case, we assume that $v^1\not\equiv 0$. Hence both $T^1_n(0),T^{-1}_n(0)\to \I$ as $n\to \I$ where $T^{-1}_n= (T^{1}_n)^{-1}$. Finally, we found another PS- sequence $(z^2_n)_n$ for $I_{\lambda,a,0}$, where $z^2_n(x):=\bar{u}_n(x)- v^1(T_n^{-1}x)$(see Claim II). For notational purposes, define $z^1_n:= \bar{u}_{n}$. \\
 	
 	Note that $z^2_n\rightharpoonup 0$ and $z^2_n(T_n^1)\rightharpoonup 0$. Now again, if $z^2_n$ does not go to zero strongly, there exists $(T^2_n)$ such that $v^2_n:=z^2_n\circ T^2_n\rightharpoonup v^2$. Assume $v^2\not\equiv 0$ (if $v^2 \equiv 0,$ then we land in the Aubin Talenti profile). Thus both $T^2_n(0),T^{-2}_n(0)\to \I$ as $n\to \I$. Further, we can find another PS- sequence $(z^3_n)_n$ for $I_{\lambda,a,0}$, where $z^3_n(x):=z^2_n(x)- v^2(T_n^{-2}x)$(see Claim II). Moreover, observe that $z^3_n\rightharpoonup 0$ and $z^3_n(T_n^2)\rightharpoonup 0$. Also, if we have $T_n^{-2} T_n^1(0) \to \infty$, then $z^3_n(T_n^1)\rightharpoonup 0$.\\\\
 Now we use the method of induction as follows to get our required claim.
 Firstly, observe the following things from the above discussion:
 \begin{enumerate}
 \item $T^{-j}_n:= (T^{j}_n)^{-1},$ and $T^{-j}_n(0) \to \I \;\forall j$.
\item $z^j_n(x):=z^{j-1}_n(x)- v^{j-1}(T_n^{-(j-1)}x)$ for $j>1$ where $v^{j-1}=$ weak limit of $z^{j-1}_n\circ T^{j-1}_n$, and $z^1_n= \bar{u}_{n}$.
\item $z^j_n \rightharpoonup 0 \; \forall j.$
\item $v^{j-1} \not\equiv 0 $ and 
$z^j_n\circ T^{j-1}_n\rightharpoonup 0 \; \forall j>1$.
 \end{enumerate}
 For the following iteration scheme, we assume $i<j.$ \\
  \textbf{Induction step 1}: We prove our claim for $j=i+1$.\\
 For this, as we know from the above points that $z^j_n\circ T^j_n\rightharpoonup v^j \nequiv 0$ and $z^j_n\circ T^i_n\rightharpoonup 0$, then for any $\phi \in C_{c}^{\infty}\left(\mathbb{B}^{N}\right)$, we have 
 	\begin{align} \label{Orth}
		0 & \neq \left\langle v^j,\phi \right\rangle_{\lambda} \\ \notag 
  & =\lim _{n \rightarrow \infty}\left\langle z^j_n\circ T^j_n,\phi \right\rangle_{\lambda} \\ \notag
&= \lim _{n \rightarrow \infty}\left\langle z^j_n,\phi \circ T^{-j}_n \right\rangle_{\lambda}\\ \notag
& = \lim _{n \rightarrow \infty}\left\langle z^j_n \circ T^i_n,\phi \circ \left(T^{-j}_n T^i_n \right) \right\rangle_{\lambda}.
\end{align}
 Now if $T^{-j}_n \circ T^i_n(0) \nrightarrow \infty$, then there exists $R>0$ such that $Supp\left(\phi \circ \left(T^{-j}_n T^i_n \right)\right) \subseteq B_{R}(0)$. Then using Vitali's convergence Theorem and the fact that $z^j_n \circ T^i_n \rightharpoonup 0$ in the last step of the above calculations, we get a contradiction.\\\\
 \textbf{Induction assumption:} Let the following hold true for $j=i+(k-1) \; \forall\; (k-1)>0$:
 \begin{enumerate}
     \item $z^j_n\circ T^i_n\rightharpoonup 0$.
     \item $T^{-j}_n \circ T^i_n(0) \rightarrow \infty$.
 \end{enumerate}
 \textbf{Final step:} In this step, we prove that $T^{-j}_n \circ T^i_n(0) \rightarrow \infty$ for $j=i+k \; \forall \; k>0$. \\
 Firstly, observe that $(j-1)-i=k-1$, then consider
 \begin{equation*}
     \begin{aligned}
         z^j_n &=z^{j-1}_n- v^{j-1}(T_n^{-(j-1)}),\\
         z^j_n(T_n^i)&=z^{j-1}_n (T_n^i)- v^{j-1}(T_n^{-(j-1)}T_n^i).\\
     \end{aligned}
 \end{equation*}
 Then using the induction assumptions, it follows that $z^j_n\circ T^i_n\rightharpoonup 0$ for $j=i+k.$ Now the same steps as in \eqref{Orth}, we get the desired result.
 
 \medskip
 \noi
 		\textbf{Proof of (2):} The complete proof can be found in \cite{MR3216834}. For the sake of completeness, we sketch the proof here. \\
   Let $W$ and $v$ be any Schwartz class functions in $D^{1,2}(\mathbb{R}^N)$. 
 		\begin{equation}\label{dis_wl_0}
 			D_{\epsilon_n^k,y^k_n}W:=(\eps_n^k)^{\frac{2-N}{2}}W\(\f{x-y^k_n}{\eps_n^k}\)\rightharpoonup 0  \equiv   \langle D_{\epsilon_n^k,y^k_n}W ,v\rangle =\langle (\eps_n^k)^{\frac{2-N}{2}}W\(\f{x-y^k_n}{\eps_n^k}\),v\rangle\to 0 .
 		\end{equation}
   Consider
 		\begin{align*}
 			\langle (\eps_n^k)^{\frac{2-N}{2}}W\(\f{x-y^k_n}{\eps_n^k}\),v\rangle& = \int_{ \mathbb R^N} \nabla \left((\eps_n^k)^{\frac{2-N}{2}}W\(\f{x-y^k_n}{\eps_n^k}\)\right)\nabla v \mathrm{~d} x\\
 			&=-\int_{ \mathbb R^N}  (\eps_n^k)^{\frac{2-N}{2}}W\(\f{x-y^k_n}{\eps_n^k}\)  \De v \mathrm{~d} x\\
 			&=-\int_{ \mathbb R^N}  (\eps_n^k)^{\frac{2-N}{2}}W\(\f{x }{\eps_n^k}\)  \De v(x+y^k_n) \mathrm{~d} x.
 		\end{align*}
 		For any $\delta >0$, we split the above integral into two components as follows
 		\begin{align*}
 			\int_{ \mathbb R^N}  (\eps_n^k)^{\frac{2-N}{2}}W\(\f{x}{\eps_n^k}\)  \De v(x+y^k_n) \mathrm{~d} x&=\int_{ B(0,\delta)}  (\eps_n^k)^{\frac{2-N}{2}}W\(\f{x}{\eps_n^k}\)  \De v(x+y^k_n)\mathrm{~d} x \\
			&+\int_{  B(0,\delta)^c}  (\eps_n^k)^{\frac{2-N}{2}}W\(\f{x}{\eps_n^k}\)  \De v(x+y^k_n) \mathrm{~d} x,
   \end{align*}
 		\noi
 		  First, at infinity, we have
 		\begin{align*}
 			\int_{ B(0,\delta)^c}  (\eps_n^k)^{\frac{2-N}{2}}W\(\f{x}{\eps_n^k}\)  \De v(x+y^k_n)\mathrm{~d} x&\leq \norm{(\eps_n^k)^{\frac{2-N}{2}}W\(\f{x}{\eps_n^k}\)}_{L^\I(B(0,\delta)^c)}\int_{ \mathbb R^N}|\De v(x)| \mathrm{~d} x \\
 			&=\norm{(\eps_n^k)^{\frac{2-N}{2}}W\(\f{x-y^k_n}{\eps_n^k}\)}_{L^\I(B(y^k_n,\delta)^c)}\int_{ \mathbb R^N}|\De v(x)| \mathrm{~d} x\\
    &=o(1).
   \end{align*}
 		On the other hand,
 		\begin{align*}
 			&\int_{ B(0,\delta)}|  (\eps_n^k)^{\frac{2-N}{2}}W\(\f{x}{\eps_n^k}\)  \De v(x+y^k_n)| \mathrm{~d} x\\
    &\leq   \norm{ \De v(x) }_{L^\I(\mathbb R^N)}\int_{ B(0,\delta)}|  (\eps_n^k)^{\frac{2-N}{2}}W\(\f{x}{\eps_n^k}\) |\mathrm{~d} x\\
 			&= \norm{ \De v(x) }_{L^\I(\mathbb R^N)}\left(\int_{ B(0,\delta)}|  (\eps_n^k)^{\frac{2-N}{2}}W\(\f{x}{\eps_n^k}\) |^{2^*}\mathrm{~d} x\right)^\frac{1}{2^*}|B(0,\delta)|^{\frac{2^*-1}{2^*}}\\
 			&= \norm{ \De v(x) }_{L^\I(\mathbb R^N)}(\eps_n^k)^{\frac{2-N}{2}}\(\int_{ B(0,\delta)}|  W\(\f{x}{\eps_n^k}\) |^{2^*}\mathrm{~d} x\)^\frac{1}{2^*}\delta^{\frac{N+2}{2N}}\\
 			&	\leq \norm{ \De v(x) }_{L^\I(\mathbb R^N)} \norm{W}_{2^*}\delta^{\frac{N+2}{2N}}.
 		\end{align*}
   
   \noi
 		Finally, for any $\delta>0$, we have
 		\begin{align*}
 			&\limsup_{n\to\I}\int_{ \mathbb R^N}  (\eps_n^k)^{\frac{2-N}{2}}W\(\f{x}{\eps_n^k}\)  \De v(x+y^k_n)\mathrm{~d} x\\
    &=\int_{ B(0,\delta)}  (\eps_n^k)^{\frac{2-N}{2}}W\(\f{x}{\eps_n^k}\)  \De v(x+y^k_n)\mathrm{~d} x+\int_{  B(0,\delta)^c}  (\eps_n^k)^{\frac{2-N}{2}}W\(\f{x}{\eps_n^k}\)  \De v(x+y^k_n)\mathrm{~d} x\\
 			&\leq \norm{ |\De v(x)| }_{L^\I(\mathbb R^N)} \norm{W}_{2^*}\delta^{\frac{N+2}{2N}}.
 		\end{align*}
 		Thus $ \int_{ \mathbb R^N}  (\eps_n^k)^{\frac{2-N}{2}}W\(\f{x}{\eps_n^k}\)  \De v(x+y^k_n)\mathrm{~d} x=o(1)$.\\
 	\noi	
 		Hence 	$\langle (\eps_n^k)^{\frac{2-N}{2}}W\(\f{x-y^k_n}{\eps_n^k}\),v\rangle\to 0$. The result follows from the standard density arguments as the Schwartz class functions are dense in $ D^{1,2}(\mathbb R^N)$. \\

   \medskip

\noi   
The following statements are equivalent to $n\to\I$.
 		  
 \begin{equation}
 		\begin{aligned}
 		&(a)\,\,	\langle D_{\epsilon_n^k,y^k_n}W\, ,\,D_{\epsilon_n^l,y^l_n}V\rangle=\langle D_{\epsilon_n^k\epsilon_n^{-l}\,,\,\,\epsilon_n^{-l}(y^k_n-y^l_n)}W \,,\,V\rangle \to 0\\
 		&(b)\,\,\Big| \log\big(\frac{\eps^k_{n} }{\eps^l_{n}}\big)\Big|+ \Big| \frac{y_n^k-y^l_n}{\eps^l_{n}}\Big|\to \I \text{  as  } n\to \I \text{  for  } k\ne l.
 		\end{aligned}
 	\end{equation} 

Hence to prove $(2)$, it is enough to prove $(a)$. To show that, we use the properties of Dislocation space as in \cite{MR3216834}.\\
 Indeed, assuming in Theorem 8 of \cite{MR3216834}, $g_{n}^{(j)}= D_{y_{n}^j, \lambda_{n}^j} \in D$ and $D$ is defined as the space of Unitary operators as follows 
 \begin{equation*}
D:=\left\{D_{y, \lambda}: D_{y, \lambda} u(x):=\lambda^{\frac{2-N}{2}} u\left(\frac{x-y}{\lambda}\right)\;\; \forall u \in D^{1,2}(\mathbb{R}^N); \; y \in \mathbb{R}^N, \;\lambda>0\right\}.
\end{equation*}
To prove that $(D^{1,2}(\mathbb{R}^N), D )$ is a dislocation space, it is enough to show using proposition 3.1 of \cite{MR2294665} that 
$$D_{y_{n},\lambda_{n}} \in D, D_{y_{n},\lambda_{n}} \not\rightharpoonup 0 \implies D_{y_{n},\lambda_{n}} \text{has a strongly convergent subsequence}.$$ 
$D_{y_{n},\lambda_{n}} \rightharpoonup 0 $, i.e., there exist $v \in D^{1,2}(\mathbb{R}^N)$ such that $D_{y_{n},\lambda_{n}}v \not\rightharpoonup 0$. Then using lemma 3 of \cite{MR3216834}, we get $\lambda_n \rightarrow \lambda$ and $y_n \rightarrow y$ for some $\lambda >0, y \in \mathbb{R}^N.$ Then using equation (54) of \cite{MR3216834}, we conclude $D_{y_{n}, \lambda_{n}}\rightarrow D_{y, \lambda}$ .
Thus using Theorem 8 of \cite{MR3216834}, we get $D_{y_{n}^j, \lambda_{n}^j}^{-1} D_{y_{n}^k, \lambda_{n}^k} \rightharpoonup 0 \text{ for } k \neq j.$ Hence $(a)$ follows.
 \end{proof}

\medskip

\medskip

{\bf Acknowledgement:} The research of M.~Bhakta and A. K.~Sahoo are partially supported by the DST Swarnajaynti fellowship (SB/SJF/2021-22/09), D.~Ganguly is partially supported by the INSPIRE faculty fellowship (IFA17-MA98) and D.~Gupta is supported by the PMRF fellowship.

 \end{document}